\documentclass[12pt,leqno,amscd,amssymb]{amsart}
\usepackage[english]{babel}
\usepackage{amsfonts,amssymb,graphicx,amsmath,amsthm,fullpage,tikz,amscd,mathrsfs}\usepackage{enumitem,comment}

\usepackage[bookmarks=false]{hyperref}

\allowdisplaybreaks

\usepackage[all,cmtip]{xy}
\usepackage[dvips]{epsfig}
\usepackage{pinlabel}
\newcommand{\ig}[2]{\vcenter{\xy (0,0)*{\includegraphics[scale=#1]{fig/#2}} \endxy}}
\newcommand{\igv}[2]{\vcenter{\xy (0,0)*{\reflectbox{\includegraphics[scale=#1, angle=180]{fig/#2}}} \endxy}}
\newcommand{\igh}[2]{\vcenter{\xy (0,0)*{\reflectbox{\includegraphics[scale=#1]{fig/#2}}} \endxy}}
\newcommand{\ighv}[2]{\vcenter{\xy (0,0)*{\includegraphics[scale=#1, angle=180]{fig/#2}} \endxy}}

\newcommand{\bigon}[3] {{
\labellist
\tiny\hair 2pt
 \pinlabel {$#1$} [ ] at 26 5
 \pinlabel {$#2$} [ ] at 2 26
 \pinlabel {$#3$} [ ] at 37 26
\endlabellist
\centering
\ig{1}{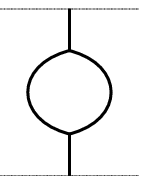}
}}
\newcommand{\squareone}[9] {{
\labellist
\tiny\hair 2pt
 \pinlabel {$#1$} [ ] at 3 6
 \pinlabel {$#2$} [ ] at 3 66
 \pinlabel {$#3$} [ ] at 37 6
 \pinlabel {$#4$} [ ] at 37 66
 \pinlabel {$#5$} [ ] at 20 15
 \pinlabel {$#6$} [ ] at 20 60
 \pinlabel {$#7$} [ ] at 3 39
 \pinlabel {$#8$} [ ] at 37 39
\endlabellist
\centering
\ig{#9}{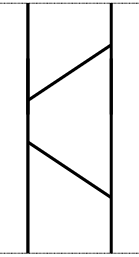}
}}
\newcommand{\squaretwo}[9] {{
\labellist
\tiny\hair 2pt
 \pinlabel {$#1$} [ ] at 3 6
 \pinlabel {$#2$} [ ] at 3 66
 \pinlabel {$#3$} [ ] at 37 6
 \pinlabel {$#4$} [ ] at 37 66
 \pinlabel {$#5$} [ ] at 20 15
 \pinlabel {$#6$} [ ] at 20 60
 \pinlabel {$#7$} [ ] at 3 39
 \pinlabel {$#8$} [ ] at 37 39
\endlabellist
\centering
\ig{#9}{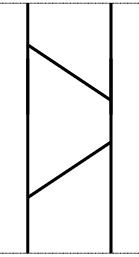}
}}
\newcommand{\mergesplit}[6] {{
\labellist
\tiny\hair 2pt
 \pinlabel {$#1$} [ ] at 4 12
 \pinlabel {$#3$} [ ] at 34 12
 \pinlabel {$#2$} [ ] at 4 60
 \pinlabel {$#4$} [ ] at 34 60
 \pinlabel {$#5$} [ ] at 30 38
\endlabellist
\centering
\ig{#6}{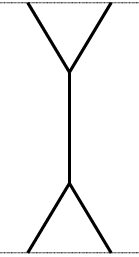}
}}
\newcommand{\rungNE}[6] {{
\labellist
\tiny\hair 2pt
 \pinlabel {$#1$} [ ] at 7 5
 \pinlabel {$#2$} [ ] at 32 5
 \pinlabel {$#3$} [ ] at 7 20
 \pinlabel {$#4$} [ ] at 32 20
 \pinlabel {$#5$} [ ] at 19 18
\endlabellist
\centering
\ig{#6}{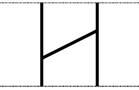}
}}
\newcommand{\rungNW}[6] {{
\labellist
\tiny\hair 2pt
 \pinlabel {$#1$} [ ] at 7 5
 \pinlabel {$#2$} [ ] at 32 5
 \pinlabel {$#3$} [ ] at 7 20
 \pinlabel {$#4$} [ ] at 32 20
 \pinlabel {$#5$} [ ] at 19 18
\endlabellist
\centering
\ig{#6}{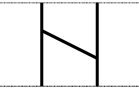}
}}

\newcommand{\longsquiggly}{\xymatrix{{}\ar@{~>}[r]&{}}}
\DeclareMathOperator{\Kar}{Kar}
\DeclareMathOperator{\core}{core}
\newcommand{\id}{\operatorname{id}}
\newcommand{\Hom}{\text{\rm Hom}}

\newcommand{\End}{\operatorname{End}}

\newcommand{\Z}{\mathbb Z}
\newcommand{\C}{\mathbb C}
\newcommand{\inv}{^{-1}}
\newcommand{\R}{\mathbb R}

\newcommand{\snum}{\delta}

\newtheorem{thm}{Theorem}[section] 
\newtheorem{prop}[thm]{Proposition}
\newtheorem{lem}[thm]{Lemma}
\newtheorem{cor}[thm]{Corollary}

\theoremstyle{definition}
\newtheorem{defn}[thm]{Definition}
\newtheorem{notation}[thm]{Notation}

\newtheorem{ex}[thm]{Example}

\theoremstyle{remark}
\newtheorem{rem}[thm]{Remark}

\numberwithin{equation}{section}

\newcommand{\mi}{\underline}
\newcommand{\ma}{\overline}

\newcommand{\expr}{\leftrightharpoons} 
\newcommand{\JIcosets}{W_J\backslash W/W_I}

\newcommand{\SC}{\mathcal{SC}}
\newcommand{\HC}{\mathcal{H}}
\newcommand{\mt}{\emptyset}
\DeclareMathOperator{\Fl}{Fl}
\newcommand{\uw}{\underline{w}}
\newcommand{\ox}{\overline{x}}
\newcommand{\oy}{\overline{y}}
\newcommand{\oz}{\overline{z}}
\DeclareMathOperator{\BS}{BS}
\DeclareMathOperator{\Cox}{Cox}
\DeclareMathOperator{\leftdes}{LD}
\DeclareMathOperator{\rightdes}{RD}

\def\hat{\widehat}

\DeclareMathOperator{\Webs}{Webs}
\newcommand{\SCWebAlg}{\SC(S_\star)}
\newcommand{\SCWebDiag}{\SC_{\Webs}} 
\newcommand{\un}{\underline{n}}
\newcommand{\um}{\underline{m}}
\newcommand{\hg}{\mathfrak{h}}
\DeclareMathOperator{\cube}{cube}

\begin{document}
\title[]{A singular Coxeter presentation}

\author[]{Ben Elias}
\address{University of Oregon.}
\email{belias@uoregon.edu}

\author[]{Hankyung Ko}
\address{Uppsala University}
\email{hankyung.ko@math.uu.se}

\maketitle

\begin{abstract}
We enlarge a Coxeter group into a category, with one object for each finite parabolic subgroup, encoding the combinatorics of double cosets. This category, the singular Coxeter monoid, is connected to the geometry of partial flag varieties. Our main result is a presentation of this category by generators and relations. We also provide a new description of reduced expressions for double cosets. We describe all the braid relations between such reduced expressions, and prove an analogue of Matsumoto's theorem. This gives a proper development of ideas first introduced by Geordie Williamson. In type $A$ we also equip the singular Coxeter monoid with a diagrammatic presentation using webs.
\end{abstract}


\section{Introduction}

Weyl groups are ubiquitous in representation theory and geometry, and their Coxeter presentations are extremely useful for understanding the length function, the Bruhat order, and
other features arising in the geometry of the flag variety. In this paper, we enlarge a Coxeter group into a category, with one object for each finite parabolic subgroup, encoding the combinatorics of double cosets. This category, the \emph{singular}\footnote{For more on what the confusing word ``singular'' means, see Remark \ref{rmk:singular}. For now, singular should connote ``relating to double cosets for parabolic subgroups in $W$, rather than to elements of $W$.'' It also connotes a change in categorical type, similar to the change from a group to a groupoid or from a monoid to a category.} \emph{Coxeter monoid}, is connected to the geometry of partial flag varieties. Our main result is a presentation of this category by generators and relations. We also provide a new description of \emph{reduced expressions for a double coset}. We describe all the braid relations between such reduced expressions, and prove an analogue of Matsumoto's theorem. This gives a proper development of ideas first introduced by Geordie Williamson in his PhD thesis \cite[\S 1.3]{GWthesis}.

This paper is primarily combinatorial, so we begin this introduction with the combinatorics. Readers who wish to be motivated first by geometry are welcome to skip ahead to
\S\ref{subsec-intro-geometry}, and return to the rest of the introduction afterwards.

\begin{notation} Let $(W,S)$ be a Coxeter system. For $I\subset S$, we denote by $W_I$ the subgroup of $W$ generated by $I$. Then $(W_I,I)$ is a Coxeter system, and $W_I$ is
called a \emph{(standard) parabolic subgroup} of $(W,S)$. When $W_I$ is finite, we say that $I$ is \emph{finitary}, and we write $w_I$ for the longest element of $W_I$. The identity element of $W$ is denoted by $e$.
\end{notation}

\begin{notation} Let $I$ and $J$ be finitary. It is well-known (see Lemma \ref{doublecosetlem}) that each double coset $p \in \JIcosets$ has a unique maximal element and a unique
minimal element in the Bruhat order. Write $\ma{p}$ for the maximal element and $\mi{p}$ for the minimal element. We refer to double cosets in $\JIcosets$ as $(J,I)$-cosets.
\end{notation}

\begin{notation} To avoid needing to write set braces constantly for subsets of $S$, we use shorthand. We write $s$ for $\{s\}$, $st$ for $\{s,t\}$, $Is$ for $I \sqcup \{s\}$, and $I \setminus s$ for $I \setminus \{s\}$. Whether $st$ represents an element of $W$ or a subset of $S$ should be clear from context. \end{notation}

\subsection{Reduced expressions for double cosets: definition and examples}

So we can jump into the fun, let us make a temporary definition that will be refined later.

\begin{defn} \label{def:redexpintro}
A \emph{(singular, multistep) expression} is the data of a sequence of finitary subsets of $S$ of the form
\begin{subequations} \label{introreddef}
\begin{equation} \label{introredsequence} [[I_0 \subset K_1 \supset I_1 \subset \ldots \subset K_m \supset I_m]]. \end{equation}
It is \emph{reduced} if
\begin{equation} \label{hkdef} \begin{split}
    \ell(w_{K_1} w_{I_1}\inv w_{K_2} \cdots w_{I_{m-1}}\inv w_{K_m}) &= \ell(w_{K_1}) - \ell(w_{I_1}) + \ell(w_{K_2}) - \ldots - \ell(w_{I_{m-1}}) + \ell(w_{K_m})\\
\end{split} \end{equation}
\end{subequations}
holds, where $\ell(w)$ denotes the Coxeter length of $w\in W$. 
In this case it expresses the $(I_0,I_m)$-coset 
\begin{equation} \label{rexexpresses} p= W_{I_0}w_{K_1} w_{I_1}\inv w_{K_2} \cdots w_{I_{m-1}}\inv w_{K_m} W_{I_m}.\end{equation}
By convention, if we write an expression as $[[K_1 \supset I_1 \subset \ldots]]$, this means that $I_0 = K_1$, and similarly $[[\ldots \subset K_m]]$ means that $I_m = K_m$. This does not affect \eqref{hkdef}.
\end{defn}

Note that if $[[I_0 \subset K_1 \supset I_1 \subset \ldots \subset K_m \supset I_m]]$ is a reduced expression for $p$, then the maximal element $\ma{p}$ satisfies
\begin{equation} \label{mapintroreddef} \ma{p} = w_{K_1} w_{I_1}\inv w_{K_2} \cdots w_{I_{m-1}}\inv w_{K_m}. \end{equation}

\begin{ex}\label{regular} 
For an ordinary reduced expression
$s_1 s_2 \cdots s_m$ in $(W,S)$ we assign a singular reduced expression, which bounces between finitary subsets of size zero and one: \begin{equation} \label{ordtosing} s_1 s_2 \cdots s_m
\longsquiggly [[\mt, s_1, \mt, s_2, \mt, \ldots, \mt, s_m, \mt]]. \end{equation} 
In this case, we have $w_{I_i}=e$ so that the condition \eqref{hkdef} becomes
\[\ell(s_1s_2\cdots s_m) = \ell(s_1)+\ell(s_2) +\cdots \ell(s_m).\]
This agrees with the definition of a reduced expression in $(W,S)$, making the correspondence \eqref{ordtosing} a bijection. 
The ordinary expression expresses the element $w = s_1 \cdots s_m$, while the singular expression expresses the $(\mt, \mt)$-coset $\{w\}$.
\end{ex}

\begin{ex} The double coset $\{w\}$ may have singular reduced expressions which do not come from \eqref{ordtosing}. For example, if $m_{st} = 3$ then $\{sts\}$ also has the
reduced expression $[[\mt, st, \mt]]$. \end{ex}

\begin{ex} When $S$ is finitary, there is a unique $(\mt,S)$-coset, containing all of $W$. Then $[[\mt,S]]$ is a reduced expression for this coset. However, $[[\mt,s,\mt,S]]$ is
not a reduced expression for any $s \in S$. \end{ex}

\begin{ex} Suppose $W$ has type $A_3$ with $S = \{s,t,u\}$ and $m_{su} = 2$. Then $[[st,s,su,s,st]]$ is a reduced expression for the $(st,st)$-coset containing the longest element
$w_S$. This is because \begin{equation} w_S = stsuts = (sts) \cdot (s\inv) \cdot (su) \cdot (s\inv) \cdot (sts), \end{equation} and the lengths add appropriately. Another reduced
expression for this coset is $[[st,stu,st]]$. Meanwhile, we claim that $[[st,stu,tu]]$ is the unique reduced expression for the $(st,tu)$-coset containing $w_S$. \end{ex}

In this paper we introduce the \emph{(singular) braid relations}, and prove in Theorem \ref{thm:matsumoto} that any two singular reduced expressions for a double coset are related
by the singular braid relations. However, while multistep expressions are the most convenient for defining reduced expressions, it makes the application of relations slightly annoying.
To transform $[[\mt,st,\mt]]$ into $[[\mt,s,\mt,t,\mt,s,\mt]]$, we need to be able to go from $\mt$ to $st$ in two steps, first adding $s$ and then adding $t$. For a variety of reasons, we prefer single-step expressions in this paper.

\begin{defn} \label{def:exprintro} A \emph{(singular, single-step) expression} $I_{\bullet} = [I_0, I_1, \ldots, I_d]$ is a sequence of finitary subsets of $S$, where for each $1
\le i \le d$ we have either $I_i = I_{i-1} \sqcup s$ or $I_i = I_{i-1} \setminus s$ for some $s \in S$. The number $d$ is called the \emph{width} of the expression\footnote{The
\emph{length} of an expression will be defined in \S\ref{ss:length},  and better reflects the length function in $W$.}. A single-step expression is \emph{reduced} if the corresponding
multi-step expression is a reduced expression.  \end{defn}

The singular reduced expressions from \eqref{ordtosing} are already single-step expressions. The multi-step expression $[[\mt,st,\mt]]$ corresponds to four distinct single-step
expressions, \begin{equation} \label{otherfour} [\mt,s,st,s,\mt], \quad [\mt,t,st,s,\mt], \quad [\mt, s, st, t, \mt], \quad [\mt,t,st,t,\mt].\end{equation} From now on, all expressions are single-step unless stated otherwise.

\begin{rem} An equivalent criterion for reducedness for single-step expressions is stated in Theorem \ref{thm:betterred}. That these two versions are equivalent is
proven in Lemma \ref{lem:sameredasintro}.
 \end{rem}

\begin{ex} Let $S = \{s,t\}$ with $m_{st} = 3$, so that $W$ is the symmetric group $S_3$. Let us contrast ordinary reduced expressions with singular reduced expressions. See \S\ref{subsec:coxeterintro} and \S\ref{sec:coxcomplex} for more details on the constructions below.
\begin{equation} \Cox(S_3) := \qquad {
\labellist
\small\hair 2pt
 \pinlabel {$e$} [ ] at 35 5
 \pinlabel {$s$} [ ] at 5 30
 \pinlabel {$t$} [ ] at 70 30
 \pinlabel {$st$} [ ] at 5 75
 \pinlabel {$ts$} [ ] at 70 75
 \pinlabel {$sts$} [ ] at 35 100
\endlabellist
\centering
\ig{1}{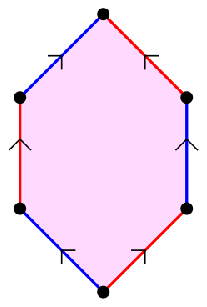}
} \qquad \Cox_{\mt}(S_3) := \qquad {
\labellist
\tiny\hair 2pt
 \pinlabel {$\{e\}$} [ ] at 35 6
 \pinlabel {$\{s\}$} [ ] at 5 30
 \pinlabel {$\{t\}$} [ ] at 70 30
 \pinlabel {$\{st\}$} [ ] at 5 75
 \pinlabel {$\{ts\}$} [ ] at 70 75
 \pinlabel {$\{sts\}$} [ ] at 35 100
 \pinlabel {$\{e,s\}$} [ ] at 19 14
 \pinlabel {$\{e,t\}$} [ ] at 55 14
 \pinlabel {$\{s,st\}$} [ ] at -2 55
 \pinlabel {$\{t,ts\}$} [ ] at 75 55
 \pinlabel {$\{st,sts\}$} [ ] at 18 91
 \pinlabel {$\{ts,sts\}$} [ ] at 57 91
\endlabellist
\centering
\ig{1.6}{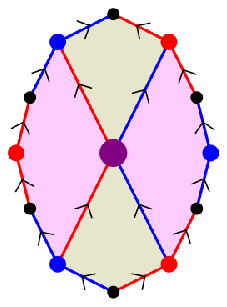}
} \end{equation}
The picture on the left is the (completed) dual Coxeter complex. Vertices correspond to elements of $W$. Oriented paths starting at the bottom vertex $e$ are reduced expressions. Any two reduced expressions are related by the braid relation, seen here as a $2$-cell.

The picture on the right is the \emph{singular (dual) Coxeter complex} where the left parabolic subgroup is trivial. Vertices correspond to $(\mt,I)$-cosets for various finitary $I$, and the color of the vertex represents
the choice of $I$. Oriented paths starting at the bottom vertex $\{e\}$ are reduced expressions. Any two reduced expressions are related by the braid relations,
seen here as $2$-cells. Altogether $\{sts\}$ has six reduced expressions: the two corresponding to $sts$ and $tst$ as in \eqref{ordtosing}, and the four in
\eqref{otherfour}.

To spell this out in symbols rather than pictures, let us describe the singular braid relations which transform $sts$ into $tst$, or more precisely,
\begin{subequations} \label{ex:braidintro}
\begin{equation} \label{introbraidstart} [ \mt, s, \mt, t, \mt, s, \mt ] \expr [\mt, t, \mt, s, \mt, t, \mt]. \end{equation}
Thoughout this paper we use $\expr$ to represent the equivalence of expressions. We start with the most interesting relation, our true analogue of the braid relation,
\begin{equation} \label{introbraidreln} [s, \mt, t, \mt, s] \expr [s, st, s], \end{equation}
which is a relation between two reduced expressions for the $(s,s)$-coset containing $sts$. Applying \eqref{introbraidreln} within the left-hand side of \eqref{introbraidstart} we obtain
\[ [ \mt, s, \mt, t, \mt, s, \mt ] \expr [ \mt, s, st, s, \mt]. \]
Now we can apply two relatively simple relations, the \emph{up-up} and \emph{down-down relations}, namely
\begin{equation} \label{introeasyreln} [\mt, s, st] \expr [\mt, t, st], \qquad [st, s, \mt] \expr [st, t, \mt]. \end{equation}
to obtain
\[ [ \mt, s, st, s, \mt] \expr [ \mt, t, st, t, \mt]. \]
We conclude with the analogue of the braid relation again, this time in the form
\begin{equation} [t, \mt, s, \mt, t] \expr [t, st, t], \end{equation}
to deduce \eqref{introbraidstart}.
\end{subequations}
\end{ex}

As seen in the previous example, singular expressions and their relations are more ``local'' or ``zoomed in'' than ordinary expressions. Each generator of $W$ (multiplication by $s
\in S$) is split into two smaller pieces (via $[\mt,s,\mt]$, a composition of $[\mt,s]$ with $[s,\mt]$). The braid relation $sts = tst$ follows from four smaller relations.

\begin{rem} \label{rmk:DHP1}  Fix $I \subset S$ finitary. The theory of reduced expressions for double cosets gives rise to a partial order on the set of all $(I,J)$-cosets as $J$ varies, analogous to the weak Bruhat order. When $I = \mt$, the Hasse diagram of this partial order is the 1-skeleton of $\Cox_{\mt}$ above. After writing this paper we discovered (thanks to a tip from Nathan Reading) that this partial order (when $I = \mt$) was studied from a combinatorial perspective by Dermenjian-Hohlweg-Pilaud in \cite{DHP}, where it is called the \emph{facial weak order}. They prove that the facial weak order is a lattice, and explicitly describe the meet and join \cite[Theorem 3.19]{DHP}. They investigate other properties of this lattice of interest to combinatorialists (such as the relationship with Cambrian lattices \cite{Reading}). Even prettier pictures of the singular Coxeter complex $\Cox_{\mt}$ can be found in \cite[p.7-8]{DHP}. Also see Remark \ref{rmk:DHP2} below. \end{rem}

\begin{rem} Singular (reduced) expressions of a very special form, like \eqref{ordtosing} but concatenated with $[[\mt,I]]$, are studied by Contou-Carrere and by Gaussent (see \cite{Gaussent}), where they are also related to the geometry of Bott-Samelson resolutions. \end{rem}

\subsection{Reduced expressions for double cosets: theory}

All the braid relations we need will be generalizations of either \eqref{introbraidreln} or \eqref{introeasyreln}. The up-up and down-down relations are easy to generalize:
\begin{equation} \label{introassoc} [I, Is, Ist] \expr [I, It, Ist], \qquad [Ist, Is, I] \expr [Ist, It, I] \end{equation}
whenever $Ist$ is finitary. We call the generalizations of \eqref{introbraidreln} the \emph{switchback relations}. They are described explicitly
in \S\ref{ss:specialswitch} and \S\ref{ssstypes}, and are governed by relatively simple and interesting combinatorics. There is one switchback relation for each finite Coxeter system $(W,S)$ and each pair $s \ne t \in
S$ (though one can reduce to the case where $W$ is irreducible).

\begin{ex} For a finite dihedral group, the regular braid relation is implied by four singular braid relations in much the same way as \eqref{ex:braidintro}. For example, when
$m_{st} = 4$ the switchback relation is \begin{equation} [s, \mt, t, \mt, s, \mt, t] \expr [s, st, t]. \end{equation} \end{ex}

\begin{ex} In type $E_8$, with simple reflections numbered as in \S\ref{sss:E8}, one switchback relation will be
\begin{equation} \label{E8rotation} [\hat{3}, \hat{32}, \hat{2}, \hat{27}, \hat{7}, \hat{72}, \hat{2}, \hat{23}, \hat{3}, \hat{38}, \hat{8}] \expr [\hat{3}, S, \hat{8}]. \end{equation}
Here $\hat{a}$ is shorthand for $S \setminus s_a$ and $\hat{ab}$ is shorthand for $S \setminus \{s_a, s_b\}$. \end{ex}

\begin{rem} If vertical height represents the size of a subset of $S$, then the switchback relation states that the path going over the top of the mountain agrees with the zigzagging path below\footnote{There were many other good suggestions for the name of this relation: the sawtooth or serriform relation, the waffle cone relation, the hankerchief relation, the iceberg relation, etcetera. The Loch Lomond relation was a close contenter, after ``you'll take the high road and I'll take the low road," the main issue being that, for us, the two roads reach Scotland at precisely the same time.}. \[ \ig{1}{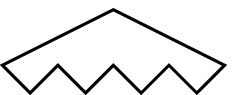} \] 
\end{rem}

\begin{rem} The sequence of simple reflections $327238$ appearing in \eqref{E8rotation} is an example of a \emph{rotation sequence}, which we define in Definition \ref{def:useq}. Since this paper first appeared, we learned that rotation sequences and related concepts seem to appear in work in progress of Wemyss and Iyama (see \cite[\S 3.2 and references therein]{WemyssSurvey}, where they study singularities using mutations and wall-crossings. \end{rem}

\begin{rem} It is a wonderful feature of ordinary Coxeter theory that all relations in Coxeter groups come from finite rank $2$ parabolic subgroups. This is in stark contrast to
the singular Coxeter monoid, which has switchback relations in arbitrary rank! However, even the switchback relation is planar in some sense, see \S\ref{moregeometry2}. Also, the
generating relations in the singular Coxeter monoid satisfy the property that they involve expressions $[I_0, I_1, \ldots, I_d]$ where the size of $I_i$ lies between $k$ and $k+2$
for some $k$. \end{rem}

In Theorem \ref{thm:matsumoto}, we prove that any two reduced expressions are related by the braid relations, which we call the \emph{singular Matsumoto theorem}. Matsumoto's
theorem lies tacitly beneath so much of Coxeter combinatorics that it is hard to point to a particular ``application'' of the theorem, rather than to areas of study or
general techniques (e.g. positive lifts to braid groups, reduced expression graphs, etcetera) which it makes possible.

\begin{rem} \label{rmk:Zam} The \emph{reduced expression graph} of an element $w \in W$ is the graph whose vertices are reduced expressions, and whose edges are braid relations.
Matsumoto's theorem states that this graph is connected. The cycles in such a graph are classified; the most interesting ones are the Zamolodchikov cycles associated to the longest
elements of rank 3 Coxeter groups. See \cite{EWFenn} for more details. One can also define reduced expression graphs for any double coset $p$. Classifying the cycles in this graph
is a very interesting open problem, related to finding the $3$-cells in the singular Coxeter complex. \end{rem}

\begin{rem} Matsumoto's theorem is used in a crucial way in the diagrammatic description of the Hecke category and its basis, the double leaves basis \cite{LibLL, EWGr4sb}. The
singular Matsumoto theorem was a crucial missing piece in a construction of singular double leaves for the singular Hecke category (see \S\ref{subsec-intro-geometry}),
which will appear in forthcoming work with Nicolas Libedinsky. \end{rem}


The original definition of a (reduced) singular expression is due to Williamson \cite[Definition 1.3.2]{GWthesis}, and is reformulated as our Definition \ref{def:reduced}. He
proves \cite[Proposition 1.3.4]{GWthesis} that every double coset has a reduced expression. This original definition has a very different flavor, is somewhat technical, and is not
obviously equivalent to the definition above. Many basic and essential properties are obfuscated (and were not previously known). For example, it is not clear that $[I_0, \ldots,
I_d]$ being reduced implies that the contiguous subword $[I_k, \ldots, I_\ell]$ is reduced for $0 \le k < \ell \le d$. This paper provides four equivalent criteria for an expression to be
reduced (Definition \ref{def:redexpintro}, Theorem \ref{thm:betterred}, Proposition \ref{p.rexlength}, Proposition~\ref{thm:inandout}) which we have found to be more useful in practice. We give more
details on the original definition in \S\ref{subsec-paths} and following.

In \S\ref{sec-newdesc} and \S\ref{sec-lotsoprops} we prove many properties of singular reduced expressions, some of which are in analogy with properties of ordinary reduced
expressions. In \S\ref{ss:locality} we prove that reduced expressions are closed under taking contiguous subwords. In \S\ref{ss:length} we introduce a length function on
expressions and on double cosets. We prove that reduced expressions minimize the length function amongst all expressions (see \S\ref{subsec-intro-singcox}) for a given double
coset. In \S\ref{ss:extension} we prove that any reduced expression can be extended to a reduced expression for a coset containing the longest element (in a finite Coxeter group).
In \S\ref{subsec:products} we prove basic facts about products of Coxeter groups.

We also prove several results which have a new flavor. In \S\ref{subsec:getbig} and \S\ref{subsec:getsmall} we examine the questions: for a given $(J,I)$-coset $p$, for
which simple reflections $s \notin J$ admits a reduced expression for $p$ starting with $[J,Js,\ldots]$? How big can $K_1$ be in a multistep reduced expression $[[J \subset K_1
\supset \ldots]]$ for $p$? For which simple reflections $s \in J$ can a reduced expression start with $[J,J\setminus s, \ldots]$? How small can $I_1$ be in $[[J \supset I_1 \subset
\ldots ]]$? We give precise answers to these questions, proving the existence of reduced expressions which start by getting as big as possible, or as small as possible. These theorems are surprisingly useful.

\subsection{The singular Coxeter monoid} \label{subsec-intro-singcox}

In ordinary Coxeter theory, expressions represent products of generating elements inside a group $W$. The singular Coxeter monoid is a category where singular expressions represent
products of generating morphisms. This allows us to discuss expressions (not necessarily reduced) for a double coset, and relations between these. The singular Coxeter monoid is a singular version, not of the group $W$, but of the Coxeter monoid $(W,*,S)$.

\begin{defn}
Let $(W,S)$ be a Coxeter system. The \emph{Coxeter monoid} $(W,*,S)$ (sometimes called the \emph{Coxeter star monoid} or the \emph{$0$-Hecke monoid}\footnote{The linearization of $(W,*,S)$ is the \emph{$0$-Hecke algebra}. We do not linearize $(W,*,S)$ in this paper.}) has a presentation with generators $S$, and the following relations.
\begin{subequations}
\begin{equation} \text{ The $*$-quadratic relation:} \quad s*s = s \text{ for each } s \in S. \end{equation}
\begin{equation} \text{ The braid relation:} \quad \underbrace{s * t * \cdots}_{m} = \underbrace{t * s * \cdots}_{m}, \text{ for each } s, t \in S \text{ with } m = m_{st} < \infty. \end{equation} \end{subequations}
\end{defn}

One can prove that the elements of $(W,*,S)$ are in natural bijection with elements of $W$, and in what follows we do not distinguish the former from the latter. Thus an element
$w\in W$ can be written as $w=s*t*\cdots *u$ for some $s,t,\cdots, u\in S$. The string $[s,t,\cdots ,u]$ is called a \emph{$*$-expression} of $w$. While the set of $*$-expressions
for $w$ does not agree with the set of ordinary expressions for $w$, the reduced expressions (i.e., expressions of shortest length) in both contexts agree.

\begin{defn} \label{defn:SC} The \emph{singular Coxeter monoid} $\SC=\SC(W,S)$ is the category defined as follows.
\begin{subequations}
\begin{itemize}
    \item $\operatorname{Ob}(\SC)$ consists of the finitary subsets $I\subset S$.
    \item For $I,J \in \operatorname{Ob}(\SC)$, $\Hom_{\SC}(I,J)$ is the set of $(J,I)$-cosets, i.e.  \begin{equation} \Hom_\SC(I,J) := W_J\backslash W /W_I. \end{equation}
    \item For $I,J,K\in \operatorname{Ob}(\SC)$ the composition `$*$' is given by
    \begin{equation} \label{SCcompformula} W_K x W_J * W_J y W_I := W_K (\ox*\oy) W_I,\end{equation}
    where $\ox$ and $\oy$ are the longest elements in $W_KxW_J$ and $W_JyW_I$ respectively.
\end{itemize}
\end{subequations}
\end{defn}

In fact, the element $\ox * \oy$ from \eqref{SCcompformula} is the maximal element in its double coset, see Lemma \ref{lem:maxofcoset}. Thus one can think about the morphisms in
$\SC$ as being the maximal elements of double cosets, with composition given by $*$. This makes it easy to verify that $\SC$ is a well-defined category (details can be found in \S\ref{subsec-cosets}). The identity element of the object $I$ is $W_I$ itself, viewed as an $(I,I)$-coset. Note that $\End_{\SC}(\mt) \cong (W,*,S)$ as monoids.

\begin{rem} By sending a double coset to its maximal element, one can construct a faithful functor from $\SC$ to $(W,*,S)$ (the latter being viewed as a category with one object).
In fact, $\SC$ is obtained from $(W,*,S)$ as the idempotent completion or Karoubi envelope, see Theorem \ref{thm:Karoubi}. \end{rem}

A given subset of $W$ can appear as a double coset in multiple ways, for different choices of $I$ and $J$. In this paper, whenever we speak about a double coset $p$, it is implicit
that $p$ carries with it the data of a pair $(J,I)$ such that $p$ is a $(J,I)$-coset, just as a morphism in a category implicitly carries the data of its source and target.

\begin{ex} \label{ex:ambiguous} If $I$ is finitary and $s \in I$ then $t = w_I s w_I$ is another simple reflection (possibly equal to $s$) in $I$. Then $\{s w_I, w_I\} = \{w_I t,
w_I\}$ is the underlying set of an $(s,\mt)$-coset $p$ and an $(\mt,t)$-coset $q$ and an $(s,t)$-coset $r$. We treat $p$, $q$, and $r$ as distinct double cosets. \end{ex}

\begin{rem} Double cosets and their enumeration were studied in \cite{BKPST}, though (unlike us) they did not wish to distinguish between a subset $p \subset W$ viewed as a
$(J,I)$-coset, or viewed as a $(J', I')$-coset, for different parabolic subgroups. \end{rem}

\subsection{Presenting the singular Coxeter monoid}

\begin{notation} Let $I, J \subset S$ be finitary. The \emph{minimal} $(J,I)$-coset is the double coset containing the identity element $e$. If $S$ is finitary, the \emph{maximal} $(J,I)$-coset is the double coset containing the longest element $w_S$. \end{notation}

The current paper gives a presentation of $\SC$ by generators and relations, see Theorem \ref{thm:presentation}. There is one pair of generators for each $s \notin I$ with $Is$
finitary. Our generators are the morphisms \[ Is \to I, \qquad I \to Is, \] given by the minimal $(I,Is)$- and $(Is,I)$-cosets, respectively. Both double cosets have the underlying
set $W_{Is}$. Their maximal element is $w_{Is}$, and their minimal element is $e$.

\begin{defn}
To an expression $I_{\bullet} = [I_0, I_1, \ldots, I_d]$ we can associate a composition 
\[ I_d \to \cdots \to I_1 \to I_0 \]
of generating morphisms in $\SC$. If $p$ is the $(I_0,I_d)$-coset given by this composition, we say that $[I_0, I_1, \ldots, I_d]$ \emph{expresses} $p$, and we write
\[ [I_0, I_1, \ldots, I_d] \expr p. \]
We also write $I_{\bullet} \expr I'_{\bullet}$ when they express the same double coset.
By convention, a width zero expression $[I]$ expresses the identity element of the object $I$, i.e. the minimal $(I,I)$-coset.
\end{defn}

If $[I_0, I_1, \ldots, I_d] \expr p$ and $\ma{p}$ is the maximal element of $p$, then by \eqref{SCcompformula} we can describe $\ma{p}$ as a star product of $d$ elements. However, using the fact that
\begin{equation} w_I * w_J = w_J = w_J * w_I \qquad \text{ when } I \subset J, \end{equation}
it is not hard to prove the simpler formula
\begin{equation} \label{mapstarintro} \ma{p} = w_{I_0} * w_{I_1} * \cdots * w_{I_d}. \end{equation}
One can think of $p$ as the unique $(I_0, I_d)$-coset containing the element in \eqref{mapstarintro}.

Any $*$-expression $[s,t,\cdots ,u]$ has a lift to $\SC$ as a singular expression $[\mt, s, \mt, t, \mt, \cdots, \mt, u, \mt]$. It helps to spell this out in a small example.

\begin{ex} \label{stuintro} Consider the ordinary expression $[s,t,u]$ with $s, t, u \in S$ distinct, and its corresponding singular expression $[\mt, s, \mt, t, \mt, u, \mt]$. An
ordinary expression gives rise to a sequence of elements via composition one index at a time, in this example giving $e$, $s$, $st$, $stu$. Similarly, a singular expression gives
rise to a sequence of double cosets\footnote{In Definition \ref{good subex}, we call this the \emph{forward path} of the expression.} (whose maximal element is obtained by
composing \eqref{mapstarintro} one index at a time). In this example we get the sequence \begin{equation} \label{introforwardexample} \{1\} \subset \{1, s\} \supset \{s\} \subset
\{s, st \} \supset \{st\} \subset \{st, stu\} \supset \{stu\}. \end{equation} Each of the singleton cosets above is an $(\mt, \mt)$-coset. The first doubleton is an $(\mt,
s)$-coset, the second an $(\mt, t)$-coset, and the third an $(\mt,u)$-coset. \end{ex}

Now we discuss the relations in our presentation of $\SC$. The quadratic relation $s * s = s$ in $(W,*,S)$ follows from a more local relation $[s,\mt,s] \expr [s]$ in $\SC$. Our \emph{general $*$-quadratic relation} is
\begin{equation} [Is,I,Is] \expr [I], \end{equation}
whenever $Is$ is finitary.

The relations in $\SC$ are the braid relations (the up-up, down-down, and switchback relations), and the general $*$-quadratic relation. In Theorem \ref{thm:reducing} we prove an even stronger statement, that any expression can be transformed into a reduced expression using the braid relations
and using the $*$-quadratic relation only in the length-reducing direction.

\begin{rem} One could define a category by generators and relations similar to $\SC$, but whose whose presentation consists only of the braid relations without the quadratic
relation. Whether this is a reasonable singular analogue of the braid monoid is currently unclear. \end{rem}

\begin{rem} The category $\SC$ is filtered by the length function. After linearizing, we can take the associated graded of $\SC$, another variant with a similar presentation: the braid relations are unchanged, and the quadratic relation is replaced by $[Is,I,Is] = 0$. This is analogous to replacing the relation $s * s = s$ with the relation $\partial_s \circ \partial_s = 0$, a relation which holds for Demazure operators on the cohomology ring of the flag variety. In forthcoming work with Libedinsky, we interpret the associated graded category of $\SC$ using Demazure operators acting on various (equivariant) cohomologies. \end{rem}

In type $A$, one can encode an expression diagrammatically as a \emph{web}, a special kind of labeled trivalent graph. This idea is developed in \S\ref{sec:typeAwebs}, where we
also discuss the connection between $\SC$ and the Cautis-Kamnitzer-Morrison web category \cite{CKM}. 

\subsection{The singular Coxeter complex}\label{subsec:coxeterintro}

The (completed) dual Coxeter complex $\Cox(W)$ is a contractible CW complex with one $k$-cell for each coset in $W / W_I$, for each finitary $I \subset S$ of size $k$. Thus the
$0$-cells correspond to elements of $W$, the $1$-cells to right multiplication by a simple reflection, and the $2$-cells to braid relations. For a comprehensive reference on the dual Coxeter complex see \cite[\S 12.3]{AbBr}. We repeat our picture from above in type $A_2$.

\begin{subequations}
\begin{equation}\label{A2coxintro} \Cox(S_3) := \qquad {
\labellist
\small\hair 2pt
 \pinlabel {$e$} [ ] at 35 5
 \pinlabel {$s$} [ ] at 5 30
 \pinlabel {$t$} [ ] at 70 30
 \pinlabel {$st$} [ ] at 5 75
 \pinlabel {$ts$} [ ] at 70 75
 \pinlabel {$sts$} [ ] at 35 100
\endlabellist
\centering
\ig{1}{CoxA2worient}
} \qquad \Cox_{\mt}(S_3) := \qquad {
\labellist
\tiny\hair 2pt
 \pinlabel {$\{e\}$} [ ] at 35 6
 \pinlabel {$\{s\}$} [ ] at 5 30
 \pinlabel {$\{t\}$} [ ] at 70 30
 \pinlabel {$\{st\}$} [ ] at 5 75
 \pinlabel {$\{ts\}$} [ ] at 70 75
 \pinlabel {$\{sts\}$} [ ] at 35 100
 \pinlabel {$\{e,s\}$} [ ] at 19 14
 \pinlabel {$\{e,t\}$} [ ] at 55 14
 \pinlabel {$\{s,st\}$} [ ] at -2 55
 \pinlabel {$\{t,ts\}$} [ ] at 75 55
 \pinlabel {$\{st,sts\}$} [ ] at 18 91
 \pinlabel {$\{ts,sts\}$} [ ] at 57 91
\endlabellist
\centering
\ig{1.6}{SingCoxA2worient}
} \end{equation}

For any Coxeter system $(W,S)$, we introduce a new $2$-dimensional CW complex $\Cox_{\mt}(W)$, constructed as follows. There will be one $0$-cell for each $(\mt,I)$ coset, for all
finitary $I$. In other words, there is one $0$-cell in $\Cox_{\mt}(W)$ for each cell of any dimension in $\Cox(W)$. The $1$-cells in $\Cox_{\mt}$ correspond to composition with a
generator in $\SC$; generators either add or subtract a simple reflection, which colors the edge. However, we only include as edges those compositions \emph{which occur in a
reduced expression} (this also explains the orientation). For example, there is no edge between the purple vertex and $\{t,ts\}$ above. Because of these missing edges $W$ does not act on $\Cox_{\mt}$ on the left, as it does on $\Cox$. The $2$-cells in $\Cox_{\mt}$ come from the braid relations. Above, we've colored the up-up and down-down relations khaki, and the switchback relations fuschia.

As we have drawn it, $\Cox_\mt$ appears to have a straight line through the purple vertex, including the blue vertex $v_1 = \{e,s\}$ and the red vertex $v_2 = \{ts,sts\}$. This is
no accident! One can think of $\Cox_{\mt}$ as describing the Coxeter hyperplane arrangement in $\R^{|S|}$, with one vertex for each facet (an $(\mt, \mt)$ coset for each chamber,
an $(\mt,s)$ coset for each $s$-wall, etcetera). The purple vertex represents the origin in this example. Here is $\Cox_{\mt}$ with the Coxeter arrangment overlaid upon it.
\begin{equation} \label{overlay} \ig{1}{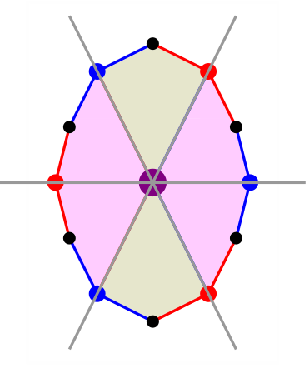} \end{equation}
Each reflection hyperplane has a positive side (containing $w_0$, when $W$ is finite) and a negative side (containing $e$). In Proposition~\ref{thm:inandout} we prove that an edge which goes onto a hyperplane is reduced if and only if it comes from the negative side, and an edge which goes off of a hyperplane is reduced if and only if it goes to the positive side. This principle explains all the arrows above, and also the lack of edge between the purple vertex and $\{t,ts\}$.

Understanding the extra linear structure on singular Coxeter complexes is rather useful, because reduced expressions are close to geodesics! Recall that a \emph{flat} of a
hyperplane arrangement is a subspace obtained as the intersection of some of the hyperplanes. In a reduced expression, once you leave a flat you can never return (you enter from
the negative side of some hyperplane, and leave from the positive side). So if two vertices lie on the same flat $F$, then any reduced path between them is constrained to stay in
$F$. For example, there is a unique reduced path from $v_1$ to $v_2$, which passes through the origin, since both vertices lie on a one-dimensional flat.

Our switchback relation may seem complicated, but it can be easily summarized. Take two distinct one-dimensional flats, and the appropriate vertices $v_1$ and $v_3$ on each.
For example, take the blue vertex $v_1 = \{e,s\}$ and the blue vertex $v_3 = \{st,sts\}$. Then there are two reduced paths from $v_1$ to $v_3$: the path through the origin, and the
path which avoids the origin but stays within the two-dimensional span of $v_1$ and $v_3$.

The switchback relation $[s, \mt, t, \mt, s] \expr [s,st,s]$ from \eqref{introbraidreln} is a relation between $(s,s)$-cosets, not between $(\mt, s)$-cosets. Of course, when \eqref{introbraidreln} is applied locally within a longer expression like \eqref{introbraidstart}, it does produce a relation between paths in $\Cox_{\mt}$. However, the reader may still want to visualize the switchback relation itself. Indeed, there is a CW-complex $\Cox_s(S_3)$ whose vertices are $(s,I)$-cosets, drawn as below.
\begin{equation} \Cox_s(S_3) := \qquad 
	{
	\labellist
	\tiny\hair 2pt
	 \pinlabel {$\{e,s\}$} [ ] at 35 8
	 \pinlabel {$\{e,s\}$} [ ] at 4 18
	 \pinlabel {$\{e,s,t,st\}$} [ ] at -5 35
	 \pinlabel {$\{t,st\}$} [ ] at 2 58
	 \pinlabel {$\{t,st,ts,sts\}$} [ ] at 3 73
	 \pinlabel {$\{ts,sts\}$} [ ] at 35 88
	 \pinlabel {$\{ts,sts\}$} [ ] at 60 63
	\endlabellist
	\centering
	\ig{1.6}{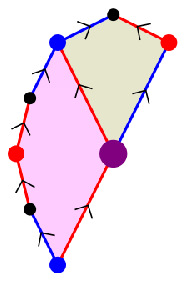}
	} \qquad \Cox_{st}(S_3) := \qquad \ig{1.6}{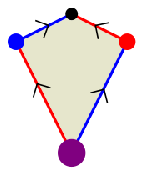}
\end{equation}
\end{subequations}
The switchback relation from \eqref{introbraidreln} is the fuschia 2-cell above. One can also define $\Cox_J$ for any finitary $J$, whose vertices are $(J,I)$-cosets for various $I$. Also above is $\Cox_{st}(S_3)$, where every vertex is the double coset with underlying set $S_3$.
We give examples for the symmetric group $S_4$ in \S\ref{sec:coxcomplex}.

The reader may note that $\Cox_J$ seems to embed inside $\Cox_{\mt}$, preserving the orientations. We prove this in Proposition~\ref{singcoxembeds}. One can think about $\Cox_J$ as a
fundamental domain for the parabolic subgroup $W_J$ inside $\Cox_{\mt}$, even though technically there is no left action. Each vertex in $\Cox_J$ matches a vertex in $\Cox_{\mt}$
with the same maximal element. For example, the unique reduced path from $v_1$ to $v_2$ in $\Cox_{\mt}$ discussed above implies that there is a unique reduced path between the
corresponding vertices in $\Cox_s$, whence a unique reduced expression for the $(s,t)$-coset $\{ts,sts\}$. See Proposition~\ref{prop:sss}\eqref{uniqrex} for a generalization.

We have not discussed $k$-cells for $k \ge 3$ inside the singular Coxeter complex. This is because we do not yet understand the combinatorics. One expects that what we have
described above is only the $2$-skeleton of a polytopal CW complex $\Cox_{\mt}(W)$, homeomorphic to and refining $\Cox(W)$, but a description of the higher cells is currently
missing. The $3$-cells should correspond to the generating cycles in reduced expression graphs, discussed in Remark \ref{rmk:Zam}.

\begin{rem} \label{rmk:DHP2} As noted in Remark \ref{rmk:DHP1}, the 1-skeleton of the singular Coxeter complex $\Cox_{\mt}$ also appears in work of Dermenjian-Hohlweg-Pilaud \cite{DHP}. Although their pictures like \cite[p.8]{DHP} appear to be part of a higher-dimensional CW complex, there is no suggestion of higher cells in that paper, which only discusses the 1-skeleton. In particular, our braid relations (the 2-cells) are new, and the higher cells remain mysterious. \end{rem}

\subsection{Motivation from geometry} \label{subsec-intro-geometry}

This section is designed to be readable even if the rest of the introduction was skipped.

Let $(W,S)$ be a Coxeter system of a Weyl group associated to a semisimple Lie group $G$ with Borel subgroup $B$. The Bruhat decomposition states that the $B$-orbits on the flag
variety $G/B$ are in bijection with $W$, and the closure relation corresponds to the Bruhat order. Thus it is not surprising that many geometric features of flag varieties can be
combinatorially encoded inside the poset $(W, \le)$. As the classic example, Kazhdan-Lusztig polynomials encode the intersection cohomology of Schubert varieties, and have a
combinatorial definition in terms of the Hecke algebra.

Geometric constructions can be associated not just to elements of $W$ but to expressions as well. One effective tool to study intersection cohomology (and perverse sheaves) is to
study the cohomology rings of resolutions of singularities. To any expression $\uw = [s_1, \ldots, s_d]$ (i.e. any word in $S$) one has an associated Bott-Samelson variety
$\BS(\uw)$, equipped with a map to $\Fl$. When $\uw$ is a reduced expression for $w \in W$, then $\BS(\uw)$ is a resolution of singularities for the Schubert variety of $w$ inside
$\Fl$. Many geometric features of $\BS(\uw)$ are encoded inside the theory of reduced expressions and subexpressions. Just as flag varieties have an affine paving indexed by $W$,
$\BS(\uw)$ has an affine paving indexed by subexpressions of $\uw$.

Ultimately, one is studying $B$-equivariant perverse sheaves on the flag variety, or $B \times B$-equivariant perverse sheaves on $G$. This is a monoidal category under
convolution, whose Grothendieck group is the Hecke algebra. It is this convolution structure which explains why multiplication in the Hecke algebra is useful, and it also explains
why expressions appear frequently in the combinatorics; one can think of Bott-Samelson varieties as convolutions over sequences of simple reflections.

The complicated geometric category of perverse sheaves has a simpler algebraic model constructed by Soergel \cite{Soer90} using so-called \emph{Soergel bimodules}. More recently,
Soergel bimodules can be replaced by a combinatorially-defined monoidal category, the \emph{diagrammatic Hecke category} $\HC$ \cite{EKho, EWGr4sb}. The objects in $\HC$ are
expressions, and one can construct a basis of morphisms between two expressions using the combinatorics of subexpressions, the \emph{double leaves basis} \cite{LibLL, EWGr4sb}.

What about $P_I$-orbits in $G/P_J$, where $P_I$ and $P_J$ are parabolic subgroups of $G$ associated to subsets $I, J \subset S$? As $I$ and $J$ vary, these orbits are parametrized
by double cosets $W_I\backslash~W/W_J$ for the parabolic subgroups $W_I, W_J \subset W$. What is the algebraic object encoding the intersection cohomologies of the closures of
these orbits? What is the algebraic/combinatorial category encoding equivariant perverse sheaves? These questions were studied at length by Williamson \cite{GWthesis} in his
thesis, which was rewritten into the article \cite{WillSingular}.

Just as convolution turned $B \times B$-equivariant perverse sheaves into a monoidal category, convolution turns parabolic-equivariant sheaves into a 2-category. There is one
object for each subset $I \subset S$, and the morphism category from $J$ to $I$ is the category of $P_I \times P_J$-equivariant perverse sheaves on $G$. To encode intersection
cohomology, Williamson introduces the \emph{Hecke algebroid}\footnote{It is called the Hecke category in \cite[Definition 2.2.1]{GWthesis} and called the Schur algebroid in \cite{WillSingular}.}, a category whose objects are also in bijection with subsets $I \subset
S$. Just as $B \times B$-equivariant sheaves are a special case of $P_I \times P_J$-equivariant sheaves corresponding to the case $I = J = \mt$, the endomorphism algebra of $\mt$
inside the Hecke algebroid is precisely the Hecke algebra. More interestingly, all morphism spaces $\Hom(I,J)$ inside the Hecke algebroid have a Kazhdan-Lusztig basis, in bijection
with double cosets $W_I\backslash~W/W_J$.

Williamson defines analogues for double cosets of the notion of (reduced) expressions\footnote{They are called \emph{(reduced) right translation sequences} in \cite[Definition 1.3.2]{GWthesis}.}. He uses reduced expressions to construct resolutions of singularities for general Schubert varieties\footnote{This is implicit in the thesis, which does connect (without proof) singular Soergel bimodules to perverse sheaves.}. Unlike the ordinary theory of reduced expressions which has an
incredibly rich literature, Williamson's reduced expressions for double cosets have received very little attention, and many basic properties were unknown until now. We warn the reader that the article \cite{WillSingular} omits some of the development on this theme which appeared in the original thesis \cite{GWthesis}.

Williamson also defines the analogue of Soergel bimodules, the 2-category of \emph{singular Soergel bimodules} \cite[\S 3]{GWthesis}. We mention at this point that (singular)
Soergel bimodules can be defined not just for Weyl groups but for any Coxeter system \cite{Soer07}. They only have a connection to geometry for Weyl groups and
other crystallographic groups, but many desirable properties which one can prove for geometric reasons in the Weyl group case can be proven in other ways for arbitrary Coxeter
groups, e.g. \cite{EWHodge}. When working with infinite Coxeter groups, the objects in the Hecke algebroid or the 2-category of singular Soergel bimodules are parametrized by
subsets $I \subset S$ whose parabolic subgroup $W_I$ is finite, in which case we call $I$ \emph{finitary}. Singular Soergel bimodules and their geometric analogues (especially in
affine type) have already played a major role in geometric and modular representation theory, see \cite{ELosev, EQuantumI, RicWil} for example.

The primary goal of this paper is to introduce the braid relations between expressions for double cosets, and to establish an analogue of Matsumoto's theorem, for arbitrary Coxeter
groups. We also provide a presentation of the singular Coxeter monoid, an algebraic category underlying the Hecke algebroid in the same way that the Coxeter group
underlies the Hecke algebra. The technology we develop is a crucial step towards the combinatorial study of the (diagrammatic) singular Hecke category.

\begin{rem} \label{rmk:singular} In the literature, the (sometimes misleading) words \emph{singular} and \emph{parabolic} are used (not interchangeably) as adjectives to describe
objects related to $W_I\backslash~W/W_J$. Examples include singular or parabolic category $\mathcal{O}$, singular Soergel bimodules, etcetera. Usually the adjective singular
involves a transformation in categorical type, similar to the transformation from a monoid to a category. For example, ordinary Soergel bimodules are a monoidal category, whereas
singular Soergel bimodules are a 2-category, with objects indexed by finitary parabolic subgroups. The exception is the Hecke algebroid, which might have been called the ``singular
Hecke algebra.'' We prefer the ``singular'' notation to the ``-oid'' notation in this paper to avoid conflicts with the literature, such as the unrelated concept of
Weyl groupoids. \end{rem}

\subsection{Outline of paper}

In \S\ref{sec-basicstuff} we recall basic facts about $*$-multiplication, double cosets, and Williamson's theory of singular expressions. We tried to place all the previously known definitions and results in \S\ref{sec-basicstuff}. In addition, \S\ref{sec-basicstuff} proves basic properties of $\SC$, and relates them to singular expressions.

In \S\ref{sec-newdesc} we give several new equivalent descriptions of singular reduced expressions. In \S\ref{sec-lotsoprops} we use these new decsriptions to prove many properties of singular reduced expressions.

In \S\ref{sec-relating} we study relations between expressions. We construct the switchback relation and the other relations, and verify that they hold in $\SC$. We prove the singular Matsumoto theorem in \S\ref{mats}, and the section culminates with a proof that our presentation of $\SC$ is correct in \S\ref{ss:presentation}. In \S\ref{ssstypes} we explicitly elaborate on the switchback relation in all finite types.

In \S\ref{sec:typeAwebs} we connect $\SC$ with webs in type $A$. We construct a monoidal category built from $\SC$ for all symmetric groups. Our main result states that this
monoidal category and the Cautis-Kamnitzer-Morrison category of type $A$ webs are isomorphic after taking the associated graded with respect to our length filtration.

In \S\ref{sec:coxcomplex} we provide additional details about singular Coxeter complexes. Our intent is not to prove theorems, but mostly to showcase this new construction and its
mysteries.

\subsection{Acknowledgments}

The first-named author suggested the singular Matsumoto theorem as a thesis problem to his erstwhile student Janelle Currey. It was during this collaboration that we realized the
utility of separating reducedness from the forward path, and other early insights. Janelle also drew inspiring pictures of the singular Coxeter complexes (much nicer than the
reproductions in \S\ref{sec:coxcomplex}), which led to Proposition \ref{singcoxembeds}. We wish to express our deep gratitude to Janelle for her initial work on this project.

For assistance in naming the switchback relation, the first author wishes to thank pu-band-random and the Goremands. Special accolades go to Mark Daniels for the winning name.
Thanks also to Chris Randall for the runner-up suggestion of Loch Lomond. When comparing reduced expressions we now casually refer to the high road and the low road.

We would also like to thank Sara Billey, Patricia Hersh, Nicolas Libedinsky, Rafael Mrden, Nathan Reading, Vic Reiner, and Geordie Williamson for useful conversations.  The first author was supported by NSF grants DMS-1553032 and DMS-1800498. This paper was written while the first author was visiting the IAS, a visit supported by NSF grant DMS-1926686.

\section{Singular expressions and $\SC$} \label{sec-basicstuff} 


\subsection{Properties of $*$-multiplication} \label{ss:basics}

We recall some notation. Fix a Coxeter system $(W,S)$.  We denote by $\leq$ the Bruhat order on $W$ and $\ell(w)$ the Coxeter length of $w\in W$. 
For $I\subset S$, the parabolic subgroup generated by $I$ is denoted by $W_I$, and $W^I$ (resp., $^I W$) denotes the set of shortest representatives for the cosets $W/W_I$ (resp., $W_I\setminus W$). Even though we regularly use the $*$-product, when we discuss cosets like $x W_I$ we always refer to cosets in the group $W$.

\begin{notation} We use the notation $w=x.y$ as a shorthand for the statement: $w=xy$ and $\ell(w)=\ell(x)+\ell(y)$. \end{notation}

From the definition of $*$-multiplication, we see that
\begin{equation} \label{basicstar} x*s = x.s \text{ when } xs>x, \quad x * s = x \text{ when } xs < x.\end{equation} 
We now state some basic and well-known results.

\begin{lem} \label{starissubexp} For any $x, y \in W$, $x * y = x . y'$ for some $y' \le y$. In particular, if $y \in W_I$ then $x * y \in x W_I$. \end{lem}

\begin{proof} Choosing a reduced expression $y = s_1 s_2 \cdots s_d$, we know that $y = s_1 * s_2 * \cdots * s_d$. Also, $x * s$ is either equal to $x . s$ or to $x$. Thus
\begin{equation} x * s_1 * s_2 * \cdots * s_d = x . s_{i_1} . s_{i_2} . \cdots . s_{i_k} \end{equation}
for some $k \le d$ and $1 \le i_1 < i_2 < \ldots < i_k \le d$.  Then $s_{i_1} s_{i_2} \cdots s_{i_k}$ is a reduced expression for some $y' \in W$. Using the ``subexpression definition'' of the Bruhat order, one sees that $y' \le y$. \end{proof}

\begin{lem} \label{dotvsstar} For $x, y \in W$ one has $xy = x.y$ if and only if $xy = x*y$. \end{lem}

\begin{proof} That $xy = x.y$ implies $x*y = xy$ is fairly well-trodden (being built into the implicit bijection between $(W,*,S)$ and $W$), so we leave it to the reader. For the
converse, Lemma \ref{starissubexp} says there is some $y'$ such that $x*y = x.y'$. If $x*y = xy$ then $y = y'$ (since $W$ is a group), so $xy = x*y = x.y$. \end{proof}
	
\begin{notation} Let $\leftdes(w) \subset S$ (resp. $\rightdes(w) \subset S$) denote the left (resp. right) descent set of an element $w \in W$. Recall that descent sets are always finitary. \end{notation}

\begin{lem} \label{lem:dontchangeleftdescent} For any $w, x \in W$, $\leftdes(w) \subset \leftdes(w * x)$ and $\rightdes(x) \subset \rightdes(w * x)$. \end{lem}

\begin{proof} It is an immediate consequence of \eqref{basicstar} that $x*s = x$ if and only if $s \in \rightdes(x)$. Now the result follows from associativity:
\begin{equation} s \in \rightdes(x) \iff x*s = x \implies (w*x)*s = w*(x*s)= w * x \iff s \in \rightdes(w * x). \end{equation}
The statement about left descents is proven similarly.
\end{proof}

\begin{lem} \label{descentvspreserved} If $I \subset S$ is finitary, then $w * w_I = w$ if and only if $I \subset \rightdes(w)$. \end{lem}

\begin{proof} By Lemma \ref{lem:dontchangeleftdescent}, the right descent set of $w * w_I$ contains $\rightdes(w_I) = I$, which proves one direction. Conversely, if $I \subset \rightdes(w)$ then 
\[ w * w_I = w * s_{i_1} * \cdots * s_{i_d} = w. \]
Here, $s_{i_1} \cdots s_{i_d}$ is a reduced expression for $w_I$, consisting only of simple reflections in $I$, and the second equality holds by \eqref{basicstar}. \end{proof}

\subsection{$\SC$ as Karoubi envelope} \label{ss:karoubi}

We now discuss maximal elements of double cosets in the context of star multiplication. In Lemma \ref{doublecosetlem} we recall that each $(J,I)$-coset $p$ has a unique (maximal) element $\ma{p}$ satisfying the property that $J \subset \leftdes(\ma{p})$ and $I \subset \rightdes(\ma{p})$.

\begin{lem} \label{lem:maxofcoset} For any $(I,J)$-coset $p$ and any $(J,K)$-coset $q$, $\ma{p} * \ma{q}$ is the maximal element in its $(I,K)$-coset. \end{lem}

In the language of \eqref{SCcompformula}, this would say that for any $x, y \in W$ one has $\overline{(\ox * \oy)} = \ox * \oy$. 

\begin{proof} By Lemma \ref{lem:dontchangeleftdescent}, $\ma{p} * \ma{q}$ has $I$ in its left descent set and $K$ in its right descent set. Thus $\ma{p} * \ma{q}$ is maximal in its double coset. \end{proof}
	

\begin{prop} The category $\SC$ is well-defined. \end{prop}

\begin{proof} Let $I$ be finitary. If $I \subset \leftdes(x)$ then $w_I * x = x$ by Lemma \ref{descentvspreserved}. Thus the minimal $(I,I)$-coset, whose maximal element is $w_I$,
serves as an identity element inside $\End_{\SC}(I)$.

Given double cosets $W_I x W_J$, $W_J y W_K$, and $W_K z W_L$, with maximal elements $\ox$, $\oy$, and $\oz$ respectively, associativity is the statement that
\begin{equation} W_I\; \overline{(\ox * \oy)} * \oz\; W_K = W_I \;\ox * \overline{(\oy * \oz)}\; W_K. \end{equation}
This follows from the associativity of $*$, since $\overline{(\ox * \oy)} = \ox * \oy$ by the previous lemma. \end{proof}

After this point we never use the notation $\ox$, $\oy$ again, preferring the notation $\ma{p}$ for a double coset $p$. For an element $x \in w$, the notatation $\ox$, understood
as the maximal element of the $(I,J)$-coset containing $x$, will depend strongly on the choice of $(I,J)$. Meanwhile, $p$ is a double coset already, and $(I,J)$ are understood as
part of its data.


There is a standard construction one can apply to any category, which formally adjoins the image of any idempotent as a new object. Applying this process to all idempotents in a
category (typically preadditive, though not in our case), one obtains the so-called \emph{Karoubi envelope}\footnote{Standard online references, such as Wikipedia, are decent sources for learning about the Karoubi envelope.}.

\begin{lem} Let $w \in W$ satisfy $w * w = w$. Then $w = w_I$ for some finitary $I \subset S$. \end{lem}

\begin{proof} Let $I = \rightdes(w)$. We know that $w_I$ is the unique element of $W_I$ whose right descent set contains $I$. Thus $w = w_I$ if and only if $w \in W_I$. Pick a reduced expression $w = s_{i_1} \cdots s_{i_d}$. Following the proof of Lemma \ref{starissubexp}, one has 
\[ w * w = w * s_{i_1} * \cdots * s_{i_d}. \]
Using \eqref{basicstar} repeatedly, $w * w = w$ if and only if $w * s_{i_k} = w$ for all $1 \le k \le d$, if and only if $s_{i_k} \in I$ for all $1 \le k \le d$, in which case $w \in W_I$. \end{proof}

Thus the Karoubi envelope of $(W,*,S)$ (viewed as a category with one object) has one object for each finitary parabolic subgroup. Let us denote this Karoubi envelope by
$\Kar(W)$, and denote the objects by $I$ as in $\SC$.

\begin{defn} Let $F : \SC \to \Kar(W)$ be the functor which is defined on objects as $I \mapsto I$. On morphisms, a $(J,I)$-coset $p \in \Hom_{\SC}(I,J)$ is sent by $F$ to the element $\ma{p} \in \Hom_{\Kar(W)}(I,J)$. \end{defn}

\begin{thm} \label{thm:Karoubi} The functor $F$ is well-defined, and is an equivalence of categories between $\SC$ and $\Kar(W)$. \end{thm}

\begin{proof} By the Karoubi construction, $\Hom_{\Kar(W)}(I,J)$ is equal to those elements of $W$ which are preserved by composition on the right with $w_I$ and composition on the
left with $w_J$. By Lemma \ref{descentvspreserved}, this is the same as those elements with $I$ in their right descent set and $J$ in their left descent set. Any such element is
maximal in its $(J,I)$-coset, and each $(J,I)$-coset has a unique such element. Thus $F$ does send each double coset to a well-defined morphism, and $F$ is a bijection on Hom sets.

Clearly $\id_J \in \End_{\SC}(J)$ is sent by $F$ to $w_J \in \End_{\Kar(W)}(J)$, which is the identity element there. If $p$ is a $(J,I)$-coset and $q$ is a $(K,J)$-coset, then by Lemma~\ref{lem:maxofcoset}, $F(q
* p) = \ma{q * p} = \ma{q} * \ma{p} = F(q) \circ F(p)$. Thus $F$ is a functor. \end{proof}

\subsection{Maximal and minimal elements in double cosets} \label{subsec-cosets}

Recall that $W^I$ denotes the set of minimal representatives inside the right cosets $W / W_I$. If $I \subset J$ then $W^I \cap W_J$ has a longest element $x$ which satisfies $x .
w_I = w_J$. We often write \begin{equation} x = w_J w_I\inv \end{equation} (even though $w_I\inv = w_I$) to emphasize the fact that the lengths add in the formula $x . w_I = w_J$.
Similarly, $w_I\inv w_J$ is the longest element of $^I W \cap W_J$.

We record some useful properties of parabolic double cosets in $W$.

\begin{lem}\label{doublecosetlem} 
Let $J,I\subset S$ be finitary and let $p\in \JIcosets$.
\begin{enumerate}
    \item \label{doublecosetinterval} The subposet $p\subset (W,\leq)$ is an interval, that is, there are elements $\ma{p},\mi{p}\in p$ such that $p=\{x\in W\ |\ \mi{p}\leq x\leq \ma{p}\}$.
    \item \label{double rex decomposition} (Howlett's theorem) Let $K = {J\cap \mi{p} I \mi{p}\inv}$ and $L = {I \cap \mi{p}\inv J \mi{p}}$. The following map is bijective.
    \begin{equation} \label{keybij}
        \begin{split}
            (W^K\cap W_J)\times W_I  &\to p\\
            (x,y) &\mapsto x.\mi{p}.y
        \end{split}
    \end{equation}
    Similarly, the following map is bijective.
	\begin{equation} \label{keybij2}
        \begin{split}
            W_J \times ({}^{L}W \cap W_I)  &\to p\\
            (x,y) &\mapsto x.\mi{p}.y
        \end{split}
    \end{equation}
	\item One has
	\begin{equation} \label{maxformula} \ma{p} = (w_J w_K\inv) . \mi{p} . w_I = w_J . \mi{p} . (w_{L}\inv w_I). \end{equation}
	\item If $z$ is any element of $p$, then there exist $x, x' \in W_J$ and $y, y' \in W_I$ such that
	\begin{subequations}
	\begin{equation} \label{elementfrommin} z = x . \mi{p} . y, \end{equation}
	\begin{equation} \label{maxfromelement} \ma{p} = x' . z . y'. \end{equation}
	\end{subequations}
	\item If $z \in p$ then 
\begin{subequations}
	\begin{equation} \label{beingmaximal} z = \ma{p} \iff J \subset \leftdes(z), I \subset \rightdes(z), \end{equation}
	\begin{equation} \label{beingminimal} z = \mi{p} \iff J \cap \leftdes(z) = \emptyset, I \cap \rightdes(z) = \emptyset. \end{equation}
\end{subequations}
\end{enumerate}
\end{lem}


\begin{proof}
See \cite[Proposition 31]{Reading} for (1) and \cite[Proposition 8.3]{GarsiaStanton} for (2).
To give the reader some additional context, we elaborate on the bijection in \eqref{keybij}. If $w = x . \mi{p} . y$ then the shortest element in the right coset $w W_I$ is $x . \mi{p} = w y\inv$, and the shortest element in the left coset $W_J (x . \mi{p})$ is $\mi{p}$. This gives an effective way to compute $x$ and $y$ from a given element $w$.


Since \eqref{keybij} (resp. \eqref{keybij2}) preserves length up to the constant $\ell(\mi{p})$, the longest element of $p$ is the image of the $(x,y)$ where $x$ and $y$ are the longest elements in their domain. This proves \eqref{maxformula}.

The bijection \eqref{keybij} immediately proves the existence of $x$ and $y$ such that \eqref{elementfrommin} holds. Moreover, we can assume that $x \in W^K \cap W_J$. By basic Coxeter theory, there exist $x'$ and $y'$ such that $x' . x = (w_J w_K\inv)$ and $y . y' = w_I$. Then combining \eqref{elementfrommin} and \eqref{maxformula} we obtain \eqref{maxfromelement}.

It is clear that $J \subset \leftdes(\ma{p})$ and $I \subset \rightdes(\ma{p})$. Conversely, suppose $J \subset \leftdes(z)$ and $I \subset \rightdes(z)$. Writing $\ma{p} = x' . z
. y'$ as in \eqref{maxfromelement}, we see that $x' = y' = e$, or else the lengths will not add. The analogous statement for $\mi{p}$ is straightforward. \end{proof}

\begin{rem} A final warning. A subset $p \subset W$ can be an $(I,J)$-coset in multiple different ways, as noted in Example \ref{ex:ambiguous}. The elements $\mi{p}$ and $\ma{p}$ do not depend on the choice of $(I,J)$ among such valid choices, though the elements $x$ and $y$ appearing in formulas like \eqref{keybij} will depend on this choice. Again, we never discuss a double coset $p$ without the pair $(I,J)$ being established by the context. \end{rem}

Let $p$ be a $(J,I)$-coset. The subset $K = J \cap \mi{p} I \mi{p}\inv$ associated to $p \in \JIcosets$ will appear many times in this paper, and we call it the \emph{(left)
redundancy}. The parabolic subgroup $W_K \subset W_J$ accounts for the redundancy between left multiplication by $W_J$ and right multiplication by $W_I$ on $\mi{p}$. Similarly,
$L = I \cap \mi{p}\inv J \mi{p}$ is called the \emph{right redundancy}. The following lemma helps explain why redundancy is a parabolic subgroup.

\begin{lem}\label{J WJ same} (Kilmoyer's theorem)
Let $I,J\subset S$ be finitary and $p \in \JIcosets$. Then 
\begin{equation} \label{K=} \mi{p} I \mi{p}\inv \cap W_J = \mi{p} I \mi{p}\inv \cap J = \mi{p} W_I \mi{p}\inv \cap J. \end{equation}
\end{lem}

\begin{proof}
We prove the first equality, and leave the second to the reader. Let $s \in I$ be such that $x:= \mi{p} s \mi{p}\inv \in W_J$. Since $\mi{p} < \mi{p} . s$, the element $\mi{p} s = x \mi{p}$ has length $\ell(\mi{p}) + 1$. Since $x \in W_J$, $x \mi{p} = x . \mi{p}$, so $x \mi{p}$ has length $\ell(\mi{p}) + \ell(x)$. It follows that $\ell(x) = 1$, so $x \in J$.
\end{proof}

One major use of \eqref{keybij} will be to compare minimal and maximal elements of different double cosets when one is contained inside the other. The moral of the next lemma is: when we change the parabolic subgroup on the right, then we can get between minimal (resp. maximal) elements only using multiplication on the right.

\begin{lem} \label{lem:onlyneedright} Let $J, I, I'$ be finitary with $I \subset I'$. Suppose that $p$ is a $(J,I)$-coset and $q$ is a $(J,I')$-coset with $I \subset I'$ and $p \subset q$. Then there exist $y, y' \in W_{I'}$ such that $\mi{p} = \mi{q} . y$ and $\ma{p} . y' = \ma{q}$. \end{lem}

\begin{proof} By \eqref{elementfrommin} we know that $\mi{p} = x . \mi{q} . y$ for some $x \in W_J$ and $y \in W_{I'}$. However, left multiplication by any non-identity element of $W_I$ will increase the length of $\mi{p}$, a contradiction if $x \ne e$ (since multiplication by $x^{-1}$ will both increase and decrease length). 

By \eqref{maxfromelement} we know that $\ma{q} = x' . \ma{p} . y'$ for some $x' \in W_J$ and $y' \in W_{I'}$. However, $J \subset \leftdes(\ma{p})$, so $x' = e$. \end{proof}

\subsection{Expressions and paths} \label{subsec-paths}

Most of the definitions\footnote{Only the names have been changed to protect the innocent.} in this section were made in \cite[Chapter 2, Section 1.3]{GWthesis}.

\begin{defn}\label{def:expression}
A \emph{(singular) expression} is a sequence of finitary subsets $I_\bullet = [I_0,\cdots, I_d]$ in $S$ such that $I_i$ and $I_{i-1}$ differ by one element (i.e., either $I_{i-1} = I_i\sqcup\{s\}$ or $I_{i-1}\sqcup\{s\}= I_i$ for some $s\in S$) for each $i$. An expression with $I_0 = J$ will be called a \emph{$J$-expression}.
\end{defn}

As noted in the introduction, an expression represents a composition of generating morphisms in the category $\SC$. We will temporarily ignore this fact, and first focus on the analogue of subexpressions.

\begin{defn}\label{subex}
Let $I_\bullet$ be a $J$-expression. A \emph{path} subordinate to $I_\bullet$ is a sequence of double cosets $p_{\bullet} = [p_0,\cdots ,p_d]$, where $p_i\in W_J\setminus W / W_{I_i}$, satisfying two conditions: \begin{itemize} \item $e\in p_0$, \item $p_{i-1}\cap p_i\neq \emptyset$ for all $1 \le i \le d$. \end{itemize}The \emph{endpoint} of a path $[p_0,\cdots,p_d]$ is the $(J,I_d)$-coset $p_d$. The \emph{(left) redundancy sequence} of the path is the sequence $(K_0, \ldots, K_d)$ where $K_i = J \cap \mi{p_i} I_i \mi{p_i}\inv$. 
\end{defn}

There is a dichotomy which runs throughout this paper: in an expression, either $I_i = I_{i-1} \sqcup s$ or $I_i = I_{i-1} \setminus s$ (for some $s \in S$). If $s \notin I$, then a $(J,I)$-coset is either entirely contained in or disjoint from any $(J,Is)$-coset. Because of this, the condition that $p_{i-1} \cap p_i \neq \emptyset$ in a path can be rephrased more clearly as follows. \begin{itemize} \item If $I_i = I_{i-1} \sqcup s$ then $p_{i-1} \subset p_i$. Note that $p_i$ is uniquely determined as the double coset containing $p_{i-1}$. \item If $I_i = I_{i-1} \setminus s$ then $p_i \subset p_{i-1}$. There may be many choices of $p_i$, depending on how $p_{i-1}$ splits into double cosets upon restriction from $W_{I_{i-1}}$ to $W_{I_i}$. \end{itemize}

We give several examples. In each, we point out one special path, the forward path, see Definition \ref{good subex}.

\begin{ex} Let $m_{st} = 3$ and consider the expression $[\mt, s, \mt, t, \mt, s, \mt]$. There are eight paths subordinate to this expression, corresponding to the eight subexpressions of $sts$. Here are two paths, and the subexpressions they correspond to.
\[ s^1 t^0 s^0: \qquad (\{e\}, \{e,s\}, \{s\}, \{s, st\}, \{s\}, \{e, s\}, \{s\}) \]
\[ s^0 t^1 s^1: \qquad (\{e\}, \{e,s\}, \{e\}, \{e, t\}, \{t\}, \{t, ts\}, \{ts\}) \]
Note that every path is a sequence of $(\mt, I)$-cosets for various $I$, so the redunancy path is constantly the empty set. This last example is the forward path.
\[ s^1 t^1 s^1: \qquad (\{e\}, \{e,s\}, \{s\}, \{s, st\}, \{st\}, \{st, sts\}, \{sts\}) \]
\end{ex}

\begin{ex} \label{ex:rando} Let $m_{st} = 3$ and consider the expression $[s, \mt, s, st, t, \mt, s]$. There are three paths subordinate to this expression. Note that every path is a sequence of $(s,I)$-cosets for various $I$. We also record their redundancy sequences.
\[ (\{e,s\}, \{e,s\}, \{e,s\}, W_{st}, \{e,s,t,st\}, \{e,s\}, \{e, s\}) \quad \text{redundancy } = (s, \mt, s, s, \mt, \mt, s) \]
\[ (\{e,s\}, \{e,s\}, \{e,s\}, W_{st}, \{e,s,t,st\}, \{t,st\}, \{t, ts, st, sts\}) \quad \text{redundancy } = (s, \mt, s, s, \mt, \mt, \mt) \]
\[ (\{e,s\}, \{e,s\}, \{e,s\}, W_{st}, \{ts,sts\}, \{ts,sts\}, \{t,ts, st, sts\}) \quad \text{redundancy } = (s, \mt, s, s, s, \mt, \mt) \]
We write $W_{st} = \{e,s,t,st,ts,sts\}$ for the entire parabolic subgroup, as a double coset. The third path above is the forward path.
\end{ex}

\begin{ex} Let $m_{st} = 3$ and consider the expression $[\mt, s, st, s, \mt]$. There are six paths subordinate to this expression, one for each element $w$ of the parabolic subgroup $W_{st}$. The corresponding path is
\[ (\{e\}, \{e,s\}, W_{st}, \{w, ws\}, \{w\}). \]
The path associated to $w = sts$ is the forward path of the expression.
\end{ex}

\begin{defn}\label{good subex}
A path $p_\bullet$ is said to be \emph{forward} if for each $i$ with $I_{i+1}\subset I_{i}$, the double coset $p_{i+1}$ is \emph{maximal} in $p_i$, meaning that $\ma{p_{i+1}}=\ma{p_i}$ (see Lemma \ref{doublecosetlem} \eqref{doublecosetinterval}). 
An expression has a unique forward path. We say that $p$ \emph{is the endpoint of an expression} $I_\bullet$ if the forward path of $I_\bullet$ has the endpoint $p$. We say \emph{$I_\bullet$ expresses $p$} and write $I_\bullet \expr p$.
\end{defn}

A singular expression is the analogue of a $*$-expression. A path is the analogue of a subexpression\footnote{Perhaps it is better to say that a path is the analogue of the Bruhat stroll of a subexpression, see \cite[\S 2.4]{EWGr4sb}. In the combinatorics of the Hecke category, the primary use of subexpressions is based on their Bruhat strolls, see e.g. the Deodhar defect formula, or the light leaf construction \cite[\S 2.4 and \S 6.1]{EWGr4sb}.}. The forward path is the analogue of the unique subexpression which is the
expression itself. If the singular expression represents a sequence of generating morphisms in $\SC$, then the endpoint of the forward path is the overall composition, as we now
prove. Recall that a double coset is called \emph{minimal} if it contains the identity element.

\begin{prop} \label{prop:expressesinSC} Let $I_\bullet$ be a $J$-expression, with forward path $p_{\bullet}$.  For each $i$, let $q_i$ be the minimal $(I_{i-1},I_i)$-coset. Let $\id_J$ be the minimal $(J,J)$-coset. Then for all $0 \le i \le d$ we have the following equality of morphisms in $\SC$:
	\begin{equation} p_i = \id_J * q_1 * \cdots * q_i. \end{equation} \end{prop}

\begin{proof} By definition of a path, $p_0 = \id_J$. We need only confirm that $p_i = p_{i-1} * q_i$ for each $1 \le i \le d$. To show that two double cosets are equal, we need only show they intersect nontrivially. There are two cases to check: when $I_{i-1} \subset I_i$ and when $I_{i-1} \supset I_i$ .
	
If $I_{i-1} \subset I_i$ then $\ma{q_i} = w_{I_i}$. By definition, $p_{i-1} * q_i$ is the coset containing $\ma{p_{i-1}} * w_{I_i}$, and this element is contained in $\ma{p_{i-1}}
W_{I_i}$ by Lemma \ref{starissubexp}. Thus $p_{i-1} * q_i$ contains $\ma{p_{i-1}}$. Since $p_i$ is the unique double coset containing $p_{i-1}$, it also contains $\ma{p_{i-1}}$,
and thus equals $p_{i-1} * q_i$.

If $I_{i-1} \supset I_i$ then $p_i$ contains $\ma{p_{i-1}}$ by definition of a forward path. Meanwhile, $\ma{q_i} = w_{I_{i-1}}$, and $\ma{p_{i-1}}$ already has $I_{i-1}$ in its right descent set. Thus $\ma{p_{i-1}} * w_{I_{i-1}} = \ma{p_{i-1}}$. Again, $p_{i-1} * q_i = p_i$, as desired.
 \end{proof}

\begin{cor} \label{cor:starformap} Let $I_{\bullet} = [I_0, \ldots, I_d]$ and $I_{\bullet}\expr p$. Then
\begin{equation} \label{starformap} \ma{p} = w_{I_0} * w_{I_1} * \cdots * w_{I_d}. \end{equation}
\end{cor}

\begin{proof} This is just a more explicit unraveling of Proposition \ref{prop:expressesinSC}, rephrased using Theorem \ref{thm:Karoubi}. Using the notation of Proposition \ref{prop:expressesinSC}, it is easy to see that $\ma{q_i} = w_{I_{i-1}} * w_{I_i}$, and $\ma{\id_J} = w_{I_0}$. By Theorem \ref{thm:Karoubi}, $\ma{q * q'} = \ma{q} * \ma{q'}$ for two composable double cosets. Thus
\[ \ma{p} = w_{I_0} * (w_{I_0} * w_{I_1}) * (w_{I_1} * w_{I_2}) * \cdots * (w_{I_{d-1}} * w_{I_d}) = w_{I_0} * w_{I_1} * \cdots * w_{I_d} \]
as desired.
\end{proof}

\subsection{Williamson's definition of reduced expressions}

Any $p\in \JIcosets$ has an expression. In fact, any $p$ has a reduced expression in the following sense (see Proposition~\ref{p has rex}).

\begin{defn}\label{def:reduced}
Let $I_{\bullet}$ be a $J$-expression, let $p_\bullet$ be a subordinate path, and let $K_{\bullet}$ be the redundancy sequence of $p_{\bullet}$. Then $p_{\bullet}$ is \emph{reduced} if, for each $i$ with $I_{i}\subset I_{i+1}$, one has
\begin{subequations} \label{reducedconditions} 
\begin{equation} \label{reducedminimal} \text{the double coset } p_i \text{ is \emph{minimal} in } p_{i+1}, \text{ meaning } \mi{p_i} = \mi{p_{i+1}}, \end{equation}
\begin{equation} \label{sameK} K_i = K_{i+1}. \end{equation}
\end{subequations}
We say that $I_{\bullet}$ is \emph{reduced} if its forward path is \emph{reduced}.
\end{defn}

The reader may wish to recall Lemma \ref{J WJ same} and the preceding discussion for more on redundancy. We will give several other ways to think about the mysterious condition
\eqref{sameK} below. The reader may also wish to skip immediately to Lemma~\ref{lem:reducedmaxs}, which gives an efficient condition for checking whether a path is reduced.

If $I_{i+1} = I_i \sqcup s$ and \eqref{reducedminimal} holds, then \eqref{sameK} is equivalent to
$\mi{p_i} s \mi{p_i}\inv \notin J$. This is because $\mi{p_i} s \mi{p_i}\inv$ is the only possible element in $K_{i+1} \setminus K_i$. Note that when $J = \mt$ one need only worry
about the first condition \eqref{reducedminimal}.

\begin{ex} \label{ex:nonrexpart1} The expression $[\mt,s,\mt,t, st]$ is not reduced. Its forward path is
\[ (\{e\}, \{e,s\}, \{s\}, \{s,st\}, W_{st}), \]
and $\{s,st\}$ is not the minimal $(\mt,t)$-coset inside the $(\mt,st)$-coset $W_{st}$. \end{ex}

\begin{ex} When $m_{st} = 3$ there are six reduced expressions for the $(\mt, \mt)$-coset $\{sts\}$:
\[ [\mt,s,\mt,t,\mt,s,\mt], \quad [\mt,t,\mt,s,\mt,t,\mt], \]
\[ [\mt,s,st,s,\mt], \quad [\mt,s,st,t,\mt], \quad [\mt,t,st,s,\mt], \quad [\mt,t,st,t,\mt]. \]
\end{ex}

\begin{ex} \label{ex:vsstandardrex}
The expression $[\mt,s_1,\mt,s_2,\cdots,\mt,s_d]$ is reduced if and only if $s_1\cdots s_d$ is a reduced expression in $W$. 
\end{ex}

\begin{ex} \label{ex:smts} As in the previous example, the expression $[\mt, s, \mt ,s]$ is not reduced. Its forward path is $(\{e\}, \{e,s\}, \{s\}, \{e,s\})$, and reducedness fails because of
\eqref{reducedminimal}: $\{s\}$ is not minimal inside $\{e,s\}$. Meanwhile, the expression $[s,\mt, s]$ is also not reduced. Its forward path is $(\{e,s\}, \{e,s\}, \{e,s\})$, and
this time the reducedness fails because of \eqref{sameK}: the redundancy increases from $K_1 = \mt$ to $K_2 = s$. See Remark \ref{rmk:primal} for more commentary on this phenomenon. \end{ex}

When examining whether an expression is reduced or not, it helps to examine conditions \eqref{reducedconditions} one index at a time.

\begin{defn} \label{defn:pqred} Let $J, I, I' \subset S$ be finitary, where $I \subset I'$ or $I' \subset I$. Let $p$ be a $(J,I)$-coset and $q$ be a $(J,I')$-coset with $p \cap q \ne \emptyset$. 
Let $K = J \cap \mi{p} I \mi{p}\inv$ and $K' = J \cap \mi{q} I' \mi{q}\inv$. We say that $[p,q]$ is \emph{reduced} if either 
\begin{itemize} \item $I \supset I'$, or \item $I \subset I'$ and $K = K'$ and $\mi{p} = \mi{q}$. \end{itemize}
We say that $[q,p]$ is \emph{reduced} if $\ma{p} = \ma{q}$.

Now use the notation of Definition \ref{def:reduced}. We say that $p_{\bullet}$ is \emph{reduced at index $i$} if \emph{$[p_i, p_{i+1}]$ is reduced}. When $p_{\bullet}$ is the forward path of $I_{\bullet}$, we also say that $I_{\bullet}$ is \emph{reduced at index $i$}. \end{defn}

\begin{ex} The expression $[s, \mt, s, st, t, \mt, s]$ from Example \ref{ex:rando} is not reduced. It is not reduced at index $1$ (the first instance of  $[\mt, s]$) because $K_1 =
\mt$ and $K_2 = s$. It is not reduced at index $5$ (the second instance of  $[\mt, s]$) because $\mi{p_5} = ts$ and $\mi{p_6} = t$. It is reduced at all other indices. \end{ex}

The existence of reduced expressions for arbitrary double cosets was proven by Williamson.

\begin{prop}\label{p has rex}
Each parabolic double coset in $W$ has a reduced expression.
\end{prop}

\begin{proof}
Our definition of a reduced expression agrees with that of a reduced right translation sequence in \cite[Chapter 2, Section 1.3]{GWthesis}. 
The desired result is proven in \cite[Chapter 2, Proposition 1.3.4]{GWthesis}. 
\end{proof}

The following recap may help the reader keep track of various definitions. An expression $I_\bullet$ may have many subordinate paths $p_\bullet$.

When $I_i \supset I_{i+1}$, then $p_{i+1}$ is one of the many $(J,I_{i+1})$-cosets inside the $(J,I_i)$-coset $p_i$, a non-deterministic choice. However, a forward path forces a single deterministic choice for $p_{i+1}$, the maximal such double coset. Conversely, $p_i$ is uniquely determined by $p_{i+1}$.

When $I_i \subset I_{i+1}$, then $p_{i+1}$ is uniquely determined by $p_i$, but $p_i$ is not determined by $p_{i+1}$. For a reduced path, \eqref{reducedminimal} says that $p_i$ is
uniquely determined by $p_{i+1}$, as the minimal double coset. So a reduced forward path is deterministic in both directions: $p_i$ and $p_{i+1}$ determine each other uniquely.
However, reducedness also has second requirement \eqref{sameK}. One use of \eqref{sameK} is to prevent infinite loops like $[s, \mt, s, \mt, s, \mt, \ldots]$, where the underlying
sets of the double cosets in the path $p_i$ are constant, see Example \ref{ex:smts}.

\begin{lem} Let $I_{\bullet}$ be a $J$-expression. If $I_i \subset I_{i+1}$ and $[p_i, p_{i+1}]$ is reduced then $p_i \subsetneq p_{i+1}$. \end{lem}

\begin{proof} Let $I_{i+1} = I_i \sqcup s$. Then $\mi{p_i} s \mi{p_i}\inv \notin W_J$, thus $\mi{p_i} s \notin W_J \mi{p_i}$. Using the length-preserving bijection \eqref{keybij}, one deduces that $\mi{p_i}s \notin p_i$. But $\mi{p_i} s \in p_{i+1} = W_J \mi{p_i} W_{I_{i+1}}$. \end{proof}

Thus part of being reduced is the statement (about the forward path) that ``when a double coset is supposed to get bigger, it actually does get bigger.'' However, this condition is strictly weaker than reducedness.

\begin{ex} \label{nonrex} Let $m_{st} = 3$ and consider the expression $[s,\mt,t,st]$ with 
	\[ p_{\bullet} = (\{e,s\}, \{e,s\}, \{e,s,t,st\}, W_{st}), \quad K_{\bullet} = (s, \mt, \mt, s). \]
This expression is not reduced at index $2$ because the redundancy increases, though $\{e,s,t,st\} \subsetneq W_{st}$. Note that the expression is reduced at index $0$, even though $\{e,s\} = \{e,s\}$; in a reduced expression, ``when a double coset is supposed to get smaller, it is permitted to stay the same size.''
\end{ex}

\begin{rem} \label{rmk:primal} If one places $\mt$ at the start of the expression in Example \ref{nonrex}, one obtains the non-reduced expression $[\mt,s,\mt,t,st]$ studied in Example \ref{ex:nonrexpart1}. These two non-reduced expressions are both non-reduced at the ``same'' index $[t,st]$ but for different reasons, one because of \eqref{sameK} and one because of \eqref{reducedminimal}. One can think that \eqref{reducedminimal} is the primal condition for reducedness, while \eqref{sameK} exists to ensure that reduced expressions are closed under precomposition, including $J$-expressions inside $\mt$-expressions. In Proposition \ref{embedding of expression} we will prove that a $J$-expression is reduced if and only if it is reduced after precomposition with $[\mt, s_1, s_1 s_2, \ldots, J]$ for some enumeration $J = \{s_1, \ldots, s_{|J|}\}$. \end{rem}

\begin{rem} Paths for $J$-expressions are strolls in the real hyperplane arrangement of the Coxeter group, staying within the fundamental domain for $W_J$ (equivalently, strolls in the $J$-singular dual Coxeter complex $\Cox_J$). Reduced expressions will be oriented strolls. Condition \eqref{sameK} can be interpreted as the statement that one never returns to a reflection hyperplane after leaving it. For more on this geometric intuition, see \S\ref{moregeometry}. \end{rem}

\section{New descriptions of reduced expressions} \label{sec-newdesc}

Studying a singular expression in terms of its forward path makes it harder to study contiguous subwords and other structures. For example, the forward path of
$[I_0, \ldots, I_d]$ is a sequence of $(I_0, I_i)$-cosets, while the forward path of a contiguous subword $[I_k, \ldots, I_{\ell}]$ is a sequence of $(I_k, I_i)$-cosets. Comparing the apples with the oranges takes some work. This is one of the main reasons to find alternatives for Williamson's definition of reduced expressions.

\subsection{Another criterion for one-step reducedness}

First, we give two easy lemmas.

\begin{lem}\label{lem:KK'}
Let $J, I, I' \subset S$ be finitary with $I\subset I'$. For a $(J,I)$-coset $p$ and $(J,I')$-coset $q$ with $p \subset q$, we have 
\begin{equation} K = J \cap \mi{p}I\mi{p}\inv \subset K' = J\cap \mi{q}I'\mi{q}\inv.\end{equation}
\end{lem}

\begin{proof}
By Lemma \ref{lem:onlyneedright} we have $\mi{p} W_{I'} = \mi{q} W_{I'}$. By \eqref{K=}, we have
\begin{equation} K = J\cap \mi{p}W_I\mi{p}\inv \subset J\cap \mi{p}W_{I'}\mi{p}\inv = J\cap \mi{q}W_{I'}\mi{q}\inv = K'.\end{equation}
\end{proof}

\begin{lem}
Whenever $I$ and $J$ are finitary with $I \subset J$ or $J \subset I$, one has
\begin{equation} \label{dottransfer} w_I . (w_{I \cap J}\inv w_J) = (w_I w_{I \cap J}\inv) . w_{I \cap J} . (w_{I \cap J}\inv w_J) = (w_I w_{I \cap J}\inv) . w_J. \end{equation}
\end{lem}

\begin{proof} This is obvious. Note also that either $(w_{I \cap J}\inv w_J)$ or $(w_I w_{I \cap J}\inv)$ is the identity. \end{proof}

Here is a useful condition on maximal elements which is equivalent to being reduced. Like in Lemma \ref{lem:onlyneedright}, when we change the parabolic subgroup on the right, we should be able to compare cosets using only right multiplication. Recall that Definition \ref{defn:pqred} stated what it means for $[p,q]$ to be reduced.

\begin{lem}\label{lem:reducedmaxs}
Let $J, I, I' \subset S$ be finitary, with $I \subset I'$. For a $(J,I)$-coset $p$ and $(J,I')$-coset $q$ with $p \subset q$, the following are equivalent.
\begin{enumerate}
\item $[p,q]$ is reduced.
\item $\ma{q} = \ma{p}(w_I\inv w_{I'})$.
\item $\ma{p}(w_I\inv w_{I'}) = \ma{p}.(w_I\inv w_{I'})$, that is, the lengths add in this product.
\end{enumerate}
\end{lem}

\begin{proof}
Let $K = J \cap \mi{p} I \mi{p}\inv$ and $K' = J \cap \mi{q} I' \mi{q}\inv$. By \eqref{maxformula} we have $\ma{q} = x' . \mi{q} . w_{I'}$ and $\ma{p} = x . \mi{p} . w_I$, where $x = w_J w_K\inv$ and $x' = w_J w_{K'}\inv$.

We first prove that (1) implies (2) and (3). If $[p,q]$ is reduced then $K = K'$ and $\mi{p} = \mi{q}$. Thus $x = x'$ and 
\[ \ma{q} = (x . \mi{p}) . w_{I'} = (\ma{p} w_I\inv) . w_{I'} = \ma{p} . (w_I\inv w_{I'}), \]
as desired\footnote{Note the crucial use of the mysterious condition \eqref{sameK}, which essentially implied the left multiplication played no significant role when comparing $\ma{p}$ and $\ma{q}$.}. The remainder of the proof will show that (2) implies (1), and (3) implies (2).



Suppose that $\ma{q} = \ma{p} w_I\inv w_{I'}$. Then $x' \mi{q} w_{I'} = x \mi{p} w_{I'}$, so $x'.\mi{q} = x.\mi{p}$. But $\mi{q} \le \mi{p}$ because of the containment of cosets, and $x' \le x$ because $K \subset K'$ by Lemma \ref{lem:KK'}. By comparing lengths, we see that $\mi{p} = \mi{q}$ and $x = x'$, thus $K = K'$. So $[p,q]$ is reduced. Thus (2) implies (1).

Now suppose that the lengths add in the expression $\ma{p}.(w_I\inv w_{I'})$. Since $I$ is in the right descent set of $\ma{p}$, we can use \eqref{dottransfer} to show that
\begin{equation} \ma{p}.(w_I\inv w_{I'}) = (\ma{p} w_I\inv) . w_I . (w_I\inv w_{I'}) = (\ma{p} w_I\inv) . w_I'. \end{equation}
Thus the lengths also add in $(\ma{p} w_I\inv) . w_I'$. By Lemma \ref{dotvsstar}, this means that $\ma{p} w_I\inv w_{I'} = (\ma{p} w_I\inv) * w_{I'}$. By Lemma \ref{lem:dontchangeleftdescent} this means that $J \subset \leftdes(\ma{p} w_I\inv w_{I'})$ and $I' \subset \rightdes(\ma{p} w_I\inv w_{I'})$. Since $\ma{p} w_I\inv w_{I'}$ is clearly in $q$, Lemma \ref{beingmaximal} implies that it is equal to $\ma{q}$. Thus (3) implies (2).
\end{proof}

We can summarize the above Lemma as
\begin{equation} \label{reducedmaxformula} [p,q] \text{ is reduced} \quad \iff \ma{q} = \ma{p} . (w_I\inv w_{I'}) =(\ma{p}w_I\inv).w_{I'}. \end{equation}

\begin{cor} \label{cor:updatemax} Let $I_{\bullet}$ be an expression and let $p_{\bullet}$ be its forward path. Then $I_{\bullet}$ is reduced if and only if for $0 \le i < d$ we have
\begin{equation} \label{updatemax} \ma{p_{i+1}} = \ma{p_i} \text{ when } I_{i+1} \subset I_i, \qquad \ma{p_{i+1}} = \ma{p_i} . w_{I_i}\inv w_{I_{i+1}} \text{ when } I_i \subset I_{i+1}. \end{equation}
\end{cor}

\subsection{Another criterion for reducedness}

We first give a reinterpretation of Corollary \ref{cor:starformap}, giving a formula for $\ma{p}$ associated to any reading of the expression.

\begin{prop} \label{thm:betterforward}
Let $I_{\bullet} = [I_0, \ldots, I_d]$ be a singular expression with $I_{\bullet} \expr p$. Define elements $y_i$ and $z_i$ in $W$ by the formulas
\begin{equation} \label{yzdef} y_i = w_{I_i} w_{I_i \cap I_{i+1}}\inv, \qquad z_i = w_{I_{i-1} \cap I_{i}}\inv w_{I_i}. \end{equation}
For $0 \le k \le d$, define $a_k \in W$ as
\begin{equation} \label{adef} a_k := y_0 * y_1 * \cdots * y_{k-1} * w_{I_k} * z_{k+1} * \cdots * z_{d-1} * z_d, \end{equation}
Then for any $0 \le k \le d$ we have
\begin{equation} \label{mapyz} \ma{p} = a_k. \end{equation}
\end{prop}

Note that when $I_{i-1} \supset I_i$ we have $z_i = e$, and when $I_i \subset I_{i+1}$ we have $y_i = e$.

\begin{proof} For any $1 \le i \le d$ we claim that
\begin{equation}  w_{I_{i-1}} * z_i = w_{I_{i-1}} * w_{I_i} = y_{i-1} * w_{I_i}. \end{equation}
This is easy, and left to the reader; see also \eqref{dottransfer}.

Let us show that \eqref{adef} agrees with \eqref{starformap}. Starting at $w_{I_k}$ and working outward, we have
\[ \cdots * y_{k-2} * y_{k-1} * w_{I_k} * z_{k+1} * z_{k+2} * \cdots \; = \; \cdots y_{k-2} * w_{I_{k-1}} * w_{I_k} * w_{I_{k-1}} * z_{k+2} * \cdots \; = \; \ldots = w_{I_0} * \cdots * w_{I_d}, \]
as desired. \end{proof}

When we use \eqref{mapyz} for a given value of $k$, we say that we are using the expression for $\ma{p}$ which is \emph{centered at $k$}. We now give an formulation of reducedness which is adapted to any reading of the expression.

\begin{thm} \label{thm:betterred} Continue the notation of Proposition~\ref{thm:betterforward}.
We say that the lengths add in \eqref{adef} if
\begin{equation} \label{eq:yzred} a_k = y_0 . y_1 . \cdots . y_{k-1} . w_{I_k} . z_{k+1} . \cdots . z_{d-1} . z_d. \end{equation}
The following are equivalent. \begin{enumerate}
\item $I_{\bullet}$ is reduced.
\item The lengths add in \eqref{adef} for some $0 \le k \le d$.
\item The lengths add in \eqref{adef} for all $0 \le k \le d$.
\end{enumerate}
\end{thm}

\begin{proof}  It is an immediate consequence of \eqref{dottransfer} that
\begin{equation} \label{baroof} w_{I_k} . z_{k+1} = y_k . w_{I_{k+1}}, \end{equation}
and in particular, these lengths always add. Since $a_k$ and $a_{k+1}$ have the same expression outside of \eqref{baroof}, this proves that $a_k = a_{k+1}$, and that the lengths add in $a_k$ if and only if they add in $a_{k+1}$. Hence it suffices to prove that $I_{\bullet}$ is reduced if and only if the lengths add in \eqref{eq:yzred} for $k=0$. We prove this statement by induction on the width $d$. When $d=0$, $I_{\bullet}$ is automatically reduced, and $\ma{p} = w_{I_0} = a_0$. Let us now fix an expression $I_{\bullet}$ of width $d$ and assume the result for all expressions of width $d-1$.

Suppose that $I_{\bullet}$ is reduced. Since the forward path of $[I_0, \ldots, I_{d-1}]$ agrees with the first $d-1$ steps of the forward path for $I_{\bullet}$, this subsequence is also reduced. By induction, the lengths add in the following expression for $\ma{p_{d-1}}$.
\begin{equation} \label{hellohankyung} \ma{p_{d-1}} = w_{I_0} . z_1 . \cdots . z_{d-1}.\end{equation}
By \eqref{updatemax}, the fact that $[p_{d-1},p_d]$ is reduced implies that
\begin{equation} \label{nicetoseeyou}  \ma{p_d} = \ma{p_{d-1}} . z_d. \end{equation}
Combining these, we deduce that $\ma{p_d} = a_0$ and the lengths add, as desired.

Conversely, suppose that the lengths add in the expression for $a_0$ in \eqref{eq:yzred}. By induction, this implies that $[I_0, \ldots, I_{d-1}]$ is reduced, and
\eqref{hellohankyung} holds. Moreover, the lengths add in $\ma{p_{d-1}} . z_d$, which by Lemma \ref{lem:reducedmaxs} implies that $[p_{d-1},p_d]$ is reduced. Thus $I_{\bullet}$ is
reduced. \end{proof}

\begin{rem} Recall the discussion from \S\ref{subsec-paths}, where we observed that the forward path is associated to the left reading of a singular expression. In contrast, one should think that the element $a_k$ is associated with any reading which starts at $I_k$ and moves outward contiguously. Because Theorem \ref{thm:betterred} is adapted to all readings of an expression, we think of it as an intrinsic definition of reducedness. It should be no surprise that $a_0$ played a special role in the proof, since it corresponds to the left reading, and is thus most obviously connected with the definition of reduced expressions. \end{rem}

\begin{lem} \label{lem:sameredasintro} The criterion for reducedness given in Theorem \ref{thm:betterred} is equivalent to the criterion for reducedness stated in the introduction in equation \eqref{introreddef}. \end{lem}

\begin{proof} Consider any singular expression of the form $[I_0 \subset I_1 \subset \ldots \subset I_d]$. Then it is easy to verify that $z_1 . z_2 . \cdots . z_d = w_{I_0}\inv
w_{I_d}$. Similarly, consider any singular expression of the form $[I_0 \supset I_1 \supset \ldots \supset I_d]$. It is easy to verify that $z_1 . z_2 . \cdots . z_d = e$. 

Now let $I_{\bullet}$ be arbitrary. Combining these observations, it is straightforward to match the expression for $a_0$ from \eqref{eq:yzred} with the expression for $\ma{p}$ in
\eqref{mapintroreddef}. That the lengths add in \eqref{eq:yzred} is equivalent to \eqref{hkdef}. \end{proof}

\subsection{Locality of (reduced) expressions}\label{ss:locality}

\begin{defn} If $I_\bullet$ and $I'_\bullet$ are two $J$-expressions with the same endpoint, we say that $I_\bullet$ and $I'_\bullet$ are \emph{equivalent}, and write $I_\bullet\expr I'_\bullet$. \end{defn}

We identify equivalence classes of expressions with double cosets, and freely use ``transitivity'' when examining statements like $I_\bullet \expr p \expr I'_\bullet$.

\begin{notation} We use interval notation for contiguous subwords of singular expressions: if $I_{\bullet} = [I_0,\ldots,I_d]$ is an expression, and $0 \le m \le n \le d$, then we write
\[ I_{[m,n]}:= [I_m, \ldots, I_n]. \]
\end{notation}

The following statement is the \emph{locality of expressions}.

\begin{prop} \label{prop:locality1}
Let $I_{\bullet} = [I_0,\ldots,I_d]$ be an expression, and fix $0 \le m \le n \le d$. If $H_{\bullet} = [I_m = H_0, H_1, \ldots, H_{r-1}, H_r = I_n]$ is another $I_m$-expression satisfying $H_{\bullet} \expr I_{[m,n]}$, then 
\begin{equation} I_{\bullet} \expr I'_{\bullet} := [I_0, \ldots, I_m, H_1, \ldots, H_{r-1}, I_n, \ldots, I_d]. \end{equation}
\end{prop}

\begin{proof} By Proposition \ref{prop:expressesinSC}, the double coset expressed by $I_{\bullet}$ is just a composition in $\SC$ of generators. If $[I_0, \ldots, I_m] \expr p$ and $[I_m, \ldots, I_n] \expr p'$ and $[I_n, \ldots, I_d] \expr p''$, then $I_{\bullet} \expr p * p' * p''$. But $H_{\bullet} \expr I_{[m,n]} \expr p'$ as well, so $I'_{\bullet} \expr p * p' * p''$. \end{proof}

The following statement is the \emph{locality of reduced expressions}.

\begin{prop}\label{prop:locality2}
Use the same setup and notation as in Proposition \ref{prop:locality1}. If $I_{\bullet}$ is reduced then $I_{[m,n]}$ is reduced. If $I_{\bullet}$ and $H_{\bullet}$ are reduced, then $I'_{\bullet}$ is reduced.\end{prop}

\begin{proof} Let $I_{\bullet}$ be reduced, and use the notation of Theorem \ref{thm:betterred}. We will compare $I_{\bullet}$ and $I_{[m,n]}$ by using \eqref{mapyz} centered at $m$. By Theorem \ref{thm:betterred}, the lengths add in \eqref{eq:yzred} for $k=m$. This product \eqref{eq:yzred} has a subword where the lengths add:
\begin{equation} \label{amton} a := w_{I_m} . z_{m+1} . \cdots . z_n. \end{equation}
Again by Theorem \ref{thm:betterred}, this implies that $I_{[m,n]}$ is reduced. Moreover, $I_{[m,n]} \expr p'$ where $\ma{p'} = a$.

Now let $z'_i$ and $y'_i$ and $a'_k$ be the elements of $W$ associated to the expression $I'_{\bullet}$, and $z''_i$, $y''_i$ be associated with $H_{\bullet}$. Note that $z''_i = z'_{m+i}$. If $H_{\bullet}$ is also reduced then we have another expression for $a$, as
\begin{equation} \label{afromH} a = w_{I_m} . z''_1 . \cdots . z''_r = w_{I_m} . z'_{m+1} . \cdots . z'_{m+r}. \end{equation}
Substituting \eqref{afromH} in place of \eqref{amton} inside the product formula \eqref{eq:yzred} for $a_m$, one obtains the product formula \eqref{eq:yzred} for $a'_m$. Thus the lengths add in the formula for $a'_m$, and Theorem \ref{thm:betterred} implies that $I'_{\bullet}$ is reduced. \end{proof}

\begin{rem} In a previous version of this manuscript, we proved this result directly from Williamson's definition of reduced expressions, by studying forward paths in depth. That
proof took an additional four pages. Again, we recall the discussion from the end of \S\ref{subsec-paths}: intrinsic constructions involving singular expressions (e.g.
multiplication in $\SC$, or the criterion for reducedness from Theorem \ref{thm:betterred}) are more useful for proving statements about subwords than constructions involving the
left reading (e.g. forward paths). \end{rem}

\subsection{The length function(s)} \label{ss:length}

The (weak left) Bruhat graph of a Coxeter system $(W,S)$ is a ranked poset, where each $w \in W$ is ranked by its length. A length
function is an easy way to detect whether a given edge goes up or down in the Bruhat order. In other words, $ws > w$ if and only if $\ell(ws) > \ell(w)$, otherwise $ws < w$. In
addition to the length function on elements of $W$, one also as a length function on (ordinary) expressions. The length of an expression for $w$ is always at least the length of
$w$, and equality holds if and only if the expression is reduced.

We transport these ideas to the setting of double cosets, defining the length of a double coset, and the length of a (singular) expression. Three crucial features our length function will satisfy are these: \begin{itemize}
\item The length of an expression for $p$ is always at least the length of $p$, with equality if and only if the expression is reduced.
\item The length of an expression will strictly increase with the addition of each index.
\item The length of an identity expression (i.e. a width zero expression) is zero.
\end{itemize}
These constraints imply that the length of an $(I,J)$-coset $p$ can not be determined by the underlying set of $p$, but also needs the data of $(I,J)$. For example, the $(s,s)$-coset $p = \{e,s\}$ is the identity element in $\SC$ of the object $s$, and has a length zero expression $[s]$. The $(s,\mt)$-coset $q = \{e,s\}$ has the longer reduced expression $[s,\mt]$, so we should have $\ell(q) > \ell(p)$ even though they have the same underlying set.

\begin{defn}\label{def:lengths}
Henceforth we write $\ell_W$ for the usual length function on $W$. If $J$ is finitary we write $\ell(J)$ for $\ell_W(w_J)$.

For $p\in \JIcosets$ we let
\begin{equation} \label{eq:lengthofp} \ell^+(p) = \ell_W(\ma{p}) - \ell(J), \quad \ell^-(p) = \ell_W(\mi{p}) + \ell(J) - \ell(J \cap \mi{p} I \mi{p}\inv), \quad \ell(p) =\ell^+(p) + \ell^-(p).\end{equation}
We refer to these as the \emph{$+$-length}, the \emph{$-$-length}, and the \emph{length} of $p$ respectively.

Given an expression $I_\bullet = [I_0,\ldots, I_n]$ with the forward path $p_\bullet$, we let $d_0^+ = d_0^- = 0$. We let
\begin{equation} d_i^+=
\begin{cases}
\ell(I_i) - \ell(I_{i-1}) & \text{ if } I_{i-1}\subset I_i \\
0 & \text{ if } I_{i-1}\supset I_i
\end{cases},\quad 
d_i^-=
\begin{cases}
0 & \text{ if } I_{i-1}\subset I_i \\
\ell(I_{i-1}) - \ell(I_i) & \text{ if } I_{i-1}\supset I_i
\end{cases},
\end{equation}
for $i\geq 1$.
We define the \emph{($\pm$-)lengths} of $I_\bullet$ as
\begin{equation}\ell^\pm (I_\bullet)=\sum_{i=0}^n d_i^\pm,\ \text{ and }   \ell(I_\bullet)=\ell^+(I_\bullet)+\ell^-(I_\bullet) .\end{equation}
\end{defn}

Be warned! If $w \in W$ and we let $p = \{w\}$ be an $(\emptyset, \emptyset)$-coset, then $\ell(p) = 2 \ell_W(w)$!

\begin{ex}\label{exleninc}
Let $I_\bullet = [\emptyset,\{s_1\},\{s_1,s_2\},\cdots ,\{s_1,\cdots s_d\}]$ with forward path $p_{\bullet}$. Then for all $0 \le i \le d$ we have
\begin{equation} \ell^+(p_i) = \ell(I_i) = \ell^+(I_{[0,i]}), \quad \ell^-(p_i) = 0 = \ell^-(I_{[0,i]}). \end{equation}
\end{ex}

\begin{ex}\label{exlenth}
Let $I_\bullet$ be an expression of the form $[\emptyset,s_1,\emptyset,s_2,\emptyset,\cdots,s_d,\emptyset]$.
Then $\ell^+(I_\bullet) = \ell^-(I_\bullet) = d$. Meanwhile, if $p$ is the endpoint of the expression, i.e. $p = \{w\}$ where $w = s_1 \cdots s_d$, then $\ell_W(w) \le d$ with equality if and only if $I_{\bullet}$ is reduced. Thus $\ell^+(I_{\bullet}) = \ell^+(p)$ if and only if $I_{\bullet}$ is reduced if and only if $\ell^-(I_{\bullet}) = \ell^-(p)$.
\end{ex}

\begin{prop}\label{p.rexlength}
Let $I_\bullet = [I_0, \ldots, I_n]$ be an expression and let $p$ be its endpoint. Then $\ell^+(I_{\bullet}) \ge \ell^+(p)$ and $\ell^-(I_{\bullet}) \ge \ell^-(p)$. Moreover, the following are equivalent.
\begin{itemize} \item $I_\bullet$ is reduced. \item $\ell^+(I_\bullet)=\ell^+(p)$ and $\ell^-(I_\bullet)=\ell^-(p)$. \item $\ell(I_\bullet)=\ell(p)$. \end{itemize}
\end{prop}

\begin{proof}
Let $p_{\bullet}$ be the forward path of $I_{\bullet}$. We will prove (for $1 \le i \le n$) that $d_i^+ \ge \ell^+(p_i) - \ell^+(p_{i-1})$ and $d_i^- \ge \ell^-(p_i) - \ell^-(p_{i-1})$, and that $[p_{i-1},p_i]$ is reduced if and only if both inequalities are equalities. Since $d_0^+ = \ell^+(p_0)$ and $d_0^- = \ell^-(p_0)$ by definition, we deduce that
\begin{equation} \sum_{i=0}^k d_i^+ \ge \ell^+(p_k), \quad \sum_{i=0}^k d_i^- \ge \ell^-(p_k), \quad \text{ with  both equalities if and only if } I_{[0,k]} \text{ is reduced}. \end{equation}
Setting $k = n$ we get the desired result.

Let $K = J \cap \mi{p_{i-1}} I_{i-1} \mi{p_{i-1}}\inv$ and $K' = J \cap \mi{p_i} I_i \mi{p_i}\inv$ for the rest of this proof.

When $I_{i-1} \supset I_i$ then $[p_{i-1},p_i]$ is always reduced. By definition of a forward path, $\ell^+(p_i) - \ell^+(p_{i-1}) = 0 = d_i^+$.  By \eqref{maxformula} we have
\begin{equation} (w_J w_K\inv) . \mi{p_{i-1}} . w_{I_{i-1}}  = \ma{p_{i-1}} = \ma{p_i} = (w_J w_{K'}\inv) . \mi{p_i} . w_{I_i}. \end{equation}
Thus 
\begin{equation} \ell^-(p_i) + \ell(I_i) = \ell_W(\mi{p_i}) + \ell(I_i) + \ell(J) - \ell(K') = \ell_W(\mi{p_{i-1}}) + \ell(I_{i-1}) + \ell(J) - \ell(K) = \ell^-(p_{i-1}) + \ell(I_{i-1}), \end{equation}
from which we conclude that
\begin{equation} d_i^- = \ell(I_{i-1}) - \ell(I_i) = \ell^-(p_i) - \ell^-(p_{i-1}). \end{equation}

If $I_{i-1} \subset I_i$ then Lemma \ref{lem:KK'} implies that $K \subset K'$. Since $p_{i-1} \subset p_i$ we have $\mi{p_i} \le \mi{p_{i-1}}$, with equality if and only if $p_{i-1}$ is minimal in $p_i$ (by definition). Thus
\begin{subequations}
\begin{equation} \ell_W(\mi{p_i}) - \ell_W(\mi{p_{i-1}}) \le 0, \text{ with equality if and only if $p_{i-1}$ is minimal in } p_i, \end{equation}
\begin{equation} \ell(K) - \ell(K') \le 0, \text{ with equality if and only if } K = K'. \end{equation}
Combining these we have
\begin{equation} d_i^- = 0 \ge (\ell_W(\mi{p_i}) - \ell_W(\mi{p_{i-1}})) + (\ell(K) - \ell(K')) = \ell^-(p_i) - \ell^-(p_{i-1}), \end{equation}
with equality if and only if $[p_{i-1}, p_i]$ is reduced.
\end{subequations}

Meanwhile, we know from Proposition \ref{prop:expressesinSC} that $\ma{p_i} = \ma{p_{i-1}} * w_{I_i}$, and since $I_{i-1}$ is in the right descent set of $\ma{p_{i-1}}$, we also have
\begin{equation} \ma{p_i} = \ma{p_{i-1}} * (w_{I_{i-1}}\inv w_{I_i}). \end{equation}
In particular, $\ell_W(\ma{p_i}) - \ell_W(\ma{p_{i-1}}) \le d_i^+ = \ell(I_i) - \ell(I_{i-1})$. Moreover, Corollary~\ref{cor:updatemax} states that $[p_{i-1}, p_i]$ is reduced if and only if $\ma{p_i} = \ma{p_{i-1}} . (w_{I_{i-1}}\inv w_{I_i})$, in which case $\ell_W(\ma{p_i}) - \ell_W(\ma{p_{i-1}}) = d_i^+$. 
\end{proof}

\section{Additional properties of reduced expressions} \label{sec-lotsoprops}

\subsection{Comparing forward paths with subexpressions}

\begin{prop}\label{prop:subexp}
Let  $I_\bullet = [I_0,\cdots, I_d]$ be a reduced expression, and pick $0 \le m \le n \le d$. The forward paths $p_\bullet$ of $I_\bullet$ and $q_\bullet$ of $I_{[m,n]}$ satisfy
\begin{equation}\label{maxsgeneral}
    \ma{p_i}=(\ma{p_m} w_{I_m}\inv). \ma{q_i} = \ma{p_m} . (w_{I_m}\inv \ma{q_i}).
\end{equation}
for $m\le i\le n$.
\end{prop}

\begin{proof} By Proposition \ref{prop:locality2}, $I_{[m,n]}$ is reduced, and the same for any other interval in $I_{\bullet}$. Pick $i$ such that $m \le i \le n$. Because $I_{[0,i]}$ is a reduced expression for $p_i$, \eqref{mapyz} centered at $m$ gives the formula
\begin{equation} \ma{p_i} = y_0 . \cdots . y_{m_1} . w_{I_m} . z_{m+1} . \cdots . z_i, \end{equation}
from which we deduce
\[ \ma{p_i} = \ma{p_m} . z_{m+1} . \cdots . z_i. \]
Meanwhile, \eqref{mapyz} for $I_{[m,i]}$ centered at $m$ gives
\[ \ma{q_i} = w_{I_m} . z_{m+1} . \cdots . z_i. \]
From these observations \eqref{maxsgeneral} follows easily. \end{proof}

\begin{rem} One should not expect nice formulas which compare the forward path of a non-reduced expression with its contiguous subwords. For example, when $S$ is finitary, take
any expression $I_{\bullet} = [I_0, \ldots, I_d]$ and extend it to the non-reduced expression $[[S, I_0, \ldots, I_d, S]]$. Every double coset in the forward path has the longest
element $w_S$ as its maximal element, so one loses effectively all the information in $I_{\bullet}$. \end{rem}

\subsection{Concatenation}

\begin{prop} \label{prop:concatreduced} Let $I_\bullet= [I_0,\ldots,I_m] \expr p$ and $K_\bullet=[K_0,\ldots,K_n] \expr q$ be reduced expressions with $K_0 = I_m =: J$. Let
\begin{equation} I_{\bullet} \circ K_{\bullet} := [I_0,\ldots,I_m=K_0,\ldots,K_n] \end{equation}
denote the concatenation. Then
\begin{equation} \label{yay} I_{\bullet} \circ K_{\bullet} \text{ is reduced } \iff \ma{p} w_J\inv \ma{q} = (\ma{p} w_J\inv) . \ma{q} = \ma{p} . (w_J\inv \ma{q})  \iff \ma{p} w_J\inv \ma{q} = \ma{p * q}. \end{equation}
\end{prop}

\begin{proof} Let us use Theorem \ref{thm:betterred} to prove the first implication. Write $z_i = w_{I_{i-1} \cap I_{i}}\inv w_{I_i}$ and $z'_i = w_{K_{i-1} \cap K_{i}}\inv
w_{K_i}$, and use \eqref{mapyz} centered at $k=0$ to describe both $\ma{p}$ and $\ma{q}$. Then \[ \ma{p} * (w_J\inv \ma{q}) = (w_{I_0} . z_1 . \cdots . z_m) * (w_J\inv w_J . z'_1 .
\cdots . z'_n) = (w_{I_0} . z_1 . \cdots . z_m) * (z'_1 . \cdots . z'_n), \] which is exactly the product in \eqref{adef} for $I_{\bullet} \circ K_{\bullet}$. Thus the lengths add
if and only if $I_{\bullet} \circ K_{\bullet}$ is reduced.

Now we prove the second implication.
By Proposition \ref{prop:expressesinSC}, we know that $\ma{p * q} = \ma{p} * \ma{q}$. Since $J$ is in the right descent set of $\ma{p}$ and in the left descent set of $\ma{q}$, we have
\[ (\ma{p} w_J\inv) * \ma{q} = \ma{p} * \ma{q} = \ma{p} * (w_J\inv \ma{q}). \]
Finally, Lemma \ref{dotvsstar} tells us that $\ma{p} w_J\inv \ma{q} = (\ma{p} w_J\inv) * \ma{q}$ if and only if $\ma{p} w_J\inv \ma{q} = (\ma{p} w_J\inv) . \ma{q}$.
\end{proof}

\subsection{Extension to the longest element} \label{ss:extension}

In ordinary Coxeter theory, a reduced expression for $w \in W$ can be extended to a longer reduced expression if and only if $w$ is not the longest element. In particular, a reduced expression for $w$ can always be extended to a reduced expression for the longest element. We prove the analogous statements here. 

\begin{prop}\label{p to w0}
Let $(W,S)$ be a finite Coxeter group with longest element $w_S$. Let $I_{\bullet} = [I_0, \ldots, I_m = J] \expr p$ be reduced. Then there exists an expression $K_{\bullet} = [J = K_0, \ldots, K_n]$ such that $I_{\bullet} \circ K_{\bullet}$ is a reduced expression for the $(I_0, K_n)$-coset containing $w_S$.  In fact, we can find such a reduced expression where $K_n = K'$ if and only if $K' \subset K := \rightdes(w_J \ma{p}\inv w_S)$. \end{prop}

\begin{proof} Suppose that $K_{\bullet} \expr q$ and $p * q$ has maximal element $w_S$. By Proposition \ref{prop:concatreduced}, this happens if and only if $\ma{p} w_J\inv .
\ma{q} = w_S$. In particular $\ma{q} = w := w_J \ma{p}\inv w_S$. Let $K = \rightdes(w)$. If $K' \not\subset K$, then $w$ can not be the maximal element of any $(I_0,K)$
coset, so no such $q$ exists.

Now suppose $K' \subset \rightdes(w)$. We claim that $w$ is maximal in some $(J,K')$-coset $q$. By \eqref{beingmaximal}, it is equivalent to prove that $J \subset \leftdes(w)$ and $K' \subset \rightdes(w)$. The latter holds by assumption. We prove that $J \subset \leftdes(w)$ using two standard facts about descent sets.

The first fact is that, for any $x \in W$ one has
\begin{equation} \leftdes(x) \sqcup \leftdes(x w_S) = S. \end{equation}
In other words, the left descent sets of $x$ and $x w_S$ are complementary. This is because multiplication by $w_S$ will act on inversion sets by taking the complement. Thus $\leftdes(w)$ is complementary to $\leftdes(w_J \ma{p}\inv)$.

The second fact is that $\leftdes(x) = \rightdes(x\inv)$. Thus $\leftdes(w_J \ma{p}\inv) = \rightdes(\ma{p} w_J\inv)$. Since $\ma{p}$ has $J$ in its right descent set, $\ma{p}
w_J\inv$ is minimal in its right $W_J$ coset, so $J \cap \rightdes(\ma{p} w_J\inv) = \mt$. Thus the complement $\leftdes(w)$ contains $J$.

So, let $q$ be the $(J,K')$-coset such that $\ma{q} = w$. By Proposition \ref{p has rex}, there is some reduced expression $K_{\bullet}$ for $q$. It may help to write $w = (\ma{p} w_J\inv)\inv w_S$, in that this accurately reflects the fact that lengths add in the expression $(\ma{p} w_J\inv) . w = w_S$. Because the lengths add, \eqref{yay} states that the concatenation is reduced, and that $p * q$ has maximal element $w_S$.
\end{proof}

\begin{cor} Let $(W,S)$, $I_{\bullet}$, $p$, and $K$ be as in Proposition \ref{p to w0}. Then there is a reduced expression extending $I_{\bullet}$ which passes through the
parabolic subgroup $S$ if and only if $\ma{p} = w_J$. \end{cor}

\begin{proof} If there is a reduced expression passing through $S$, then by truncating (see Proposition \ref{prop:locality2}) we obtain a reduced expression ending in $S$. By necessity, the maximal element of its endpoint is $w_S$. Proposition \ref{p to w0} says that such a reduced expression exists if and only if $K = S$. But then $w_J \ma{p}\inv w_S = w_S$, or equivalently $\ma{p} = w_J$. \end{proof}

\begin{ex} Suppose that $S = \{s,t\}$ and $m_{st} = 3$. The reduced expression $[\mt, s, \mt] \expr \{s\}$ can be extended to $[\mt, s, \mt, t, \mt, s] \expr \{st,sts\}$ and further to $[\mt,s,\mt,t,\mt,s,\mt] \expr \{sts\}$. However, $[\mt,s,\mt]$ can not be extended to any reduced expression containing the finitary subset $st$. Meanwhile, $[\mt,s] \expr \{e,s\}$ can be extended to the reduced expression $[\mt,s,st]$. \end{ex}

\subsection{Reversing the word} \label{ss:reversing}

In this section we use the fact that $(W,*,S)$ has an antiinvolution which sends $w \mapsto w\inv$. The notation $w\inv$ represents the inverse element in the group $W$; the Coxeter monoid $(W,*,S)$ does not have inverses.

Given an expression $I_\bullet = [I_0,\cdots,I_d]$, the sequence $[I_d,\cdots,I_0]$ is an expression and is denoted by $I\inv_\bullet$. If $p$ is a $(J,I)$-coset, then the unique $(I,J)$-coset with maximal element $\ma{p}\inv$ agrees, as a set, with $p\inv := \{w\inv\ |\ w\in p\}$. 

\begin{prop}\label{invrex}
Let $I_\bullet \expr p$. Then $I\inv_\bullet\expr p\inv$, and $I\inv_\bullet$ is reduced if and only if $I_\bullet$ is reduced. 
\end{prop}

\begin{proof}
From the formula \eqref{starformap} it is clear that $I\inv_{\bullet} \expr p\inv$. Applying the antiinvolution $w \mapsto w\inv$ to the expression $a_0$ for $\ma{p}$ in \eqref{adef} (centered at the first index), we get the expression for $\ma{p\inv}$ (centered at the last index). The lengths add in $a_0$ if and only if they add after applying the antiinvolution, so Theorem \ref{thm:betterred} implies the result. \end{proof}

\subsection{Addable elements}

Let $I_{\bullet}$ be a reduced expression. We call a simple reflection $s \notin I$ \emph{addable} if $I_{\bullet} \circ [I,Is]$ is still reduced. The next proposition states
that if both $s$ and $t$ are addable individually, then they are addable together.

\begin{prop} \label{prop:addable} Suppose that $I_{\bullet} := [J = I_0, I_1, \ldots, I_d = I]$ is a reduced expression, and $s, t \notin I$ with $Ist$ finitary. Then the following are equivalent. \begin{enumerate} \item $I_{\bullet} \circ [I, Is, Ist]$ is reduced.
	\item $I_{\bullet} \circ [I,Is]$ is reduced and $I_{\bullet} \circ [I,It]$ is reduced. \end{enumerate}
\end{prop}

\begin{proof} Suppose that $I_{\bullet} \circ [I,Is,Ist]$ is reduced. By Proposition \ref{prop:locality2}, $I_{\bullet} \circ [I,Is]$ is reduced. Note that $[I,Is,Ist] \expr [I,It,Ist]$ and both are reduced expressions. Then by Proposition \ref{prop:locality2}, both $I_{\bullet} \circ [I,It,Ist]$ and $I_{\bullet} \circ [I,It]$ are reduced as well.
	
Let $I_{\bullet} \expr p$. Now suppose that $I_{\bullet} \circ [I,Is] \expr q$ and $I_{\bullet} \circ [I,It] \expr r$ are both reduced. Note that $\mi{p} = \mi{q} = \mi{r}$. Let
$m$ be the $(J,Ist)$-coset containing $\mi{p}$. By \eqref{beingminimal} applied to $q$ and $r$, $\leftdes(\mi{p}) \cap J = \mt$ and $\rightdes(\mi{p}) \cap Ist = \mt$.
Thus (again by \eqref{beingminimal}), $\mi{p} = \mi{m}$. By \eqref{sameK} applied to $I_{\bullet} \circ [I,Is]$ and $I_{\bullet} \circ [I,It]$, respectively, we see that $s \notin \mi{q}^{-1} J \mi{q}$ and $t \notin \mi{r}^{-1} J \mi{r}$. Thus the redundancy of $m$ agrees with the redundancy of $p$ (and $q$ and $r$). Hence $[q,m]$ is reduced, so $I_{\bullet} \circ [I,Is,Ist]$ is reduced. \end{proof}

\subsection{Going in}

Let $I_{\bullet} = [I_0, \ldots, I_d]$ be an expression. We think of \emph{going in} as the statement that, if $I_0 \subset I_1$ (resp. $I_d \subset I_{d-1}$), then removing $I_0$ (resp. $I_d$) does not change the maximal elements in the forward path, nor does it alter whether the expression is reduced. We state this in two separate lemmas, mostly for notational reasons.

\begin{lem}\label{in end}
Let $I_\bullet = [I_0,\ldots, I_d]$ be an expression such that $I_d\subset I_{d-1}$. Let $p_{\bullet}$ be its forward path. Let $I_{[0,d-1]} := [I_0, \ldots, I_{d-1}]$, with forward path $q_{\bullet} = [q_0, \ldots, q_{d-1}]$. Then $p_i = q_i$ for all $0 \le i \le d-1$, and $\ma{p_d} = \ma{q_{d-1}}$. Moreover, $I_{[0,d-1]}$ is reduced if and only if $I_{\bullet}$ is reduced.
\end{lem}

\begin{proof} The statement that $p_i = q_i$ is trivial, because paths are read from left to right. Because $p_{\bullet}$ is a forward path we know that $\ma{p_d} = \ma{p_{d-1}}$. The `moreover' statement holds because there is no condition on reducedness of $[I_{d-1}, I_d]$ when $I_{d-1} \supset I_d$. \end{proof}

\begin{lem}\label{in start}
Let $I_\bullet = [I_0,\ldots, I_d]$ be an expression such that $I_0\subset I_1$. Let $p_{\bullet}$ be its forward path. Let $I_{[1,d]} := [I_1, \ldots, I_d]$, with forward path $q_{\bullet} = [q_1, \ldots, q_d]$. Then
\begin{equation}\label{maxs}
    \ma{p_i} = \ma{q_i} 
\end{equation}
for each $1\leq i\leq d$. In other words, $p_i$ is the maximal $(I_0, I_i)$-coset inside the $(I_1,I_i)$-coset $q_i$, for all $1 \le i \le d$. Moreover,  $I_{[1,d]}$ is reduced if and only if $I_\bullet$ is reduced.
\end{lem}

\begin{proof} This follows immediately from Proposition~\ref{thm:betterforward} and Theorem~\ref{thm:betterred} because of the simple observation that $y_0 = e$. \end{proof}

\begin{rem} One could have also proven Lemma \ref{in end} using the observation that $z_d = e$, using a rocket launcher to pop a balloon. However, proving Lemma \ref{in start} directly from Definition \ref{def:reduced} is more difficult. \end{rem}
	
By combining Lemmas \ref{in start} and \ref{in end}, we can effectively embed any (reduced) expression inside a (reduced) expression from $\mt$ to $\mt$, without changing the
maximal elements of the cosets in the forward path.

\begin{defn} Let $I_{\bullet} = [J=I_0, I_1, \ldots, I_d]$ be a $J$-expression, and fix an enumeration $J = \{s_1, \ldots, s_m\}$ of $J$. Define the $\mt$-expression $\iota_J I_{\bullet}$ to be
\begin{equation} \iota_J I_{\bullet} := [\mt, s_1, s_1 s_2, \ldots, J, I_1, \ldots, I_d]. \end{equation}
We think of $\iota_J$ as a function from $J$-expressions to $\mt$-expressions. Note that $\iota_J$ depends on the enumeration of $J$, though this is suppressed from the notation.

Similarly, let $I_{\bullet} = [I_0, I_1, \ldots, I_d = L]$ be an expression ending in $L$, and fix an enumeration $L = \{t_1, \ldots, t_n\}$ of $L$. Define an expression $\kappa_L I_{\bullet}$ ending in $\mt$ to be
\begin{equation} \kappa_L I_{\bullet} := [I_0, \ldots, I_{d-1},  L, \ldots, t_1 t_2, t_1, \mt]. \end{equation}
\end{defn}

\begin{prop}\label{embedding of expression}
Let $J, L\subset S$ be finitary, and pick enumerations of $J$ and $L$. For any expression $I_\bullet = [J = I_0, I_1, \ldots, I_{d-1}, I_d = L]$, one has
\begin{equation} I_{\bullet} \text{ reduced } \iff \iota_J I_{\bullet} \text{ reduced } \iff \kappa_L I_{\bullet} \text{ reduced } \iff \iota_J \kappa_L I_{\bullet} \text{ reduced}. \end{equation}
Furthermore, the maximal elements in the forward paths of these four expressions agree where they overlap, regardless of whether they are reduced or not.
Finally, the endpoints of all four expressions all have the same maximal element (while being four different kinds of double cosets).\end{prop}

\begin{proof}
Apply Lemma~\ref{in start} $|J|$ times and/or Lemma~\ref{in end} $|L|$ times.
\end{proof}

The upshot of Proposition \ref{embedding of expression} is that the $J$-singular Coxeter complex $\Cox_J$ will embed into the $\mt$-singular Coxeter complex $\Cox_{\mt}$, see
Proposition~\ref{singcoxembeds}.

\begin{rem} Let $I_{\bullet}$ be a $J$-expression, and let $\iota_J I_{\bullet}$ and $\tilde{\iota}_J I_{\bullet}$ be $\mt$-expressions arising from two different enumerations of
$J$. Then $\iota_J I_{\bullet}$ and $\tilde{\iota}_J I_{\bullet}$ are related by the up-up and down-down (braid) relations. \end{rem}
	
In an earlier proof of Lemma \ref{in start} we used the following result. Currently we do not use the result in this paper, but it seems useful so we record it for posterity.

\begin{lem}\label{lastlemma}
Let $I, I', J,J'\subset S$ be finitary, such that $I \subset I'$ and $J \subset J'$. Let $q$ be a $(J',I)$-coset. Let $q'$ be the $(J',I')$-coset containing $q$. Now let $r$ be the maximal $(J,I)$-coset contained in $q$, and let $r'$ be the maximal $(J,I')$-coset contained in $q'$.  Then $r \subset r'$. \end{lem}

\begin{proof}
To restate the lemma, set $r = W_J \ma{q} W_I$ and $r' = W_J \ma{q'} W_{I'}$. To show that $r \subset r'$, it is equivalent to show that 
\[ W_J \ma{r} W_{I'} = W_J \ma{r'} W_{I'}. \] Since $\ma{r} = \ma{q}$ and $\ma{r'} = \ma{q'}$ by assumption, we need only show that
\begin{equation} W_J \ma{q} W_{I'} = W_J \ma{q'} W_{I'}. \end{equation} 	

From Lemma \ref{lem:onlyneedright} we know that $\mi{q} W_{I'} = \mi{q'} W_{I'}$ is an equality of cosets in $W / W_{I'}$, and similarly $\ma{q} W_{I'} = \ma{q'} W_{I'}$. By Lemma \ref{lem:KK'} applied to $q \subset q'$ we have $K \subset K'$, where
\begin{equation}\label{KK'is}
K := J'\cap \mi{q}I\mi{q}\inv, \qquad K' := J'\cap \mi{q'}I'\mi{q'}\inv. \end{equation}
Now \eqref{maxformula} (this time using the second formulation) gives that $\ma{q} = w_{J'} . \mi{q} . y$ and $\ma{q'} = w_{J'} . \mi{q'} . y'$ for some $y \in W_I$ and $y' \in W_{I'}$. We now see that
\begin{equation}
W_J\ma{q}W_{I'} = W_J w_{J'} \mi{q} y W_{I'}  = W_J w_{J'} \mi{q} W_{I'} = W_J w_{J'} \mi{q'} W_{I'} = W_J w_{J'} \mi{q'} y' W_{I'} = W_J \ma{q'} W_{I'}, \end{equation}
as desired.
\end{proof}

\subsection{Going out}

Let $I_{\bullet} = [I_0, \ldots, I_d]$ be an expression. \emph{Going out} explores the relationship between $I_{\bullet}$ and $I_{[1,d]}$ when $I_0 \supset I_1$ (resp. between $I_{\bullet}$ and $I_{[0,d-1]}$ when $I_{d-1} \subset I_d$). We state this in two separate lemmas, mostly for notational reasons.

\begin{lem}\label{out start}
Let $I_\bullet = [I_0,\cdots, I_d]$ be a \emph{reduced} expression such that $I_0\supset I_1$. Let $p_\bullet$ be its forward path and $K_{\bullet}$ its (left) redundancy sequence. Let $I_{[1,d]} := [I_1,\cdots,I_d]$, with forward path $[q_1,\cdots,q_d]$ and (left) redundancy sequence $[L_1, \ldots, L_d]$. 
Then the expression $I_{[1,d]} = [I_1,\cdots,I_d]$ is reduced. Moreover,
\begin{subequations} \label{outstarteqs}
\begin{equation}\label{mins}
    \mi{p_i} = \mi{q_i},
\end{equation}
\begin{equation}\label{shiftmax}
    \ma{p_i} = (w_{I_0}w_{I_1}\inv).\ma{q_i}, 
\end{equation}
and 
\begin{equation} \label{sameredundancy}
	K_i = L_i
\end{equation}
\end{subequations}
for each $1\leq i\leq d$.
\end{lem}

There are major differences between Lemma \ref{out start} and Lemma \ref{in start}. We only have a one-sided implication for reducedness; $I_{[1,d]}$ being reduced does not imply
that $I_{\bullet}$ is reduced. Also, the formulas \eqref{mins} and \eqref{shiftmax} only hold when $I_{\bullet}$ is reduced, whereas in Lemma \ref{in start} we gave a formula
\eqref{maxs} which held regardless of whether $I_{\bullet}$ was reduced or not. Before giving the proof, we give a counterexample when $I_{\bullet}$ is not reduced.

\begin{ex} Fix $s \in S$. The expression $I_{\bullet} = [s, \mt, s, \mt]$ is not reduced, with forward path $(p_0 = \{e,s\}, \{e,s\}, \{e,s\}, p_3 = \{e,s\})$. The expression $I_{[1,3]} = [\mt, s, \mt]$ is reduced, with forward path $(q_1 = \{e\}, \{e, s\}, q_3 = \{s\})$. Note that $\mi{p_3} = e \ne s = \mi{q_3}$, and $\ma{p_2} = s \ne s s = w_{I_0} w_{I_1}\inv \ma{q_2}$. Also note that $K_2 = s$ while $L_2 = \mt$. \end{ex}

\begin{proof}[Proof of Lemma \ref{out start}] That $I_{[1,d]}$ is reduced follows from Proposition \ref{prop:locality2}, and \eqref{shiftmax} follows from Theorem \ref{thm:betterred} and the observation that $y_0 = w_{I_0} w_{I_1}\inv$.

By \eqref{maxformula}, we know that
\[ \ma{p_i} = (w_{I_0} w_{K_i}\inv) . \mi{p_i} . w_{I_i}, \qquad \ma{q_i} = (w_{I_1} w_{L_i}\inv) . \mi{q_i} . w_{I_i}. \]
Using \eqref{shiftmax} and the fact that $w_{I_0} = (w_{I_0} w_{I_1}\inv) . w_{I_1}$, we see that
\[ (w_{I_0} w_{K_i}\inv) . \mi{p_i} . w_{I_i} = \ma{p_i} = (w_{I_0}w_{I_1}\inv).\ma{q_i} = (w_{I_0} w_{L_i}\inv) . \mi{q_i} . w_{I_i}. \]
Dividing both sides by $w_{I_i}$ on the right, we get
\begin{equation} \label{foobar1} (w_{I_0} w_{K_i}\inv). \mi{p_i} = (w_{I_0}w_{L_i}\inv). \mi{q_i}. \end{equation}

Now we use the same method as in Lemma \ref{lem:reducedmaxs}. Since $\mi{p_i}$ and $\mi{q_i}$ are in the same left $W_{I_0}$ coset we have $q_i \subset p_i$, and thus $\mi{p_i} \le \mi{q_i}$ in the Bruhat order. By Lemma \ref{lem:KK'} (applied on the left rather than the right) we have $L_i \subset K_i$, and thus $(w_{I_0} w_{K_i}\inv) \le (w_{I_0} w_{L_i}\inv)$ in the Bruhat order.  But then the equality in \eqref{foobar1} implies that $\mi{p_i} = \mi{q_i}$ and $L_i = K_i$, as desired. \end{proof}

As before, the analogous statement for the tail of an expression is trivial.

\begin{lem} \label{out end}
Let $I_\bullet = [I_0,\cdots, I_d]$ be a reduced expression such that $I_{d-1} \subset I_d$. Then the expression $I_{[0,d-1]} = [I_0,\cdots,I_{d-1}]$ is reduced. It has the same forward path and the same redundancy sequence, for indices between $0$ and $d-1$. \end{lem}

\begin{proof} This is trivial, because paths are constructed using the left reading of an expression. \end{proof}
	
Once again, the converse of Lemma \ref{out end} fails. For example, $I_{\bullet} = [\mt, s, \mt, s]$ is not reduced but $I_{[0,2]} = [\mt,s,\mt]$ is reduced.

\subsection{Our favorite reduced expressions: the high road} \label{subsec:getbig}

We now investigate what can be said about finding a reduced expression for a particular double coset $p$. We first explore reduced expressions which begin by adding as many simple reflections as possible.

\begin{lem} Let $I_{\bullet}$ be a (reduced) expression with forward path $p_{\bullet}$. Then $\leftdes(\ma{p_i})$ is weakly increasing. \end{lem}

\begin{proof} By \eqref{updatemax}, $\ma{p_{i+1}} = \ma{p_i} * x$ for some $x$. By Lemma \ref{lem:dontchangeleftdescent}, $\leftdes(\ma{p_i}) \subset \leftdes(\ma{p_{i+1}})$.
\end{proof}

Now let $p$ be any $(J,I)$-coset. We know that $J \subset \leftdes(\ma{p})$, but this left descent set can be larger still.

\begin{cor} \label{cor:nottoobig} Let $p$ be a $(J,I)$-coset and $K$ be finitary with $J \subset K$. If there is a reduced expression for $p$ of the form $[[J,K]] \circ M_{\bullet}$ for any $M_{\bullet}$, then $K \subset \leftdes(\ma{p})$. \end{cor}

\begin{proof} Note that $[[J,K]]$ expresses a double coset with maximal element $w_K$. Thus by the previous lemma, the left descent set of $p$ must contain $K$. \end{proof}

So there is no reduced expression for $p$ which begins by adding any simple reflection not in $\leftdes(\ma{p})$. By the next result, there is a reduced expression which does begin with $[[J,\leftdes(\ma{p})]]$.

\begin{prop} \label{thm:favoritebig} Let $p$ be a $(J,I)$-coset, with $M = \leftdes(\ma{p})$ and $N = \rightdes(\ma{p})$. Let $q$ be the $(M,N)$-coset with $\ma{q} = \ma{p}$. If
$I_{\bullet}$ is any reduced expression for $q$, then $[[J, M]] \circ I_{\bullet} \circ [[N,I]]$ is a reduced expression for $p$. \end{prop}

\begin{proof} First, observe that there does exist an $(M,N)$-coset $q$ with $\ma{q} = \ma{p}$, since $M \subset \leftdes(\ma{p})$ and $N \subset \rightdes(\ma{p})$. Now we apply
Theorem \ref{thm:betterred} to deduce that $[[J, M]] \circ I_{\bullet} \circ [[N,I]]$ is reduced and expresses $p$. Every $y_i$ inside $[[J,M]]$ is equal to $e$, and every $z_i$
inside $[[N,I]]$ is equal to $e$, so this is straightforward. \end{proof}

This theorem states that $p$ has a reduced expression which gets as large as possible, maintains the same left descent set until it reaches the correct maximal element, and then
descends. Combining Theorem \ref{thm:favoritebig} with Corollary \ref{cor:nottoobig}, we see that the simple reflections which can be added first in a reduced expression for $p$ are
precisely $\leftdes(\ma{p}) \setminus J$.

\subsection{Our favorite reduced expressions: the low road} \label{subsec:getsmall}

Now we explore reduced expressions which begin by subtracting as many simple reflections as possible. Recall the definition of the (left) redundancy sequence from Definition
\ref{subex}.

\begin{lem} Let $I_{\bullet}$ be a reduced expression with left redundancy sequence $K_{\bullet}$. Then $K_{\bullet}$ is weakly decreasing. \end{lem}
	
\begin{proof} If $I_i \subset I_{i+1}$ then $K_i = K_{i+1}$, by definition of reduced, see \eqref{sameK}. If $I_i \supset I_{i+1}$ then $K_i \supset K_{i+1}$ (not necessarily strictly), which follows from
Lemma \ref{lem:KK'}. \end{proof}

\begin{cor} \label{cor:nottoosmall} Let $I_{\bullet} = [I_0, I_1, \ldots, I_d]$ be a reduced expression, with left redundancy sequence $K_{\bullet}$. If $I_0 \supset I_1 = I_0 \setminus s$, then $s \notin K_d$. \end{cor}

\begin{proof} If $I_0 \supset I_1$ then $K_1 = I_1$. Since $K_d \subset K_1$, we see that $s \notin K_d$. \end{proof}


\begin{defn} Let $p$ be a $(J,I)$-coset. We say that $p$ has \emph{full redundancy} if $J = \mi{p} I \mi{p}\inv$. \end{defn}

\begin{lem} Let $p$ be a $(J,I)$-coset with full redundancy. Let $I_{\bullet}$ be a reduced expression for $p$. Then the redundancy sequence (see Definition \ref{subex}) of $I_{\bullet}$ is constant. \end{lem}

\begin{proof} If $p$ has full redundancy and $I_{\bullet}$ is a reduced expression for $p$ of width $d$, then $K_0 = I = K_d$. Since the redundancy sequence is weakly decreasing,
it must be constant. \end{proof}

Now we introduce another double coset related to a given double coset $p$, which in some sense contains the ``interesting part'' of $p$.

\begin{defn}
Let $p$ be a $(J,I)$-coset, with left redundancy $K = J \cap \mi{p} I \mi{p}\inv$ and right redundancy $L = I \cap \mi{p}\inv J \mi{p}$. The \emph{core} of $p$ is the $(K,L)$-coset $p^{\core}$ with minimal element $\mi{p^{\core}} = \mi{p}$. \end{defn}

\begin{lem} The core is well-defined. The core has full redundancy. \end{lem}

\begin{proof} We know that $\leftdes(\mi{p}) \cap J = \mt$, which implies that $\leftdes(\mi{p}) \cap K = \mt$. Similar statements hold for the right descent set. So $\mi{p}$ is
the minimal element of its $(K,L)$-coset, by \eqref{beingminimal}. Since $\mi{p} L \mi{p}\inv = K$, we see that $K$ is the left redundancy of $p^{\core}$ as well, so it has full redundancy. \end{proof}

\begin{prop} \label{thm:favoritesmall} Let $p$ be a $(J,I)$-coset, with left redundancy $K = J \cap \mi{p} I \mi{p}\inv$ and right redundancy $L = I \cap \mi{p}\inv J \mi{p}$. Let $[[J,K]]$ denote any expression which removes one reflection at a time to get from $J$ to $K$. Let $[[L,I]]$ denote any expression which adds one reflection at a time to get from $L$ to $I$. If $M_{\bullet}$ is any reduced expression for $p^{\core}$, then
\begin{equation} \label{favorite} [[J,K]] \circ M_{\bullet} \circ [[L,I]] \end{equation}
is a reduced expression for $p$. \end{prop}

\begin{proof} By \eqref{maxformula}, we know that $\ma{p^{\core}} = \mi{p^{\core}} w_L = \mi{p} w_L$. By \eqref{maxformula} again, we have
\begin{equation} \ma{p} = (w_J w_K\inv) . (\mi{p} w_L) . (w_L\inv w_I) = (w_J w_K\inv) . \ma{p^{\core}} . (w_L\inv w_I). \end{equation}
Now the result follows immediately from Proposition \ref{prop:concatreduced}. \end{proof}

This theorem states that $p$ has a reduced expression which gets as small as possible, maintains the same left redundancy until it reaches the correct minimal element, and then
ascends. Combining Theorem \ref{thm:favoritesmall} with Corollary \ref{cor:nottoosmall}, we see that the simple reflections which can be removed first in a reduced expression for
$p$ are precisely $J \setminus \mi{p} I \mi{p}\inv$.

\subsection{Products of Coxeter groups} \label{subsec:products}

Every Coxeter group is a product of irreducible Coxeter groups. Thankfully, the properties of reduced expressions (as well as braid relations, etcetera) really depend only on each irreducible factor.

\begin{defn} \label{def:plusJ} Let $(W,S)$ and $(W',S')$ be two Coxeter groups. Let $I_{\bullet} = [I_0, I_1, \ldots, I_d]$ be any expression where $I_i \subset S$ for all $0 \le i
\le d$. Let $J \subset S'$ be finitary. We denote by $I_{\bullet}^{(+J)}$ the corresponding expression \begin{equation} I_{\bullet}^{(+J)} = [I_0 \sqcup J, I_1 \sqcup J, \ldots,
I_d \sqcup J] \end{equation} inside the Coxeter group $(W \times W', S \sqcup S')$. Similarly, let $p$ be an $(I_0, I_d)$-coset, where $I_0, I_d \subset S$. The $(I_0 \sqcup J, I_d
\sqcup J)$-coset inside $W \times W'$ containing $\ma{p} . w_J$ will be denoted $p^{(+J)}$. \end{defn}
	
\begin{prop} \label{prop:reduceproduct} Continue the notation of Definition \ref{def:plusJ}. The following properties hold. \begin{enumerate}
\item $\ma{p^{(+J)}} = \ma{p} . w_J$ and $\mi{p^{(+J)}} = \mi{p}$.
\item $\ell^+(p^{(+J)}) = \ell^+(p)$ and $\ell^-(p^{(+J)}) = \ell^-(p)$.
\item $\ell^+(I_{\bullet}^{(+J)}) = \ell^+(I_{\bullet})$ and $\ell^-(I_{\bullet}^{(+J)}) = \ell^-(I_{\bullet})$.
\item $I_{\bullet} \expr p$ if and only if $I_{\bullet}^{(+J)} \expr p^{(+J)}$.
\item $I_{\bullet}$ is reduced if and only if $I_{\bullet}^{(+J)}$ is reduced.
\end{enumerate}
\end{prop}

\begin{proof} The basic point is that $\ell(I \sqcup J) = \ell(I) + \ell(J)$, for any $I \subset S$. One can also check that the left redundancy of $p^{(+J)}$ is just the disjoint
union of $J$ and the left redundancy of $p$. This immediately implies all the statements above about lengths. The statement about reduced expressions follows from Proposition
\ref{p.rexlength}. The statement that $I_{\bullet} \expr p$ if and only if $I_{\bullet}^{(+J)} \expr p^{(+J)}$ is straightforward from Proposition~\ref{thm:betterforward} and
\eqref{adef} in particular: the values of $y_i$ and $z_i$ do not differ between $I_{\bullet}$ and $I_{\bullet}^{(+J)}$, the only difference is $w_{I_k}$ versus $w_{I_k \sqcup J} =
w_{I_k} . w_J$. \end{proof}

We think of Proposition \ref{prop:reduceproduct} as a another form of locality. Propositions \ref{prop:locality1} and \ref{prop:locality2} imply that, whenever
$I_{\bullet} \expr I_{\bullet}'$, then $H_{\bullet} \circ I_{\bullet} \circ K_{\bullet} \expr H_{\bullet} \circ I'_{\bullet} \circ K_{\bullet}$, and this operation plays nicely
with reducedness. Meanwhile, Proposition \ref{prop:reduceproduct} implies that $I_{\bullet}^{(+J)} \expr (I'_{\bullet})^{(+J)}$, and this operation plays nicely with reducedness.

\begin{rem} \label{rem:monoidal} One can glue the categories $\SC(W,S)$ for all Coxeter systems together into a large monoidal category. The objects are triples $(W,S,I)$ where
$(W,S)$ is a Coxeter system and $I \subset S$ is finitary. There are no morphisms between $(W,S,I)$ and $(W',S',I')$ unless $(W,S) = (W',S')$, in which case the Hom spaces are
governed by $\SC(W,S)$. We place a monoidal structure on objects where $(W,S,I) \otimes (W',S',I') := (W \times W', S \sqcup S', I \sqcup I')$. On morphisms, the operation $p
\mapsto p^{(+J)}$ corresponds to taking the tensor product with the identity morphisms of $(W',S',J)$. The second kind of locality discussed in the previous paragraph is locality
for the horizontal (tensor) product, rather than for the vertical (composition) product\footnote{We have chosen to be imprecise in this remark because being precise obfuscates the point. Asking that $(W,S) = (W',S')$ for morphisms to exist is unusual, demanding an equality of sets $S = S'$,  whereas it would be more appropriate to demand and keep track of the data of an isomorphism $(W,S) \to (W',S')$. However, this would produce additional automorphisms of a Coxeter system with Dynkin diagram symmetries. Perhaps the best way to make this construction precise is to let the objects be triples $(W,S = \{s_1, \ldots, s_r\}, I)$ where $S$ is equipped with a total order, and to declare that there are no morphisms between $(W,S,I)$ and $(W',S',I')$ unless the unique ordered map $S \to S'$ induces an isomorphism $W \to W'$, which we use to identify $W$ and $W'$. For this construction the monoidal structure will not be symmetric, which is perhaps for the best.}. \end{rem}

\begin{rem} For a different style of monoidal category built from $\SC(W,S)$ as the Coxeter system varies, see \S\ref{sec:typeAwebs}. It is still Proposition \ref{prop:reduceproduct} which implies that the monoidal structure is well-defined. \end{rem}

\section{Relating expressions} \label{sec-relating}

Recall that we write $I_{\bullet} \expr p$ if the composition of these generating morphisms is the double coset $p$, see Proposition \ref{prop:expressesinSC}. We write $I_{\bullet}
\expr I_{\bullet}'$ if $I_\bullet$ and $I'_\bullet$ express the same double coset $p$. Clearly $\expr$ is an equivalence relation on expressions. Recall from
Proposition~\ref{prop:locality2} that $\expr$ is a \emph{monoidal} relation, meaning that it is preserved by enlarging an expression to the left or right (i.e. it can be applied to
contiguous subwords). We seek a minimal set of relations which, applied locally, imply the equivalence relation $\expr$ on expressions. By Proposition \ref{prop:expressesinSC},
this will also be a list of relations in a presentation of $\SC$.

Before embarking on the study of relations, we provide yet more convenient shorthand for $J$-expressions.
	
\begin{notation} If $J \subset S$ is finitary, we may write
\[ [J + s - t + u - s] := [J, Js, Js \setminus t, Jsu \setminus t, Ju\setminus t].\]
For this to make sense (if $s$, $t$, and $u$ are distinct) we must have $t \in J$ and $s,u \notin J$.

Formally, we denote an expression $I_\bullet = [J,I_1,\cdots,I_d]$ by $[J \varepsilon_1 t_1 \varepsilon_2 t_2\cdots \varepsilon_d t_d]$, where $t_i\in S$ and $\varepsilon_i\in\{+, -\}$. Then $+t_i$ encodes the statement $I_i = I_{i-1}\sqcup \{t_i\}$, and $-t_i$ encodes $I_i = I_{i-1}\setminus \{t_i\}$. \end{notation}

\begin{ex} The expressions in Example~\ref{exleninc} and  Example~\ref{exlenth} can be denoted as 
\[[\emptyset,\{s_1\},\{s_1,s_2\},\cdots ,\{s_1,\cdots s_d\}]=[\emptyset+s_1+s_2+\cdots+s_d]\]
\[[\emptyset,s_1,\emptyset,s_2,\emptyset,\cdots,s_d,\emptyset]=[\emptyset+s_1-s_1+s_2-s_2+\cdots +s_d-s_d].\]
\end{ex}

\subsection{The quadratic relation}\label{ss:nonbraidgen}

\begin{lem} \label{lem:easynonrex} 
For any $J \subset S$ finitary and $s \in J$ we have the \emph{$*$-quadratic relation}
\begin{equation} \label{eq:starquad} [J - s + s] \expr [J]. \end{equation}
The left hand expression is not reduced, and the right hand expression has strictly smaller length.
\end{lem}

\begin{proof} Every coset in the forward path of $[J - s + s]$ has underlying set equal to $W_J$ itself. Thus the two expressions both express the identity $(J,J)$-coset. The right
hand expression is obviously reduced, so it has minimal length among all expressions. The left hand expression is not reduced at the step $+s$ since the redundancy increases.
\end{proof}

In fact, the $*$-quadratic relation will be the only non-braid relation we need to present $\SC$, as we prove in \S\ref{ss:presentation}. 

\subsection{Easy braid relations}\label{ss:braidreleasy}

\begin{lem} \label{lem:upupdowndown} For $s,t\in S$, the following relations hold for all $J$ for which they make sense, and preserve length.
\begin{equation} \label{upup} [J+s+t] \expr [J+t+s] \end{equation}
\begin{equation} \label{downdown} [J-s-t] \expr [J-t-s] \end{equation}
\end{lem}

\begin{proof} This is obvious. \end{proof}
	
We call \eqref{upup} the \emph{up-up} or \emph{$+$-associativity relation}, and we call \eqref{downdown} the \emph{down-down} or \emph{$-$-associativity relation}.

We also have braid relations analogous to the $m_{su} = 2$ commutation relation in ordinary Coxeter theory. 

\begin{lem} \label{m=2} Let $I = I_1 \sqcup I_2$, where $m_{su} = 2$ for all $s \in I_1$ and $u \in I_2$. Pick some $s \in I_1$ and $u \in I_2$, and let $J = I \setminus s$. Then
\begin{equation} \label{m2braid} [J + s - u] \expr [J - u + s].\end{equation}
This relation preserves length. \end{lem}

\begin{proof} The left-hand expression in \eqref{m2braid} is a reduced expression for the double coset $p$ with maximal element $w_I$. Because $W_I = W_{I_1} \times W_{I_2}$ we have $w_I =
w_{I_1} . w_{I_2}$. We also have \[ w_{I \setminus s} = w_{I_1 \setminus s} . w_{I_2}, \] and similar equations for $w_{I \setminus u}$ and $w_{I \setminus s,u}$. So
\begin{equation} w_{I \setminus s} w_{I \setminus s,u}\inv w_{I \setminus u} = (w_{I_1 \setminus s} w_{I_1 \setminus s}\inv w_{I_1}) . (w_{I_2} w_{I_2 \setminus u}\inv w_{I_2 \setminus u}) = w_{I_1} . w_{I_2} = w_I. \end{equation}
The lengths also add in this expression, since $\ell(I \setminus s) - \ell(I \setminus s,u) + \ell(I \setminus u) = \ell(I)$. Thus by Theorem \ref{thm:betterred}, the right-hand side of \eqref{m2braid} is a reduced expression for $p$ as well. \end{proof}

Lemma \ref{m=2} is the easy case of Proposition~\ref{prop:sss} below, which explores reduced expressions equivalent to $[J + s - t]$ in general.

\subsection{Two easy lemmas}

If $s \notin J$ and $Js$ is finitary, it is very easy to confirm that $[J+s]$ is a reduced expression.

\begin{lem} \label{lem:easyrex} Let $Js$ be finitary with $u \in J$ and $s \notin J$. Then $[J - u + s]$ is a reduced expression. \end{lem}

In the proofs in this chapter we repeatedly use the fact that an element $x$ is minimal in its right $W_I$ coset if and only if its right descent set is disjoint from $I$, if and
only if $xw_I = x . w_I$.

\begin{proof} We need only check that the sequence is reduced at the $+s$ step. Let $x = w_J w_{J \setminus u}\inv$. By
\eqref{reducedmaxformula} it is enough to check $x w_{Js\setminus u}= x.w_{Js\setminus u}$, which holds if and only if $\rightdes(x) \cap (Js \setminus u) = \mt$. Since $x \in W_J$, $\rightdes(x) \subset J$. Since $x$ is minimal in its right $W_{J \setminus u}$-coset, $\rightdes(x) \cap (J \setminus u) = \mt$. Combining these we deduce the result.
\end{proof}


The second lemma is a callback to Example \ref{nonrex}.

\begin{lem} \label{lem:easynonrex2} If $s \notin J$ and $Js$ is finitary, then $[J - u + s + u]$ is not reduced for all $u \in J$. \end{lem}

\begin{proof} We use $+$-associativity and the $*$-quadratic relation:
\begin{equation} [J - u + s + u] \expr [J - u + u + s] \expr [J + s]. \end{equation}
The length decreases when the $*$-quadratic relation is applied.
\end{proof}

\subsection{Switchback relations}\label{ss:specialswitch}

Throughout this section we write $w_0$ for the longest element in a finite Coxeter group, when the choice of Coxeter group is understood.

When $s \notin J$ and $Js$ is finitary, the expression $[J+s-t] \expr p$ is always reduced for any $t \in Js$ (including $t = s$). Note that $\ma{p} = w_{Js}$ is the longest
element of the parabolic subgroup $W_{Js}$. There may or may not be another reduced expression for $p$ which only factors through proper subsets of $Js$. In Proposition
\ref{prop:sss} we give a precise criterion for when such a reduced expression of $p$ exists, and an explicit construction of one when it does. Our \emph{switchback} (braid)
relation will relate $[J +s -t]$ with this other reduced expression. We make the switchback relation even more explicit for each finite Coxeter group (types $A$ through $H$) in
\S\ref{ssstypes}. Here we illustrate the dihedral case, which is nicely visualized in the singular Coxeter complex, see \eqref{A2coxintro} and \eqref{B2cox}.

\begin{ex} Let $S = \{s,t\}$ and $m_{st} = 3$ and $J = t$, so that $Js = S$ has maximal element $w_0 = sts$. Then $[J+s-t]$ is the unique reduced expression for the $(t,s)$-coset containing $w_0$. However, there are two reduced expressions for the $(t,t)$-coset containing $w_0$,
\begin{equation} [t,st,t] = [J + s - s] \expr [J - t + s - s + t] = [t,\mt,s,\mt,t], \end{equation} 
and this is an example of our switchback relation. \end{ex}

\begin{ex} \label{ex:dihedralss} More generally, let $S = \{s,t\}$ and $m_{st} = m < \infty$ and $J = t$. If $m$ is odd let $u = s$ and $v = t$, and if $m$ is even let $u = t$ and $v = s$. More abstractly, we have $u = w_0 t w_0$ and $v = w_0 s w_0$, so that $S = \{u,v\}$. Then $[J + s - v]$ is the unique reduced expression for the $(t,u)$-coset containing $w_0$. However, there are two reduced expressions for the $(t,v)$-coset containing $w_0$,
\begin{equation} [t,st,v] = [J + s - u] \expr [J - t + s - s + t - t + s \cdots - u + v] = [t,\mt,s,\mt,t,\mt, \ldots, \mt ,v] \end{equation}
where a total of $m-1$ operations ``subtract-then-add'' appear in the right hand expression. \end{ex}


Many times in this section we use the observation that, for a finite Coxeter system $(W,S)$, conjugation by $w_0$ preserves $S$, on which it acts by an involution.

\begin{defn} \label{def:useq}
Fix the data of $s \notin J \subset S$ with $Js$ finitary, and $t \in Js$ (we allow $t = s$), such that $s \ne w_{Js} t w_{Js}$. Define elements $u_0, u_{-1} \in S$ and finitary subsets $I_0, I_{-1} \subset S$ by the formulas
\begin{equation} \label{uinitialize} u_0 = s, \quad u_{-1} = w_{Js} t w_{Js}, \quad I_0 = J, \quad I_{-1} = Js \setminus u_{-1}. \end{equation}
Now define $u_i \in S$ and $I_i \subset S$ for all $i \in \Z$ recursively by the formulas
\begin{equation} \label{urecurse} u_{i+1} = w_{I_i} u_{i-1} w_{I_i}, \quad I_i = Js \setminus u_i. \end{equation}
We call $(u_i)_{i \in \Z}$ the \emph{$u$-sequence} or \emph{rotation sequence} associated to the triple $(J,s,t)$.
\end{defn}

Motivation for the rotation sequence comes from the reflection representation of $(W,S)$, see \S\ref{moregeometry2}. It is not immediately obvious that the elements $u_i$ are actually simple reflections, or that they live in $Js$ so that $I_i$ is well-defined.

\begin{lem} Definition \ref{def:useq} makes sense, giving a well-defined sequence $(u_i)$ of elements of $Js$. \end{lem}

\begin{proof} We claim that $u_{i-1}, u_{i+1} \in I_i \subset Js$, which we prove by induction. This suffices to show that $u_i$ and $I_i$ are well-defined. Let us first treat the
case of $i \ge 0$. Note that $u_{-1} \in Js$ since it is the conjugation by $w_{Js}$ of a simple reflection inside $Js$. Since $u_{-1} \ne s$ by assumption, we must have $u_{-1} \in
J = I_0$. Now, if $u_{i-1} \in I_i$ then conjugation by $w_{I_i}$ produces another element of $I_i$, namely $u_{i+1}$. Since $I_i = Js \setminus u_i$ we deduce that $u_{i+1} \ne
u_i$. Since $I_{i+1} = Js \setminus u_{i+1}$ and $u_i \in Js$ we deduce that $u_i \in I_{i+1}$, completing the induction step.

For $i<0$ we use the formula $u_{i-1} = w_{I_i} u_{i+1} w_{I_i}$ which is equivalent to \eqref{urecurse}. Working backwards (i.e. using decreasing induction on $i \le -1$), the
same argument will work. \end{proof}

Because $S$ is finite, eventually the pair $(u_i, u_{i+1})$ is equal to $(u_k,u_{k+1})$ for some $i \ne k$, implying that the rotation sequence is periodic.

\begin{ex} Use the same notation from Example \ref{ex:dihedralss}. There is no rotation sequence associated to $(J,s,v)$ since $w_0 v w_0 = s$. The rotation sequence associated to $(J,s,u)$ is merely the alternating sequence
\begin{equation} \ldots, s, t = u_{-1}, s = u_0, t = u_1, s, \ldots. \end{equation}
We have $u_{i+1} = u_{i-1}$ since $I_i = \{u_{i-1}\}$. \end{ex}

\begin{ex} \label{typeAssEx1} Let $W = S_n$ and $S = \{s_1, \ldots, s_{n-1}\}$ be the Coxeter group of type $A_{n-1}$. Let $s = s_i$ and $J = S \setminus s$. There is no rotation sequence associated to $(J,s_i,s_{n-i})$ since $w_0 s_i w_0 = s_{n-i}$. The rotation sequence associated to $(J,s_i, s_{n-j})$ for $j \ne i$ is actually six-periodic (for all $i, j, n$), and $u_{i+3} = w_0 u_i w_0$. We leave this as an exercise to the reader.
	
For example, when $n = 10$ and $i = 3$ and $j= 4$ we have sequence
\begin{equation} \ldots, s_4 = u_{-1}, s_3 = u_0, s_9, s_6, s_7, s_1, s_4 = u_5, s_3 = u_6, \ldots.  \end{equation}
Moreover, $u_{i+3} = w_0 u_i w_0$.
\end{ex}

\begin{prop}[Switchback relations]\label{prop:sss}
Let $s \notin J \subset S$ with $Js$ finitary, and $t \in Js$. Let $I = Js \setminus t$. Let $p$ be the $(J,I)$-coset containing $w_{Js}$.
\begin{enumerate}
    \item\label{uniqrex} If $s=w_{Js}tw_{Js}$, then $[J+s-t]$ is the unique reduced expression of $p$.
	\item\label{eqsss} If $s\neq w_{Js}tw_{Js}$, then $p$ has exactly two reduced expressions, given in \eqref{sssform}.
\end{enumerate}
To elaborate in case \eqref{eqsss}, let $(u_i)$ be the rotation sequence associated with $(J,s,t)$, see Definition \ref{def:useq}. Let $I_i = Js \setminus u_i$ as in Definition \ref{def:useq}, and $L_i = Js \setminus \{u_i, u_{i-1}\}$. Let $\snum$ be the largest integer such that
\begin{equation} \label{alternatingrex} [I_0, L_1, I_1, \ldots, L_k, I_k] \end{equation}
is reduced for all $1 \le k \le \snum$. Then $u_\snum = t$ and
\begin{equation} \label{sssform} 
\begin{split} 
    [J +s -t] \expr & [J-u_1+s -u_2 +u_1 -u_3+u_2-u_4 \cdots -u_{\snum-1}+u_{\snum-2} -t+u_{\snum-1}]  \\& = [J= I_0,L_1,I_1,L_2,I_2,L_3,I_3,L_4,\cdots,L_\snum,I_\snum = I].
\end{split}
\end{equation}
Moreover, $u_{\snum+1} = w_{Js} s w_{Js}$.
\end{prop}

We prove the result in the next section, after some remarks and easy lemmas. First we discuss the number $\snum$ associated to the triple $(J,s,t)$.

\begin{rem} Note that $u_{\snum} = t = w_{Js} u_{-1} w_{Js}$, and $u_{\snum + 1} = w_{Js} u_0 w_{Js}$. Using the recursive definition of the $u$-sequence we see that
\begin{equation} \label{addd} u_{i + \snum + 1} = w_{Js} u_i w_{Js} \end{equation}
for all $i \in \Z$. Since $w_{Js}$ is an involution, \eqref{addd} implies that the $u$-sequence is $2(\snum+1)$-periodic. When $w_{Js}$ is central in $W_{Js}$ then the $u$-sequence must
actually be $(\snum+1)$-periodic. However, the period of the $u$-sequence can be smaller still, as we already saw in dihedral type where it was $2$-periodic. This also happens
sporadicaly in type $EFH$. In particular, $\snum$ is \emph{not} the smallest index for which $u_\snum = t$, or for which \eqref{addd} holds. \end{rem}

\begin{rem} In \S\ref{ssstypes} we give explicit formulas for $\snum$ and the sequence $(u_1, u_2, \ldots, u_\snum)$, when $Js$ is an irreducible finite Coxeter group. In classical types,
the value of $\snum$ is relatively small. For example, $\snum = 2$ in type $A$, and $2 \le \snum \le 3$ for types $BCD$. \end{rem}

\begin{rem} Lemma \ref{lem:easynonrex2} gives an example where $\snum=1$. Conversely, every example where $\snum=1$ will be a consequence of Lemma \ref{lem:easynonrex2}, as we justify shortly. \end{rem}

\begin{rem} \label{rmk:cantsimplifyfurther} Within the reduced expression in \eqref{sssform} one has contiguous subwords of the form
\begin{equation} \label{intermediatenotbad} [L_k, I_k, L_{k+1}] = [L_k + u_{k-1} - u_{k+1}]. \end{equation}
While also being of the form $[J+s-t]$ for a different triple $(J,s,t)$, this subword will always fall into the regime of Proposition \ref{prop:sss}\eqref{uniqrex}, because $u_{k+1} = w_{I_k} u_{k-1} w_{I_k}$ by construction. In particular, we can not apply the switchback relation to subwords of the right hand expression in \eqref{sssform}. \end{rem}

\begin{rem} It sometimes helps to think about the switchback relation as being associated not with the triple $(J,s,t)$ where $s \ne w_{Js} t
w_{Js}$, but with the triple $(J,s,u_1)$ where $s \notin J$ and $u_1 \in J$ (so $s \ne u_1$). If we start at $J$ one might begin a reduced expression by adding $s$ or by removing
$u_1$. The switchback relation explains how to extend these beginnings to reduced expressions for the same coset. \end{rem}

Now let us continue Example \ref{typeAssEx1}.

\begin{ex} \label{typeAssEx2} Recall that when $n = 10$ and $s = s_3$ and $t= s_6$ we have $u_1 = s_9$, $u_2 = t$. The switchback relation is
\begin{equation} [J + s_3 - s_6] \expr [J - s_9 + s_3 - s_6 + s_9]. \end{equation}
The right hand side is a reduced expression if and only if $w_0 = w_{\hat{3}}. (w_{\hat{39}}\inv w_{\hat{9}}) . (w_{\hat{69}}\inv w_{\hat{6}})$, or equivalently, iff
\begin{equation} w_{\hat{3}}\inv w_0 = (w_{\hat{39}}\inv w_{\hat{9}}) . (w_{\hat{69}}\inv w_{\hat{6}}). \end{equation}
Here is the picture in $S_{10}$ which justifies this equality.
\[ {
\labellist
\tiny\hair 2pt
 \pinlabel {$=$} [ ] at 45 15
\endlabellist
\centering
\ig{1.5}{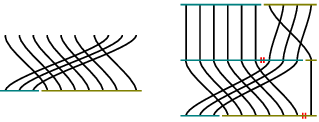}
} \]
\end{ex}

The next lemma will simplify our notation.

\begin{lem} \label{lem:stayinK} Suppose that $p$ is a parabolic double coset in $(W,S)$ and $M\subset S$ is finitary such that $\ma{p} \in W_M$. Then in any expression $I_{\bullet} = [I_0, \ldots,
I_d]$ for $p$, $I_i \subset M$ for all $0 \le i \le d$. \end{lem}

\begin{proof} If $x \le y$ in the Bruhat order on $W$, and the simple reflection $s$ appears in a reduced expression for $x$, then it also appears in any reduced expression for
$y$. Letting $p_{\bullet}$ be the forward path of $I_{\bullet}$, we know that $\ma{p_{i-1}} \le \ma{p_i}$ in the Bruhat order for all $i$, and thus $\ma{p_i} \le \ma{p} = \ma{p_d}$ for all
$i$. If $I_i \not\subset M$ then $\ma{p_i}$ has some $s \notin M$ in its right descent set, and this $s$ must appear in a reduced expression for $\ma{p}$, contradicting the
assumption $\ma{p} \in W_M$. \end{proof}

By Lemma \ref{lem:stayinK} with $M = Js$, it suffices when studying Proposition \ref{prop:sss} to restrict to the parabolic subgroup $W_{Js}$. Below we adopt the convention that $S
= Js$ and $W$ is finite and $\ma{p} = w_0$. Thus $J = S \setminus s$ and $I = S \setminus t$. Using the results of \S\ref{ss:ssirred} we could reduce further to the case when $W$
is irreducible, though the assumption of irreducibility does not simplify the proof so we do not make it here.

%

\subsection{Proof of the switchback relations}\label{ss:specialswitchproof}

Before embarking on the proof, one more lemma.

\begin{lem} \label{lem:uniquerexcriterion}  Let $s, t \in S$ for $S$ finitary (possibly $s = t$). For $J = S \setminus s$ and $I = S \setminus t$ and $p = W_J w_0 W_I$, we have $J \cap \mi{p} I \mi{p}\inv = J$ if and only if $\mi{p} = w_0 w_I\inv$ if and only if $s = w_0 t w_0$. \end{lem}

\begin{proof} Clearly $w_0 = \ma{p}$. The formula \eqref{maxformula} implies that $\mi{p} w_I = \ma{p}$ if and only if $J \cap \mi{p} I \mi{p}\inv = J$. If $\mi{p} w_I = w_0$ then $J \cap w_0 I w_0 = J \cap \mi{p} I \mi{p}\inv = J$, which implies that $w_0 t w_0 = s$. Conversely, if $w_0 t w_0 = s$ then $w_0 J w_0 = I$. Thus there is an equality of sets $W_J w_0 = w_0 W_I$, and $p = w_0 W_I$. Consequently $\mi{p} = w_0 w_I\inv$. \end{proof}
	
\begin{proof}[Proof of Proposition \ref{prop:sss}\eqref{uniqrex}]
By Lemma \ref{lem:stayinK} we can assume $Js = S$. We write $p = W_Jw_0W_I$ and note that $[J+s-t] \expr p$.

Suppose that there exists a reduced expression $J_{\bullet} = [J = J_0,\ldots, J_{m-1} = K,J_m = I]$ for $p$ ending in $+u$, that is, satisfying $I = Ku$. The fact that $[K,I]$ is
reduced implies that \begin{equation} \label{ssredundancy} J \cap \mi{p} K \mi{p}\inv = J \cap \mi{p} I \mi{p}\inv. \end{equation} The size of $K$ is one less than the size of $I$, which equals the size of $J$. Thus $J \cap \mi{p} K \mi{p}\inv$ is necessarily a proper subset of $J$. By Lemma \ref{lem:uniquerexcriterion}, this implies that $s \ne w_0 t w_0$. Taking the contrapositive, if $s = w_0 t w_0$ then there is no reduced expression for $p$ ending in $+u$ for any $u \in S$.

To finish proving \eqref{uniqrex}, we need to prove that $p$ has no reduced expressions aside from $[J+s-t]$ which end in $-u$ for some $u \in S$. Of course, since $I = S \setminus
t$ one must have $u = t$, and the penultimate parabolic in the expression is $S$ itself. Preceding this is a reduced expression for the unique $(J,S)$-coset $q$. There are a number of easy ways to argue that $[J+s]$ is the unique reduced expression for $q$; here is one. We know that $\mi{q} = e$, whereas the minimum element $\mi{p_i}$ strictly increases whenever a simple reflection is removed in a reduce expression. So we can only add simple reflections and the only possibility is $[J+s]$. \end{proof}

Now we introduce more lemmas on route to proving Proposition \ref{prop:sss}\eqref{eqsss}. The first lemma helps to justify the importance of the first term $u_{+1}$ in the rotation sequence.

\begin{lem} \label{lem:whoareu} Let $S$ be finitary, with $s, t \in S$ (possibly equal) satisfying $w_0 t w_0 \ne s$. Let $J = S \setminus s$ and $I = S \setminus t$. Let $p$ be
the $(J,I)$-coset containing the longest element $w_0$ of $(W,S)$. Then the left redundancy of $p$ is $J \setminus u$, where $u = w_J w_0 t w_0 w_J$. \end{lem}

\begin{proof} Let $u_{-1} = w_0 t w_0$, and $u = u_{+1} = w_J u_{-1} w_J$. Note that $u_{+1} \in J$, i.e. $u_{+1} \ne s$. This is because $u_{-1} \ne s$ by assumption, so $u_{-1} \in J$, and conjugation by $w_J$ preserves $J$.

Let $L$ be the right redundancy of $p$. Then \eqref{maxformula} states that $\ma{p} = w_J \mi{p} w_L w_I$. Since $\ma{p} = w_0$, we deduce that
\begin{equation} \mi{p} = w_J w_0 w_I w_L. \end{equation}
By Kilmoyer's theorem (Lemma~\ref{J WJ same}), the left redundancy $K$ is equal to $J \cap \mi{p} W_I \mi{p}\inv$. Since $L \subset I$ we have 
\begin{equation} K = J \cap w_J w_0 W_I w_0 w_J.\end{equation}
Now $I = S \setminus t$, so $w_0 W_I w_0 = W_{S \setminus u_{-1}}$. Also, $w_J (J \setminus u_{-1}) w_J = J \setminus u_{+1}$. Thus
\begin{equation} J \setminus u_{+1} \subset J \cap w_J W_{S \setminus u_{-1}} w_J = K. \end{equation}
By Lemma \ref{lem:uniquerexcriterion}, we can not have $J = K$, so we must have $J \setminus u_{+1} = K$. \end{proof}

\begin{cor} \label{cor:onlyu} Continue the notation of Lemma \ref{lem:whoareu}. For any $v \in J$, there is a reduced expression for $p$ of the form $[J - v] \circ M_{\bullet}$ if
and only if $v = u$. \end{cor}

\begin{proof} This follows from Theorem \ref{thm:favoritesmall} and Corollary \ref{cor:nottoosmall}. \end{proof}
	
To help cement these ideas in the reader's mind, let us give another proof of Corollary \ref{cor:onlyu}, this time using Proposition \ref{p to w0} instead.

\begin{proof}[Second proof of Corollary \ref{cor:onlyu}] Proposition \ref{p to w0} states that $[J-v]$ can be extended to a reduced expression for $p$ if and only if $I$ is
contained in the right descent set of $x := w_{J \setminus v} w_J w_0$. In other words, we must show that $t \notin \rightdes(x)$, or equivalently $x t > x$. Right multiplication
by $w_0$ reverses the Bruhat order, so \[ w_{J \setminus v} w_J u_{-1} = w_{J \setminus v} w_J w_0 t w_0 < w_{J \setminus v} w_J. \] Both sides of this inequality live in $W_J$,
and right multiplication by $w_J$ reverses the Bruhat order on $W_J$, so \[ w_{J \setminus v} u_{+1} = w_{J \setminus v} w_J u_{-1} w_J > w_{J \setminus v}. \] This means that
$u_{+1} \notin J \setminus v$. Since $u_{+1} \in J$, we deduce that $v = u_{+1}$. \end{proof}

The second lemma introduces a ``rotation'' operation on reduced expressions for certain double cosets containing $w_0$.

\begin{lem} \label{lem:rotation} Let $S$ be finitary, containing elements $s, t$ with $w_0 t w_0 = s' \ne s$ and $w_0 s w_0 = t' \ne t$. Let $J = S \setminus s$ and $I = S \setminus t$. If $[J - u + s] \circ M_{\bullet}$ is a reduced expression for $p$, the $(J,I)$-coset containing $w_0$, then $M_{\bullet} \circ [I - t' + t]$ is a reduced expression for $p'$, the $(S \setminus u, S \setminus t')$-coset containing $w_0$. Note: Lemma \ref{lem:whoareu} gives a formula for $u$. \end{lem}

\begin{proof} By locality, $M_{\bullet} \expr r$ and $[J - u + s]$ are both reduced expressions. Let $J' = S \setminus u$ and $L = S \setminus \{s,u\}$. We have
\begin{equation} w_0 = \ma{p} \stackrel{\eqref{yay}}{=} w_J w_L\inv \ma{r}.\end{equation}
Moreover, the lengths add, in that $\ell(w_0) = \ell(J) - \ell(L) + \ell(\ma{r})$.
Now $L = J \setminus u$ and, by Lemma \ref{lem:whoareu}, we have $w_J u w_J = s'$. Hence $w_L w_J = w_J w_{L'}$ where $L' = S \setminus {s,s'}$. So
\begin{equation} \ma{r} = w_L w_J w_0 = w_J w_{L'} w_0 = w_0 w_{w_0 J w_0} w_{w_0 L' w_0}. \end{equation}
Letting $I' = w_0 J w_0 = S \setminus t'$ and $L'' = w_0 L' w_0 = S \setminus \{t,t'\}$ we get
\begin{equation} w_0 = \ma{r} w_{L''}\inv w_{I'}, \end{equation}
and the lengths add in that $\ell(w_0) = \ell(\ma{r}) - \ell(L'') + \ell(I')$. By Proposition \ref{prop:concatreduced} this implies that $M_{\bullet} \circ [I-t'+t]$ is reduced.
\end{proof}

Now we prove the switchback relation. 

\begin{proof}[Proof of Proposition \ref{prop:sss}\eqref{eqsss}]
As before, by Lemma \ref{lem:stayinK} we can assume $Js = S$. Suppose $s\neq w_0 tw_0 =: u_{-1}$. We leave it to the reader to verify that $[J+s - t]$ is the unique reduced expression for $p$ which extends $[J + s]$. The other reduced expressions begin with $[J - v]$ for some $v \in J$. By Corollary \ref{cor:onlyu}, we must have $v = u_1$. Moreover, by Corollary \ref{cor:onlyu}, there exists a reduced expression for $p$ which extends $[J - u_{1}]$.
In the following paragraphs, we give an algorithm to extend $[J-u_1]$ towards a reduced expression of $p$.

By Lemma \ref{lem:whoareu} and Corollary \ref{cor:nottoosmall}, there is no reduced expression for $p$ which extends $[J - u_1 - v]$ for any $v \in J$. So the reduced expression must continue by adding some simple reflection not in $J \setminus u_1$, and there are only two choices: $s$ and $u_1$. Also, $[J - u_1 + u_1]$ is not reduced by Lemma \ref{lem:easynonrex}. Thus $p$ must have a reduced expression of the form $[J - u_1 + s] \circ M_{\bullet}$ for some $M_{\bullet}$.

Lemma \ref{lem:easyrex} verifies that $[J - u_1 + s]$ is reduced. It is possible that $[J - u_1 + s]$ is a reduced expression for $p$, at which point $M_{\bullet}$ has width zero
and our algorithm terminates. 
Thus assume that $M_{\bullet}$ has positive width. We note that $M_{\bullet}$ can not begin by adding a simple
reflection, since the only available choice is $u_1$, and $[J - u_1 + s + u_1]$ is not reduced by Lemma \ref{lem:easynonrex2}. So $M_{\bullet}$ begins $[I_1 - v]$ for some $v \in
I_1 = S \setminus u_1$.

By Lemma \ref{lem:rotation} (and using the notation therein), $M_{\bullet} \circ [I - t' + t]$ is a reduced expression for $p'$, the $(I_1, S \setminus t')$-coset containing $w_0$.
Since $w_0 t' w_0 = s \ne u_1$, we are back exactly in the situation of Proposition \ref{prop:sss}\eqref{eqsss} again! This time $u_1$ is playing the role of $s$, and $t'$ plays
the role of $t$. Thus the new $u_0$ is the old $u_1$, and the new $u_{-1} = w_0 t' w_0 = s$ is the old $u_0$.

Again, $M_{\bullet}$ begins $[I_1 - v]$ for some $v$. Repeating the arguments used already in this proof (citing Corollary \ref{cor:onlyu} etcetera) we see that $M_{\bullet}$ necessarily has the form $[I_1 - u_2 + u_1] \circ N_{\bullet}$ for some $N_{\bullet}$. So 
\[ [J - u_1 + s - u_2 + u_1] \circ N_{\bullet} \expr p, \qquad [I_1 - u_2 + u_1] \circ N_{\bullet} \circ [I - t' + t] \expr p'. \]
Rotating again with Lemma \ref{lem:rotation}, we see that
\[ N_{\bullet} \circ [I - t' + t - t'' + t'] \]
is a reduced expression for a double coset containing $w_0$, and $t'' = w_0 u_1 w_0$.

When $N_{\bullet}$ is width zero our algorithm terminates. Otherwise we can continue, repeating until we have exhausted the reduced expression. Note that the width of $M_{\bullet}$
is two less than the original reduced expression, the width of $N_{\bullet}$ is two smaller still, etcetera. This process will eventually terminate (let $\snum$ be the number of
steps required), resulting in a reduced expression for $p$ of the form \eqref{sssform}. Our choices were constrained at every step, so this is the unique reduced expression for $p$
which extends $[J - u_1]$.

Necessarily we have $u_{\snum} = t$. Furthermore, we know that the rotation \[ [I_1 - u_2 + u_1 - \ldots - t + u_{\snum - 1} - t' + t] \] is a reduced expression for $p'$, yielding
$u_{\snum + 1} = t' = w_0 s w_0$.
\end{proof}

\subsection{Switchback and products of Coxeter systems} \label{ss:ssirred}

\begin{lem} \label{lem:samecomponent} Let $(W,S)$ and $(W',S')$ be finite Coxeter systems, with $J = S \sqcup S' \setminus s$ and $I = S \sqcup S' \setminus t$ inside the Coxeter system $(W \times W', S
\sqcup S')$. If $s, t \in S$ then the rotation sequence associated to $(J, s, t)$ agrees with the rotation sequence associated to $(J \cap S, s, t)$. Moreover, the relation
\eqref{sssform} for $(J,s,t)$ follows from the relation \eqref{sssform} for $(J \cap S, s, t)$ via the operation $(-)^{(+S')}$, see Proposition \ref{prop:reduceproduct}. \end{lem}

\begin{proof} This is a straightforward exercise in the definitions. \end{proof}

\begin{lem} \label{lem:disjointmeansm2} Let $(W,S)$ and $(W',S')$ be finite Coxeter systems, with $J = S \sqcup S' \setminus s$ and $I = S \sqcup S' \setminus t$ inside the Coxeter system $(W \times W', S
\sqcup S')$. If $s \in S$ and $t \in S'$ then the rotation sequence associated to $(J,s,t)$ is
\begin{equation} \ldots, \quad u_{-1} = w_{S'} t w_{S'}, \quad  u_0 = s, \quad u_1 = t, \quad u_2 = w_S s w_S, \quad \ldots \end{equation}
and $u_{i+2} = w_{S \sqcup S'} u_i w_{S \sqcup S'}$. The relation \eqref{sssform} holds with $\snum = 1$, because it agrees with \eqref{m2braid}. \end{lem}

\begin{proof} Again, this is straightforward. \end{proof}

\begin{lem} Let $(W,S)$ be a finitary Coxeter system, and $s, t \in S$. The following are equivalent. \begin{enumerate}
\item There is a switchback relation with $\snum = 1$ for the triple $(S \setminus s, s,t)$.
\item $w_S = w_{S \setminus s} . (w_{S \setminus s,t}\inv w_{S \setminus t})$,
\item $\ell(S) = \ell(S \setminus s) - \ell(S \setminus s,t) + \ell(S \setminus t)$,
\item $s$ and $t$ belong in different irreducible components of $S$, i.e. we are in the setting of Lemma \ref{lem:disjointmeansm2}.
\end{enumerate} \end{lem}

\begin{proof} Let $J = S \setminus s$. Lemma \ref{lem:disjointmeansm2} says that (4) implies (1). Meanwhile, (1) says that $[J - t + s]$ is a reduced expression for the double coset containing $w_S$, which by Theorem \ref{thm:betterred} implies (2). By taking lengths, (2) implies (3). It remains to prove that (3) implies (4).
	
If $s$ and $t$ were in the same irreducible component of $S$ (call it $L$) and $S = L \sqcup L'$ then $\ell(S) = \ell(L) + \ell(L')$, $\ell(S \setminus s) = \ell(L \setminus s) + \ell(L')$, etcetera. Thus (3) implies that
\begin{equation} \label{thiswonthappen} \ell(L) = \ell(L \setminus s) - \ell(L \setminus s,t) + \ell(L \setminus t), \end{equation}
which reduces to the case when $W = W_L$ and the Coxeter group is irreducible. One can now check all cases to verify that \eqref{thiswonthappen} never occurs. \end{proof}

\begin{rem} Relatedly, in \S\ref{ssstypes}, we provide explicitly (by checking all cases) the switchback relations for finitary irreducible Coxeter systems. One always has $\snum \ge 2$. The previous two lemmas reduce the switchback relation to irreducible Coxeter systems, so we can directly confirm that (1) and (4) are equivalent. \end{rem}

\subsection{Matsumoto's Theorem}\label{mats}

We prove the singular Matsumoto theorem by repeated use of the techniques developed in \S\ref{subsec:getbig} and \S\ref{subsec:getsmall}.

\begin{thm} \label{thm:matsumoto}
Let $(W,S)$ be a Coxeter group and let $p$ be a double coset.
The braid relations from Lemma \ref{lem:upupdowndown} and Proposition \ref{prop:sss} generate all relations between reduced expressions for $p$.
\end{thm}

\begin{proof} Suppose that $p$ is a $(J,I)$-coset. Let $K_{\bullet} = [J = K_0, K_1, \ldots, K_{d-1}, K_d = I]$ and $L_{\bullet} = [J = L_0, L_1, \ldots, L_{c-1}, L_c = I]$ be two
reduced expression for $p$. We prove by induction on $\ell(p)$ (\emph{not} on the widths $d$ and $c$) that $K_{\bullet}$ and $L_{\bullet}$ are related by the braid relations. The
base case is trivial: the only expressions of length zero are identity expressions $[J]$. For the rest of this proof we write $K_\bullet \sim L_\bullet$ to indicate that $K_\bullet$ and $L_\bullet$ are related by the braid relations.

Let $K'_{\bullet} = K_{[1,d]} \expr q$ and $L'_{\bullet} = L_{[1,c]} \expr r$; these are both reduced expressions by locality. If $K_1 = L_1$ then $q = r$, since Proposition \ref{prop:concatreduced} implies
\[ w_J w_{J \cap K_1}\inv \ma{q} = \ma{p} =  w_J w_{J \cap L_1}\inv \ma{r}. \]
The double coset $q$ also has shorter length than $p$, since extending a reduced expression increases length. By induction, $K'_{\bullet} \sim L'_{\bullet}$, hence $K_{\bullet} \sim L_{\bullet}$.

We now examine three cases where $K_1 \ne L_1$. In the first case, $K_1 = J \setminus s$ and $L_1 = J \setminus t$ for $s \ne t \in J$. By Corollary \ref{cor:nottoosmall}, both $s$
and $t$ are in $J \setminus \mi{p} I \mi{p}\inv$. But Proposition~\ref{thm:favoritesmall} implies the existence of a reduced expression $M_{\bullet} \expr p$ starting with $[J - s
- t] = [J,K_1, K_1 \cap L_1]$. Applying the down-down relation, there is also a reduced expression $N_{\bullet} \expr p$ starting with $[J - t - s] = [J,L_1, K_1 \cap L_1]$. By
induction (and the previous paragraph) we have \[ K_{\bullet} \sim M_{\bullet} \sim N_{\bullet} \sim L_{\bullet}.\]

In the second case, $K_1 = Js$ and $L_1 = Jt$ for $s,t\notin J$ with $s \ne t$. We argue similarly, this time using Corollary \ref{cor:nottoobig} and Proposition~\ref{thm:favoritebig}.
Since both $s$ and $t$ live inside $\leftdes(\ma{p}) \setminus J$, we are permitted to add both and interpolate between $K_{\bullet}$ and $L_{\bullet}$.

In the final case, $K_1 = J \setminus u$ and $L_1 = Js$ for some $u \in J$ and $s \notin J$. Let $t \in J$ be such that the rotation sequence of $(J,s,t)$ has $u_1 = u$. If we can prove that $[J + s - t] \circ M_{\bullet}$ is a reduced expression for $p$ for some $M_{\bullet}$, then we will obtain a chain
\[ L_\bullet \sim [J + s - t] \circ M_{\bullet} \sim [J-u+s-u_2+u - \ldots -t + u_{\snum-1}] \circ M_{\bullet} \sim K_\bullet.\]
The middle relation is the switchback relation, while the other two relations follow by induction.

The formula for $t$ is straightforward from Definition \ref{def:useq}, we have $t = w_{L_1} w_J u w_J w_{L_1}$. However, we must rule out the possibility that $w_J u w_J = s$, so
that no rotation sequence exists. This is easy, since by assumption $u \in J$ (so $w_J u w_J \in J$) while $s \notin J$.

As before let $L'_{\bullet} = L_{[1,d]}$ express the $(L_1, I)$-coset $r$. If $t \in L_1 \setminus \mi{r} I \mi{r}\inv$ then there is some reduced expression for $r$ which starts with $[L_1 - t]$, by Proposition~\ref{thm:favoritesmall}. Composing this with $[J,L_1]$, we get the desired reduced expression for $p$ starting with $[J + s - t]$. Clearly $w_J u w_J \in J \subset L_1$ so $t = w_{L_1} w_J u w_J w_{L_1} \in L_1$. We need only prove that $t \notin \mi{r} I \mi{r}\inv$. The other thing we know, from Corollary \ref{cor:nottoosmall}, is that $u \in J \setminus \mi{p} I \mi{p}\inv$. The rest of the proof is tricks with conjugation.

By Lemma \ref{J WJ same}, it suffices to examine $J \cap \mi{p} W_I \mi{p}\inv$ and $L_1 \cap \mi{r} W_I \mi{r}\inv$. By \eqref{maxformula}, the right $W_I$ coset of $\mi{p}$ agrees with the right $W_I$ coset of $w_J\inv \ma{p}$, and similarly for $r$. Also, by Theorem \ref{thm:betterred}, $\ma{r} = \ma{p}$ (since $w_J w_{J \cap L_1}\inv = e$). So
\begin{equation}
\begin{split} u & \notin \mi{p} W_I \mi{p}\inv \iff u \notin w_J \ma{p} W_I \ma{p}\inv w_J \iff u \notin w_J \ma{r} W_I \ma{r}\inv w_J \\ \iff w_J u w_J & \notin \ma{r} W_I \ma{r}\inv \iff t = w_{L_1} w_J u w_J w_{L_1} \notin w_{L_1} \ma{r} W_I \ma{r}\inv w_{L_1} \iff t \notin \mi{r} W_I \mi{r}\inv. \end{split} \end{equation}
	
\end{proof}

\subsection{Presentation theorem}\label{ss:presentation}

Let us recall the list of relations. We have the $\pm$-associativity relations from Lemma \ref{lem:upupdowndown}
\begin{subequations} \label{relationslist}
\begin{equation} [J+s+t] \expr [J+t+s], \qquad [J - s - t] \expr [J - t - s], \end{equation}
the switchback relation (details are in Proposition \ref{prop:sss})
\begin{equation} [J + s - t] \expr [J - u_1 + s - u_2 + u_1 - \ldots - t + u_{\snum-1}], \end{equation}
and the $*$-quadratic relation
\begin{equation} [J - s + s] \expr [J]. \end{equation}
The $\pm$-associativity and switchback relations are called braid relations.
\end{subequations}

\begin{thm}\label{thm:reducing} Let $I_{\bullet}$ be a non-reduced expression. By applying the braid relations, and applying the $*$-quadratic relation in the length-reducing
direction, one can relate $I_{\bullet}$ to a reduced expression. \end{thm}

\begin{proof} We prove the result by induction on the 
width $d$ of $I_{\bullet} = [I_0, \ldots, I_d]$, where the base case is trivial. Suppose that $I_{[m,n]}$ is not reduced for a
proper interval $[m,n]$ (meaning $m \ne 0$ or $n \ne d$). By induction we can reduce $I_{[m,n]}$, and applying the same relations within $I_{\bullet}$ produces an expression $I'_{\bullet}$ with smaller width.  
By induction we can also reduce $I'_{\bullet}$, so we win.

Now suppose $I_{[0,d-1]}$ is a reduced expression for some $(J,I_{d-1})$-coset $q$ where $J = I_0$. 
Note that $I_{\bullet}$ is reduced unless $I_{d-1}t = I$ for
some $t \in I$. We can also assume that $d \ge 2$, since any expression of width $d \le 1$ is reduced.

If the right redundancy $L = \mi{q}\inv J \mi{q} \cap I$ of $q$ is a proper subset of $I$, then Proposition~\ref{thm:favoritesmall} implies that $q$ has a reduced expression ending in
$+s$ for some $s \in I \setminus L$. By Theorem \ref{thm:matsumoto}, we can get from $I_{[0,d-1]}$ to this expression using braid relations. So we can assume that $I_{\bullet} =
[J_{\bullet} + s + t]$ for some reduced expression $J_{\bullet}$. The braid relation \eqref{upup} produces $I'_{\bullet} := [J_{\bullet} + t + s]$. If 
$I'_{[0,d-1]}$ is not reduced then we conclude by induction as above. If $[J_{\bullet} + s]$ and $[J_{\bullet} + t]$ are both reduced, then by Proposition \ref{prop:addable} so is
$I_{\bullet}$, a contradiction.

So we assume that $L = I$ henceforth. In this case 
\begin{equation} \label{fooforproof} \ma{q} = w_J . \mi{q}\end{equation} for some $x \in W_J$, by \eqref{maxformula}. In particular, \eqref{fooforproof} implies $\rightdes(\mi{q}) \subset \rightdes(\ma{q})$, see Lemma \ref{lem:dontchangeleftdescent}.

Suppose that $t \in \rightdes{(\ma{q})}$. Then by Proposition~\ref{thm:favoritebig}, there is a reduced expression for $q$ ending in $-t$. Again, by Theorem \ref{thm:matsumoto} we can use braid relations to assume that $I_{\bullet} = [J_{\bullet} - t + t]$ for some reduced expression $J_{\bullet}$. Now the $*$-quadratic relation reduces $I_{\bullet}$ to a reduced expression.

So suppose that $t \notin \rightdes(\ma{q})$ and $L = I$. We claim that $I_{\bullet}$ is reduced, a contradiction which concludes the proof. If $I_{\bullet} \expr p$, then $p$ is the $(J,It)$-coset containing $q$. But $t \notin \rightdes(\mi{q})$, so by \eqref{beingminimal} we know that $\mi{q} = \mi{p}$. Moreover, $\mi{q}\inv W_J \mi{q} = \ma{q}\inv W_J \ma{q}$ by \eqref{fooforproof}. If $t \in \ma{q}\inv W_J \ma{q}$, then $\ma{q} t = x \ma{q}$ for some $x \in W_J$. But $x \ma{q} \le \ma{q} < \ma{q} t$, a contradiction. So $t \notin \ma{q}\inv W_J \ma{q}$, and the redundancy of $p$ equals that of $q$. Hence $[p,q]$ is reduced. \end{proof}

\begin{thm} \label{thm:presentation}
The relations from \eqref{relationslist} give a presentation of the singular Coxeter monoid.
\end{thm}

\begin{proof}
We need only prove that two expressions for the same double coset are related. By Theorem \ref{thm:reducing}, each of them can be related to a reduced expression. By Theorem \ref{thm:matsumoto}, any two reduced expressions are related. \end{proof}

\section{Switchback relations} \label{ssstypes}

In this section, we describe explicitly all relations guaranteed by and classified in Proposition~\ref{prop:sss}. By Lemmas \ref{lem:samecomponent} and \ref{lem:disjointmeansm2},
we can reduce to the case where $(W,S)$ is a finite irreducible Coxeter group.

In all cases, we know that a switchback relation must exist. We explicitly state the relevant part of the rotation sequence. The interesting thing to explain is what $\snum$
is, i.e. how far along the rotation sequence to travel. The algorithm from \S\ref{ss:specialswitchproof} states that we should continue working along the rotation sequence until
one has expressed a double coset containing $w_0$. We can thus determine if we have gone the right distance merely be computing the length of the purported RHS of \eqref{sssform}.
We did this by computer in types $EFH$. We give as many details in the computations below as we find interesting.

The following lemma cuts in half the amount of work needed.

\begin{lem} \label{lem:ssflipped}
Suppose that $s, t \in S$ and $w_0 t w_0 \ne s$. If the switchback relation for the triple $(S \setminus s, s, t)$ has the form
\begin{subequations} \label{ssflipped}
\begin{equation} \label{ssflipped1} [(S \setminus s) + s - t] \expr [(S \setminus s) - c_1 + s - c_2 + c_1 - \ldots - t + c_{\snum-1}], \end{equation}
then the switchback relation for the triple $(S \setminus t, t, s)$ has the form
\begin{equation} \label{ssflipped2} [(S \setminus t) + t - s] \expr [(S \setminus t) - c_{\snum-1} + t \ldots  - c_1 + c_2 - s + c_1]. \end{equation}
\end{subequations}
In other words, the rotation sequences in these two cases are reversed.
\end{lem}

\begin{proof} Writing the expressions in \eqref{ssflipped} using the notation $[I_0, I_1, \ldots, I_d]$ rather than the notation of addition and subtraction, we see that the left-hand side of \eqref{ssflipped1} is just the reverse of the left-hand side of \eqref{ssflipped2}, and similarly with the right-hand side. Now the result follows from Proposition \ref{invrex}. \end{proof}

One can also deduce many switchback relations from others using Lemma \ref{lem:rotation}. One can find a detailed discussion of some examples of both Lemma \ref{lem:rotation} and Lemma \ref{lem:ssflipped} in type $E_7$ below.

%

\subsection{Switchback in type A}

Let $(W,S)$ be of type $A_n$, with  $S=\{s_1,\cdots,s_n\}$ in the usual numbering.
We have for each $1\leq a,b\leq n$ such that $a+b\neq n+1$ (i.e., $s_a\neq w_0s_bw_0$),
\begin{equation}\label{sssA}
  [J-s_c+s_a-s_b+s_c] = [J+s_a-s_b]  
\end{equation}
with $c = a+b$ (if $a+b\leq n$) or $c=a+b-n-1$ (if $a+b>n+1$), where $J= S\setminus\{s_a\}$.
 
Confirming the value of $c$ is straightforward, so we verify that the lengths of both sides of \eqref{sssA} agree. By symmetry, we need only treat the case $a+b\leq n+1$.

Since $W_S = S_{n+1}$ and $W_J = S_a \times S_{n+1-a}$, we have 
\[ \ell(S) - \ell(J) = a(n+1-a),\] the length of the relative longest element.  Similarly,
$W_{S \setminus s_c} = S_{a+b} \times S_{n+1-a-b}$, and $W_{S \setminus s_a, s_c} = S_a \times S_b \times S_{n+1-a-b}$. The relative longest element agrees with the relative longest element of $S_{a+b}$ over $S_a \times S_b$ so we have
\[ \ell(S \setminus s_c) - \ell(S \setminus s_a, s_c) = ab. \]
Finally, $W_{S \setminus s_b} = S_b \times S_{n+1-b}$ and $W_{S \setminus s_b, s_c} = S_b \times S_{a} \times S_{n+1-a-b}$, so
\[ \ell(S \setminus s_b) - \ell(S \setminus s_a, s_c) = a(n+1-a-b). \]
That the lengths of the two expressions in \eqref{sssA} are equal is equivalent to
\[ a(n+1-a) = ab + a(n+1-a-b). \]

\begin{rem} If one knew that $\snum = 2$ and $c > a,b$, one could have solved for $c$ using the above argument. One requires
\[ a(n+1-a) = a(c-a) + (n+1-c)(c-b). \]
The solutions to this quadratic equation are $c = a+b$ and the nonsensical $c = n+1$. \end{rem}

%
%
%
%
%

\subsection{Switchback in dihedral types}

Let $(W,S)$ be of rank 2 with $S=\{s,t\}$. 
We have for $J=\{s\}$
\[[J+ k(-s+t-t+s)]=[J+t-t]\]
if $m_{st}=2k+1$ is odd, where the left hand side denotes $[J-s+t-t+s \cdots -s+t-t+s]$ where $(-s+t-t+s)$ appears $k$ times, and
\[[J+(k-1)(-s+t-t+s)-s+t]=[J+t-s]\] if $m_{st}=2k$ is even.
Note that $m_{st}=\infty$ cannot be the case, since $S$ is then not finitary.

Via the embedding to $\emptyset$ on both sides (Proposition \ref{embedding of expression}), the above relation is exactly the braid relation in the Coxeter group.


\subsection{Switchback in type BC}

Let $(W,S)$ be of type $B_{n+1}$ where the elements in $S$ are indexed according to the Dynkin diagram
\[\begin{tikzpicture}[scale=0.4,baseline=-3]
\protect\draw (4 cm,0) -- (2 cm,0);
\protect\draw (2 cm,0) -- (0 cm,0);
\protect\draw (0 cm,0) -- (-2 cm,0);
\protect\draw (-2 cm,0.1cm) -- (-4 cm,0.1 cm);
\protect\draw (-2 cm,-0.1cm) -- (-4 cm,-0.1 cm);
\protect\draw[fill=white] (4 cm, 0 cm) circle (.15cm) node[above=1pt]{\scriptsize $n$};
\protect\draw[fill=white] (2 cm, 0 cm) circle (0cm) node[above=1pt]{\scriptsize $\cdots$};
\protect\draw[fill=white] (0 cm, 0 cm) circle (.15cm) node[above=1pt]{\scriptsize $2$};
\protect\draw[fill=white] (-2 cm, 0 cm) circle (.15cm) node[above=1pt]{\scriptsize $1$};
\protect\draw[fill=white] (-4 cm, 0 cm) circle (.15cm) node[above=1pt]{\scriptsize $0$};
\end{tikzpicture}.\]
For each $0\leq a< b\leq n$, let $J = S \setminus s_a$. We have
\[ [J+s_a-s_b] = [J-s_c+s_a -s_d +s_c-s_b +s_d] \]
where $c =n+1-b+a$ and $d=a$.
The case $a>b$ is similar and determined from Lemma \ref{lem:ssflipped}. (The case $a=b$, i.e., $s_a = w_0s_bw_0$ is where there is no braid relation. See Proposition~\ref{prop:sss}~\eqref{uniqrex}.)

We follow the algorithm described in Proposition~\ref{prop:sss}~\eqref{eqsss}. We have $w_0s_bw_0=s_b=u_{-1}$, so $s_c = w_Js_bw_J$. Since $b$ belongs to the connected component $\{s_{a+1},\cdots ,s_n\}$ of $J$ of type $A$, we get $c = n+1-b+a$. 
The second step gives $s_d = w_K s_a w_K$ with $K=S\setminus\{s_c\}$. Here, $a$ belongs to the connected component of type $B$, yielding $d=a$.
In the next step, we conjugate $s_c$ by $w_J$ and are back to $s_b$. Thus by Proposition~\ref{prop:sss}, it remains to check that the endpoint of $[J-s_c+s_a -s_d +s_c-s_b +s_d]$ contains $w_0$, or equivalently (see Theorem \ref{thm:betterred}), that $w_J.(w_{J\cap K}w_K).(w_{K\cap J}w_J).(w_{J\cap I}w_I) = w_0$, where $I=S\setminus\{s_b\}$.
The latter is checked using the formula for the length of longest element in a Coxeter group of type $A$ and $B$.

\subsection{Switchback in type D}

Let $(W,S)$ be of type $D_{n+2}$ where $S =\{s_0,s_{\bar{0}},\cdots,s_n\}$ is indexed according to
\[ \begin{tikzpicture}[scale=0.4,baseline=-3]
\protect\draw (4 cm,0) -- (2 cm,0);
\protect\draw (2 cm,0) -- (0 cm,0);
\protect\draw (0 cm,0) -- (-2 cm,0);
\protect\draw (-2 cm,0) -- (-4 cm,0.7 cm);
\protect\draw (-2 cm,0) -- (-4 cm,-0.7 cm);
\protect\draw[fill=white] (4 cm, 0 cm) circle (.15cm) node[above=1pt]{\scriptsize $n$};
\protect\draw[fill=white] (2 cm, 0 cm) circle (0cm) node[above=1pt]{\scriptsize $\cdots$};
\protect\draw[fill=white] (0 cm, 0 cm) circle (.15cm) node[above=1pt]{\scriptsize $2$};
\protect\draw[fill=white] (-2 cm, 0 cm) circle (.15cm) node[above=1pt]{\scriptsize $1$};
\protect\draw[fill=white] (-4 cm, 0.7 cm) circle (.15cm) node[above=1pt]{\scriptsize $0$};
\protect\draw[fill=white] (-4 cm, -0.7 cm) circle (.15cm) node[above=1pt]{\scriptsize $\bar{0}$};
\end{tikzpicture}.\]
We view $a\mapsto \bar{a}$ as an involution on the indices, where $\bar{a}=a$ for $a\geq 1$. If $n$ is odd then this agrees with conjugation by $w_0$. If $n$ is even then conjugation by $w_0$ is the identity on $W_0$.

We have:
\begin{enumerate}
    \item For $a\in\{0,\bar{0}\}$ and either $b = \bar{a}$ (if $n$ is even) or $b=a$ (if $n$ is odd)
    \[ [J+s_a-s_b] = [J-s_{n}+s_a-s_b +s_n] ;\]
    \item 
For $a\in\{0,\bar{0}\}$ and $1\leq b < n$ 
\[ [J+s_a-s_b] = [J-s_{n-b}+s_a -s_c +s_{n-b}-s_b +s_c] ,\]
where $c = \bar{a}$ if $n-b$ is odd and $c=\bar{a}$ if $n-b$ is even.
    \item
For $a\in\{0,\bar{0}\}$
\[ [J+s_a-s_n] = [J-s_{\bar{a}}+s_a-s_n +s_{\bar{a}}] ;\]
    \item
For $1\leq a < b\leq n$
\[ [J+s_a-s_b] = [J-s_{n+1-b+a}+s_a -s_{a} +s_{n+1-b+a}-s_b +s_{a}].\]
\end{enumerate}
To obtain the case where $1 \le b < a \le n$ merely apply Lemma \ref{lem:ssflipped}. The arguments here are analogous to the case of type $B$, only with more cases to consider.

\subsection{Switchback in types EFH}
 
For each Coxeter type, we index the elements in $S$ by integers $1,\cdots,n$, and record for each $1\leq a\leq b \leq n$ (such that $s_a\neq w_0s_bw_0$) the sequence $c_\bullet$ which satisfies the relation 
\[[J +s_a -s_b] = [J-s_{c_1}+s_a -s_{c_2} +s_{c_1} -s_{c_3}+
\cdots -s_{c_{\snum - 1}}+s_{c_{\snum-2}}
-s_b+s_{c_{\snum-1}}]\]
where $J= S\setminus\{s_a\}$. The size of $c_{\bullet}$ is $\snum - 1$. To obtain the switchback when $1 \le b \le a \le n$, merely apply Lemma \ref{lem:ssflipped}. The result was obtained by computer computation. 

\subsubsection{Type $E_6$}
\begin{tikzpicture}[scale=0.4,baseline=-3]
\protect\draw (8 cm,0) -- (6 cm,0);
\protect\draw (6 cm,0) -- (4 cm,0);
\protect\draw (4 cm,0) -- (2 cm,0);
\protect\draw (2 cm,0) -- (0 cm,0);
\protect\draw (4 cm,0) -- (4 cm,1.5 cm);

\protect\draw[fill=white] (8 cm, 0 cm) circle (.15cm) node[below=1pt]{\scriptsize $6$};
\protect\draw[fill=white] (6 cm, 0 cm) circle (.15cm) node[below=1pt]{\scriptsize $5$};
\protect\draw[fill=white] (4 cm, 0 cm) circle (.15cm) node[below=1pt]{\scriptsize $4$};
\protect\draw[fill=white] (2 cm, 0 cm) circle (.15cm) node[below=1pt]{\scriptsize $3$};
\protect\draw[fill=white] (4 cm, 1.5 cm) circle (.15cm) node[right=1pt]{\scriptsize $2$};
\protect\draw[fill=white] (0 cm, 0 cm) circle (.15cm) node[below=1pt]{\scriptsize $1$};
\end{tikzpicture}

The following is a full list $(a,b:c_\bullet)$ where $c_\bullet$ are displayed without commas.


\[
\begin{tabular}{|l|l|l|} \hline
 $a$ & $b$ & $c_\bullet$ \\ \hline\hline
 \begin{tabular}{@{}l@{}}
 1 \\ 1 \\ 1 \\ 1 \\ 1  \\6  \end{tabular} & \begin{tabular}{@{}l@{}} 1\\ 2 \\3 \\ 4 \\ 5 \\6 \end{tabular} & \begin{tabular}{@{}l@{}} 6 \\ 31 \\ 52 \\435 \\26\\1 \end{tabular} \\ \hline
\end{tabular}
\qquad
\begin{tabular}{|l|l|l|} \hline
 $a$ & $b$ & $c_\bullet$ \\ \hline\hline
 \begin{tabular}{@{}l@{}}
2 \\2 \\2 \\ 2 \\ 5\\5   \end{tabular} & \begin{tabular}{@{}l@{}}  3 \\ 4 \\ 5 \\ 6 \\ 5\\6 \end{tabular} & \begin{tabular}{@{}l@{}} 36 \\ 4242 \\ 51 \\ 65 \\464 \\23 \end{tabular} \\ \hline
\end{tabular}
\qquad
\begin{tabular}{|l|l|l|} \hline
 $a$ & $b$ & $c_\bullet$ \\ \hline\hline
 \begin{tabular}{@{}l@{}}
3\\3 \\3 \\4 \\ 4    \end{tabular} & \begin{tabular}{@{}l@{}}  3\\4 \\ 6 \\ 5 \\ 6  \end{tabular} & \begin{tabular}{@{}l@{}}  414\\546 \\ 12 \\ 143 \\ 354  \end{tabular} \\ \hline
\end{tabular}
\]

\subsubsection{Type $E_7$}
 \begin{tikzpicture}[scale=0.4,baseline=-3]
\protect\draw (10 cm,0) -- (8 cm,0);
\protect\draw (8 cm,0) -- (6 cm,0);
\protect\draw (6 cm,0) -- (4 cm,0);
\protect\draw (4 cm,0) -- (2 cm,0);
\protect\draw (2 cm,0) -- (0 cm,0);
\protect\draw (4 cm,0) -- (4 cm,1.5 cm);

\protect\draw[fill=white] (10 cm, 0 cm) circle (.15cm) node[below=1pt]{\scriptsize $7$};
\protect\draw[fill=white] (8 cm, 0 cm) circle (.15cm) node[below=1pt]{\scriptsize $6$};
\protect\draw[fill=white] (6 cm, 0 cm) circle (.15cm) node[below=1pt]{\scriptsize $5$};
\protect\draw[fill=white] (4 cm, 0 cm) circle (.15cm) node[below=1pt]{\scriptsize $4$};
\protect\draw[fill=white] (2 cm, 0 cm) circle (.15cm) node[below=1pt]{\scriptsize $3$};
\protect\draw[fill=white] (4 cm, 1.5 cm) circle (.15cm) node[right=1pt]{\scriptsize $2$};
\protect\draw[fill=white] (0 cm, 0 cm) circle (.15cm) node[below=1pt]{\scriptsize $1$};
\end{tikzpicture}

\[
\begin{tabular}{|l|l|l|} \hline
 $a$ & $b$ & $c_\bullet$ \\ \hline\hline
 \begin{tabular}{@{}l@{}}
 1 \\ 1 \\ 1 \\ 1 \\ 1 \\ 1 \\ 6   \end{tabular} & \begin{tabular}{@{}l@{}} 2\\ 3 \\4 \\ 5 \\ 6 \\ 7 \\ 7  \end{tabular} & \begin{tabular}{@{}l@{}} 27 \\ 3131 \\ 4363 \\ 5242 \\ 61 \\ 76 \\ 71 \end{tabular} \\ \hline
\end{tabular}
\qquad
\begin{tabular}{|l|l|l|} \hline
 $a$ & $b$ & $c_\bullet$ \\ \hline\hline
 \begin{tabular}{@{}l@{}}
2 \\ 2 \\ 2 \\ 2 \\ 2 \\ 5 \\ 5   \end{tabular} & \begin{tabular}{@{}l@{}}  3 \\ 4 \\ 5 \\ 6 \\ 7 \\ 6 \\ 7 \end{tabular} & \begin{tabular}{@{}l@{}} 657 \\ 5152 \\ 4251 \\ 375 \\ 12 \\ 732 \\ 623 \end{tabular} \\ \hline
\end{tabular}
\qquad
\begin{tabular}{|l|l|l|} \hline
 $a$ & $b$ & $c_\bullet$ \\ \hline\hline
 \begin{tabular}{@{}l@{}}
3 \\ 3 \\ 3 \\ 3 \\ 4 \\ 4 \\ 4   \end{tabular} & \begin{tabular}{@{}l@{}}  4 \\ 5 \\ 6 \\ 7 \\ 5\\ 6 \\ 7 \end{tabular} & \begin{tabular}{@{}l@{}} 6341 \\ 5474 \\ 4143 \\ 265 \\ 7453 \\ 6464 \\ 5354 \end{tabular} \\ \hline
\end{tabular}
\]


Type $E_7$ gives many opportunities to see Lemma \ref{lem:rotation} in action.

\begin{ex} The entry $a = 1$ and $b=5$ gives the rotation sequence $(1,5,2,4,2,5)$ which ends up being $6$-periodic because conjugation by $w_0$ is trivial, see \eqref{addd}. The entry $a = 2$ and $b=5$ gives the rotation sequence $(2,4,2,5,1,5)$, which is the rotation of the previous. The entry $a = 2$ and $b=4$ gives the third rotation, $(2,5,1,5,2,4)$. There are three other rotations as well, though they are not pictured in the table above because $a > b$. For example, the rotation sequence for $a = 5$ and $b=2$ is $(5,1,5,2,4,2)$, which is both the reverse of the sequence for $a = 2$ and $b=5$, and its rotation. \end{ex}

\begin{ex} The rotation sequence for $a = 2$ and $b=3$ is $(2,6,5,7,3)$, and is $5$-periodic. When $a = 5$ and $b = 6$ we get the rotation $(5,7,3,2,6)$. Another rotation
comes from $a = 6$ and $b=2$, which is not pictured in the table above. Applying Lemma \ref{lem:ssflipped} to the case of $a = 2$ and $b=6$, we see that the case of $a = 6$ and
$b=2$ has rotation sequence $(6,5,7,3,2)$ as expected. The five rotations of $a = 2$ and $b=6$ are related to the five rotations of $a = 6$ and $b=2$ by taking the reversed
sequence. \end{ex}

\begin{ex} Finally, the rotation sequence of $a=4$ and $b=6$ is $(4,6,4,6,4,6)$, which is $2$-periodic. It both rotates and flips to give the rotation sequence of $a=6$ and $b=4$. This gives an example where the periodicity is strictly less than $\snum + 1$. \end{ex}

\subsubsection{Type $E_8$} \label{sss:E8}

 \begin{tikzpicture}[scale=0.4,baseline=-3]
\protect\draw (12 cm,0) -- (10 cm,0);
\protect\draw (10 cm,0) -- (8 cm,0);
\protect\draw (8 cm,0) -- (6 cm,0);
\protect\draw (6 cm,0) -- (4 cm,0);
\protect\draw (4 cm,0) -- (2 cm,0);
\protect\draw (2 cm,0) -- (0 cm,0);
\protect\draw (4 cm,0) -- (4 cm,1.5 cm);

\protect\draw[fill=white] (12 cm, 0 cm) circle (.15cm) node[below=1pt]{\scriptsize $8$};
\protect\draw[fill=white] (10 cm, 0 cm) circle (.15cm) node[below=1pt]{\scriptsize $7$};
\protect\draw[fill=white] (8 cm, 0 cm) circle (.15cm) node[below=1pt]{\scriptsize $6$};
\protect\draw[fill=white] (6 cm, 0 cm) circle (.15cm) node[below=1pt]{\scriptsize $5$};
\protect\draw[fill=white] (4 cm, 0 cm) circle (.15cm) node[below=1pt]{\scriptsize $4$};
\protect\draw[fill=white] (2 cm, 0 cm) circle (.15cm) node[below=1pt]{\scriptsize $3$};
\protect\draw[fill=white] (4 cm, 1.5 cm) circle (.15cm) node[right=1pt]{\scriptsize $2$};
\protect\draw[fill=white] (0 cm, 0 cm) circle (.15cm) node[below=1pt]{\scriptsize $1$};
\end{tikzpicture}

\[
\begin{tabular}{|l|l|l|} \hline
 $a$ & $b$ & $c_\bullet$ \\ \hline\hline
 \begin{tabular}{@{}l@{}}
 1 \\ 1 \\ 1 \\ 1 \\ 1 \\ 1 \\ 1    \end{tabular} & \begin{tabular}{@{}l@{}} 2\\ 3 \\4 \\ 5 \\ 6 \\ 7 \\ 8   \end{tabular} & \begin{tabular}{@{}l@{}} 3128 \\ 2821 \\ 437573 \\ 525152 \\ 6161 \\ 7686 \\ 81  \end{tabular} \\ \hline
\end{tabular}
\qquad
\begin{tabular}{|l|l|l|} \hline
 $a$ & $b$ & $c_\bullet$ \\ \hline\hline
 \begin{tabular}{@{}l@{}}
2 \\ 2 \\ 2 \\ 2 \\ 2 \\ 2 \\ 7    \end{tabular} & \begin{tabular}{@{}l@{}}  3 \\ 4 \\ 5 \\ 6 \\ 7  \\ 8 \\ 8 \end{tabular} & \begin{tabular}{@{}l@{}} 7238 \\ 658562 \\ 515251 \\ 426585 \\ 3832 \\ 1312 \\ 8787 \end{tabular} \\ \hline
\end{tabular}
\qquad
\begin{tabular}{|l|l|l|} \hline
 $a$ & $b$ & $c_\bullet$ \\ \hline\hline
 \begin{tabular}{@{}l@{}}
3 \\ 3 \\ 3 \\ 3 \\ 3 \\ 6 \\ 6    \end{tabular} & \begin{tabular}{@{}l@{}}  4 \\ 5 \\ 6 \\ 7  \\ 8  \\ 7 \\ 8  \end{tabular} & \begin{tabular}{@{}l@{}} 757341 \\ 635484 \\ 548453 \\ 414375 \\ 2723  \\ 8671 \\ 7176  \end{tabular} \\ \hline
\end{tabular}
\qquad
\begin{tabular}{|l|l|l|} \hline
 $a$ & $b$ & $c_\bullet$ \\ \hline\hline
 \begin{tabular}{@{}l@{}}
4 \\ 4 \\ 4 \\ 4 \\ 5 \\ 5 \\ 5   \end{tabular} & \begin{tabular}{@{}l@{}}  5 \\ 6 \\ 7  \\ 8 \\ 6 \\ 7 \\ 8 \end{tabular} & \begin{tabular}{@{}l@{}} 845363 \\ 746474 \\ 647464 \\ 536354 \\ 856242 \\ 734143 \\ 624265 \end{tabular} \\ \hline
\end{tabular}
\]

%

\subsubsection{Type $F_4$}

\begin{tikzpicture}[scale=0.4,baseline=-3]
\protect\draw (6 cm,0) -- (4 cm,0);
\protect\draw (4 cm,0.1cm) -- (2 cm,0.1 cm);
\protect\draw (4 cm,-0.1cm) -- (2 cm,-0.1 cm);
\protect\draw (2 cm,0) -- (0 cm,0);

\protect\draw[fill=white] (6 cm, 0 cm) circle (.15cm) node[above=1pt]{\scriptsize $4$};
\protect\draw[fill=white] (4 cm, 0 cm) circle (.15cm) node[above=1pt]{\scriptsize $3$};
\protect\draw[fill=white] (2 cm, 0 cm) circle (.15cm) node[above=1pt]{\scriptsize $2$};
\protect\draw[fill=white] (0 cm, 0 cm) circle (.15cm) node[above=1pt]{\scriptsize $1$};
\end{tikzpicture}
\begin{tabular}{|l|l|l|} \hline
 $a$ & $b$ & $c_\bullet$ \\ \hline\hline
 \begin{tabular}{@{}l@{}}
 1 \\ 1 \\1 \\ 2 \\2 \\3   \end{tabular} & \begin{tabular}{@{}l@{}} 2\\ 3 \\4 \\ 3\\ 4 \\ 4 \end{tabular} & \begin{tabular}{@{}l@{}} 2121 \\ 3242 \\ 41 \\4231 \\3132 \\4343 \end{tabular} \\ \hline
\end{tabular}

\subsubsection{Type $H_3$}

\begin{tikzpicture}[scale=0.4,baseline=-3]

\protect\draw (4 cm,0cm) -- (2 cm,0 cm);
\protect\draw (4 cm,0.05cm) -- (2 cm,0.05 cm);
\protect\draw (4 cm,-0.05cm) -- (2 cm,-0.05 cm);
\protect\draw (2 cm,0) -- (0 cm,0);

\protect\draw[fill=white] (3 cm, 0 cm) circle (0cm) node[below=0pt]{\tiny $5$};

\protect\draw[fill=white] (4 cm, 0 cm) circle (.15cm) node[above=1pt]{\scriptsize $3$};
\protect\draw[fill=white] (2 cm, 0 cm) circle (.15cm) node[above=1pt]{\scriptsize $2$};
\protect\draw[fill=white] (0 cm, 0 cm) circle (.15cm) node[above=1pt]{\scriptsize $1$};
\end{tikzpicture}
\begin{tabular}{|l|l|l|} \hline
 $a$ & $b$ & $c_\bullet$ \\ \hline\hline
 \begin{tabular}{@{}l@{}}
 1 \\ 1 \\ 2   \end{tabular} & \begin{tabular}{@{}l@{}} 2\\ 3\\ 3 \end{tabular} & \begin{tabular}{@{}l@{}} 3231 \\ 2132 \\ 3121 \end{tabular} \\ \hline
\end{tabular}

\subsubsection{Type $H_4$}

\begin{tikzpicture}[scale=0.4,baseline=-3]
\protect\draw (4 cm,0.05cm) -- (2 cm,0.05 cm);
\protect\draw (4 cm,-0.05cm) -- (2 cm,-0.05 cm);
\protect\draw (4 cm,0cm) -- (2 cm,0 cm) ;
\protect\draw (2 cm,0) -- (0 cm,0);
\protect\draw (0 cm,0) -- (-2 cm,0);

\protect\draw[fill=white] (3 cm, 0 cm) circle (0cm) node[below=0pt]{\tiny $5$};

\protect\draw[fill=white] (4 cm, 0 cm) circle (.15cm) node[above=1pt]{\scriptsize $4$};
\protect\draw[fill=white] (2 cm, 0 cm) circle (.15cm) node[above=1pt]{\scriptsize $3$};
\protect\draw[fill=white] (0 cm, 0 cm) circle (.15cm) node[above=1pt]{\scriptsize $2$};
\protect\draw[fill=white] (-2 cm, 0 cm) circle (.15cm) node[above=1pt]{\scriptsize $1$};
\end{tikzpicture}
\begin{tabular}{|l|l|l|l|} \hline
 $a$ & $b$ & $c_\bullet$ \\ \hline\hline
 \begin{tabular}{@{}l@{}}
 1 \\ 1 \\1 \\ 2 \\2 \\3   \end{tabular} & \begin{tabular}{@{}l@{}} 2\\ 3 \\4 \\ 3\\ 4 \\ 4 \end{tabular} & \begin{tabular}{@{}l@{}} 21212121\\3242313242 \\4341434143  \\4231324231  \\3132423132 \\4143414341 \\ \end{tabular} \\ \hline
\end{tabular}

\section{Type $A$: $\SC$ and webs} \label{sec:typeAwebs}

\subsection{Definition of the diagrammatic category}

In type $A$, one can give a diagrammatic presentation for $\SC$ similar to the presentation of Type $A$ webs from \cite{CKM}. More precisely, we get a presentation of the monoidal
category obtained by gluing $\SC(S_n)$ together for all $n \ge 0$.

\begin{defn} Let $\SCWebAlg$ be the monoidal category defined as follows. The objects are sequences $\un = (n_1, n_2, \ldots, n_k)$ of non-negative integers, and the monoidal structure on objects is concatenation. Let $\sum \un := n_1 + \cdots + n_k$. There are no morphisms from $\un$ to $\um$ if $\sum \un \ne \sum \um$. If $\sum \un = \sum \um = N$, then we view $\un$ as a subset in $S$ consisting of the simple generators of the parabolic subgroup $S_{\un} := S_{n_1} \times S_{n_2} \times \cdots \times S_{n_k}$ inside $S_N$, and similarly for $\um$. Then
\[ \Hom_{\SCWebAlg}(\un, \um) = \Hom_{\SC(S_N)}({\un}, {\um})= S_{\um}\setminus S_N / S_{\un},\]
with $\SC(S_N)$ controlling composition as well. The monoidal structure on morphisms is given by the product of parabolic double cosets
\begin{equation*}
    \begin{split}
     S_{\um}\setminus S_N / S_{\un}\times S_{\um'}\setminus S_{N'} / S_{\un'} &\to  S_{\um}\times S_{\um'}\setminus S_N\times S_{N'} / S_{\un}\times S_{\un'} \cong S_{\um\;\um'}\setminus S_{N+N'} / S_{\un\;\un'} \\
            (p,q) &\mapsto p\times q
    \end{split}
\end{equation*} \end{defn}

\begin{lem} The monoidal structure is well-defined. \end{lem}
	
\begin{proof} We need only verify the interchange law. It follows from 
\[(\ma{p}\times\ma{p'})*(\ma{q}\times\ma{q'})=(\ma{p}*\ma{q})\times(\ma{p'}*\ma{q'}).\] \end{proof}

The generating morphisms of $\SC(S_N)$ are maps $I \to Is$ and $Is \to I$ coming from the addition or subtraction of a single simple reflection. The map $I \to Is$ corresponds to a morphism  of the form
\[ S_{n_1} \times \cdots \times S_a \times S_b \times \cdots \times S_{n_k}  \to S_{n_1} \times \cdots \times S_{a+b} \times \cdots \times S_{n_k}. \]
It can be obtained by taking the morphism
\[ S_a \times S_b \to S_{a+b} \]
inside $\SC(S_{a+b})$, and tensoring on the right and left with identity maps. Thus the generators of $\SCWebAlg$ as a \emph{monoidal} category should be the maps
\[ S_a \times S_b \to S_{a+b}, \qquad S_{a+b} \to S_a \times S_b \]
for each $a, b \ge 1$. This motivates the following definition.

\begin{defn}\label{def:scwebdiag} Let $\SCWebDiag$ be the monoidal category defined as follows. The objects are the same as in $\SCWebAlg$. The morphisms are monoidally generated by two
diagrams:
\begin{equation}\label{diagpms} {
\labellist
\tiny\hair 2pt
 \pinlabel {$a$} [ ] at 7 7
 \pinlabel {$b$} [ ] at 32 7
 \pinlabel {$a+b$} [ ] at 27 18
\endlabellist
\centering
\ig{2}{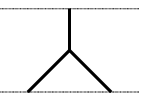}
} \qquad {
\labellist
\tiny\hair 2pt
 \pinlabel {$a$} [ ] at 7 18
 \pinlabel {$b$} [ ] at 32 18
 \pinlabel {$a+b$} [ ] at 27 7
\endlabellist
\centering
\ig{2}{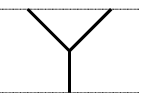}
}. \end{equation} The relations are as follows. When no labels are given, the relation holds for all valid labels.
\begin{subequations} \label{webrelations}
\begin{equation} \label{bigon} \textbf{Bigon:} \quad \ig{1}{bigon} \; = \; \ig{1}{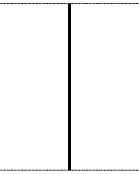} \end{equation}
\begin{equation} \label{assoc} \textbf{Associativity:} \quad \igv{1}{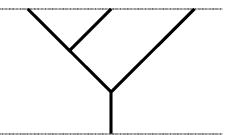} \; = \; \ighv{1}{coassocL} \end{equation}
\begin{equation} \label{coassoc} \textbf{Coassociativity:} \quad \ig{1}{coassocL} \; = \; \igh{1}{coassocL} \end{equation}
\begin{equation} \label{square1} \textbf{Square:} \qquad \qquad {
\labellist
\tiny\hair 2pt
 \pinlabel {$a$} [ ] at 3 6
 \pinlabel {$b$} [ ] at 3 66
 \pinlabel {$N-a$} [ ] at 43 6
 \pinlabel {$N-b$} [ ] at 43 66
 \pinlabel {$b$} [ ] at 20 15
 \pinlabel {$a$} [ ] at 20 60
 \pinlabel {$a+b$} [ ] at -3 39
 \pinlabel {$N-a-b$} [ ] at 50 39
\endlabellist
\centering
\ig{1.2}{square1}
} \qquad \qquad = \qquad {
\labellist
\tiny\hair 2pt
 \pinlabel {$a$} [ ] at 7 12
 \pinlabel {$N-a$} [ ] at 37 12
 \pinlabel {$b$} [ ] at 7 60
 \pinlabel {$N-b$} [ ] at 37 60
 \pinlabel {$N$} [ ] at 27 35
\endlabellist
\centering
\ig{1.2}{mergesplit}
} \quad \text{ if } a+b < N, \end{equation}
\begin{equation} \label{square2} \textbf{Square:} \qquad \qquad {
\labellist
\tiny\hair 2pt
 \pinlabel {$a$} [ ] at 3 6
 \pinlabel {$b$} [ ] at 3 66
 \pinlabel {$N-a$} [ ] at 43 6
 \pinlabel {$N-b$} [ ] at 43 66
 \pinlabel {$N-b$} [ ] at 20 15
 \pinlabel {$N-a$} [ ] at 20 62
 \pinlabel {$a+b-N$} [ ] at -10 39
 \pinlabel {$2N-a-b$} [ ] at 52 39
\endlabellist
\centering
\ig{1.2}{square2}
} \qquad \qquad = \qquad {
\labellist
\tiny\hair 2pt
 \pinlabel {$a$} [ ] at 7 12
 \pinlabel {$N-a$} [ ] at 37 12
 \pinlabel {$b$} [ ] at 7 60
 \pinlabel {$N-b$} [ ] at 37 60
 \pinlabel {$N$} [ ] at 27 35
\endlabellist
\centering
\ig{1.2}{mergesplit}
} \quad \text{ if } a+b > N, \end{equation} 
\end{subequations}
\end{defn}

It is standard practice to also permit webs where edges are labeled with $0$. The object $0$ is treated as the monoidal identity, and can be erased from
a diagram freely. For example, the generating trivalent vertex $m \otimes 0 \to m$ is equal to the identity map of $m$, and the pictures become the same after the $0$-labeled edge
is erased. With this convention, one can also make sense of various relations when they have zero labels. However, no new relations are introduced in this fashion. For example,
\begin{equation} \bigon{m}{m}{0} = {
\labellist
\tiny\hair 2pt
 \pinlabel {$m$} [ ] at 28 13
\endlabellist
\centering
\ig{1}{id}
} \end{equation}
is already true by definition.

\begin{ex} \label{allow0} To given another example where edges are labeled $0$, consider \eqref{square1} and \eqref{square2} where $a+b = N$. Both sides are already equal by definition. We call these the \emph{trivial} square relations. \end{ex}

Let us match up the relations in \eqref{webrelations} with the relations in $\SC$. The bigon relation matches the $*$-quadratic relation $[I - s + s] \expr [I]$. The associativity
relation \eqref{assoc} matches the up-up relation $[I + s + t] \expr [I + t + s]$, and coassociativity \eqref{coassoc} matches the down-down relation. Most interesting are the square relations \eqref{square1} and \eqref{square2}, which match the switchback relation from \eqref{sssA}.

\begin{rem}\label{diagcommute} There are also other instances of the up-up, down-down, and switchback relations which correspond to equalities of diagrams which hold purely from the interchange law. For example,
\[ \ig{1}{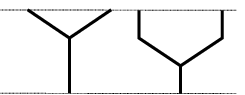} \quad = \quad \igh{1}{splitsplit} \] is an instance of up-up, and 
\[ \ig{1}{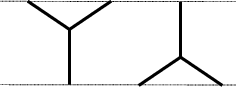} \quad = \quad \ig{1}{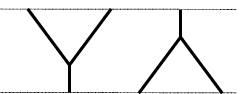} \] is a switchback relation as in Lemma \ref{m=2}. \end{rem}

\begin{prop} The categories $\SCWebAlg$ and $\SCWebDiag$ are monoidally equivalent. \end{prop}

\begin{proof} 
As noted above, the relations given in Definition~\ref{def:scwebdiag} agree with those in \eqref{relationslist} (for the restriction to the objects $\un$ with $\sum\un=N$, for each $N$).
This gives a well-defined full functor $\SCWebDiag\to\SCWebAlg$ which is the identity on objects and maps `$\pm s$' to the generators \eqref{diagpms}. 
It is clearly monoidal (see Remark~\ref{diagcommute} and Lemma~\ref{m=2}) and is faithful by
Theorem~\ref{thm:presentation}. The claim follows.  \end{proof}
	
\begin{rem} We can assign a degree to each diagram, matching the length function of an expression from \S\ref{ss:length}. We have
\begin{equation} \label{degreeofweb} \deg \left( {
\labellist
\tiny\hair 2pt
 \pinlabel {$a$} [ ] at 7 7
 \pinlabel {$b$} [ ] at 32 7
 \pinlabel {$a+b$} [ ] at 27 18
\endlabellist
\centering
\ig{2}{merge}
} \right) = ab, \qquad \deg \left( {
\labellist
\tiny\hair 2pt
 \pinlabel {$a$} [ ] at 7 18
 \pinlabel {$b$} [ ] at 32 18
 \pinlabel {$a+b$} [ ] at 27 7
\endlabellist
\centering
\ig{2}{split}
} \right) = ab. \end{equation}
With the exception of \eqref{bigon}, the relations of \eqref{webrelations} preserve degree. \end{rem}

\subsection{More on squares}

Suppose that $a + b < N$, so that \eqref{square1} holds for $a, b$. Let $a' = N-a$ and $b' = N-b$, so that $a' + b' > N$ and \eqref{square2} holds for $a'$ and $b'$. One can verify easily
that \eqref{square2} for $a', b'$ is none other than the horizontal flip of \eqref{square1} for $a, b$. As a consequence, the relations \eqref{webrelations} are closed under taking the horizontal flip.

For the rest of this section we let $x = N-a$ and $y = N-b$. Given a diagram of the following shape \begin{equation} \label{mergesplittemplate} \mergesplit{a}{b}{x}{y}{N}{.75}, \end{equation} when is it the right-hand side of a (nontrivial) square relation? Precisely when $a + b \ne N$, or equivalently, $b \ne x$. If $b < x$ then we are in the regime of \eqref{square1}, and if $b > x$ then we are in the regime of \eqref{square2}. When $b=x$ we have the picture
\begin{equation} \label{uniqrexweb} \mergesplit{a}{b}{b}{a}{\; \; \; a+b}{.75} \end{equation}
which is the unique reduced expression for its double coset, see Proposition \ref{prop:sss}\eqref{uniqrex}. Note that \eqref{uniqrexweb} appears as a subdiagram inside the left-hand side of \eqref{square1}! This corresponds to the observation in Remark \ref{rmk:cantsimplifyfurther}.

Given a diagram of the following shape \begin{equation} \label{squaretemplate} \squareone{a}{b}{x}{y}{f}{g}{j}{k}{.75}, \end{equation} when is it the left-hand side of a square
relation? Certainly $f = b \iff g = a$ and both imply $a + b \le N$, because $f + a = g + b = j \le N$. If $f = a$ and $a + b = N$ then we have a trivial square relation, see Remark \ref{allow0}. If $f = a$ and $a+b < N$ then we have \eqref{square1}.

Now let us examine diagrams of the form \eqref{squaretemplate} which are not part of a square relation.

In the first case, $f > b$ and $g > a$. Note that $a+b < N$. In this case, we claim that the result is a non-reduced expression for the maximal coset. That is, there is an equality in $\SCWebDiag$ of the form
\begin{equation} \label{nonredsquare} \squareone{a}{b}{x}{y}{>b}{>a}{}{}{.85} \qquad = \qquad \mergesplit{a}{b}{x}{y}{}{.85}, \end{equation}
where only the right-hand side is reduced. We call this the \emph{non-reduced square relation}.
In particular, we should be
able to use the relations \eqref{webrelations}, with at least one use of \eqref{bigon}, to deduce the non-reduced square relation. The reason this works is the opposite to the observation above. Above we said that the subdiagram \eqref{uniqrexweb} inside left-hand side of \eqref{square1} can not be reduced further when $f = b$ and $g=a$. In this case, since $f > b$ and $g > a$, it can be reduced further.

\begin{lem} The equality \eqref{nonredsquare} is a consequence of \eqref{webrelations}. \end{lem}
	
\begin{proof} Continuing the notation of \eqref{squaretemplate}, we assume $f > b$ and $g >a$. 
\begin{equation}
\squareone{a}{b}{x}{y}{f}{g}{}{}{1}	\quad = \quad {
	\labellist
	\tiny\hair 2pt
	 \pinlabel {$a$} [ ] at 0 6
	 \pinlabel {$b$} [ ] at 0 67
	 \pinlabel {$x$} [ ] at 58 6
	 \pinlabel {$y$} [ ] at 58 67
	 \pinlabel {$f$} [ ] at 35 12
	 \pinlabel {$g$} [ ] at 35 65
	 \pinlabel {$b$} [ ] at 14 21
	 \pinlabel {$a$} [ ] at 13 60
	 \pinlabel {$f-b$} [ ] at 35 39
	\endlabellist
	\centering
	\ig{1}{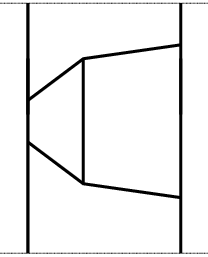}
	} \quad = \quad {
\labellist
\tiny\hair 2pt
 \pinlabel {$a$} [ ] at 0 6
 \pinlabel {$b$} [ ] at 0 67
 \pinlabel {$x$} [ ] at 58 6
 \pinlabel {$y$} [ ] at 58 67
 \pinlabel {$a$} [ ] at 28 62
 \pinlabel {$b$} [ ] at 28 16
 \pinlabel {$f-b$} [ ] at 26 40
\endlabellist
\centering
\ig{1}{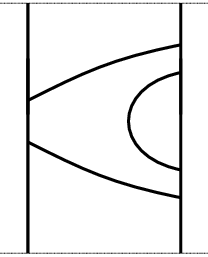}
} \quad \stackrel{\eqref{bigon}}{=} \quad \squareone{a}{b}{x}{y}{b}{a}{}{}{1} \quad \stackrel{\eqref{square1}}{=} \quad \mergesplit{a}{b}{x}{y}{}{1}. \end{equation}
The first equality is \eqref{square1} with bottom boundary $(a,f)$ instead of $(a,x)$. The second equality follows from \eqref{assoc} and \eqref{coassoc}. \end{proof}

In the second case, $f < b$ and $g < a$. It is possible that $a + b < N$ or $a + b > N$ but it will not matter in the end. We claim that the result is a reduced expression, for some non-maximal coset $q$. One can verify this by showing that the lengths add in a product of relative longest elements. An example is demonstrative.

\begin{ex} Let $a = 8$, $b = 5$, $f = 2$, $g=5$, and $N = 12$. Starting with the longest element of $S_8 \times S_4$, we multiply by the element below to get $\ma{q}$.
\[ \ig{1}{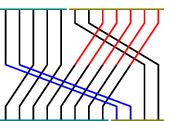} \]
First, $f$ blue strands are pulled left across the initial $a$ strands. Then $g$ red strands are pulled right across $x-f$ remaining strands.
	
However, this coset $q$ also has another reduced expression!
\[ \ig{1}{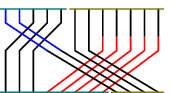} \]
First one pulls $g$ red strands right across the final $x$ strands. Then $f$ blue strands are pulled left across $a-g$ remaining strands. \end{ex}

More generally, there is a equality between two reduced expressions which we might call the \emph{rung swap}, which holds whenever $f < b$:
\begin{equation} \label{rungswap} \squareone{a}{b}{x}{y}{f}{g}{}{}{1} \qquad = \qquad \squaretwo{a}{b}{x}{y}{g}{f}{}{}{1}. \end{equation}
The square relations themselves can be viewed as special instances of the rung swap, when $f = b$ and $g = a$.
	
\begin{lem} The equality \eqref{rungswap} is a consequence of \eqref{webrelations}. \end{lem}

\begin{proof} Both sides of \eqref{rungswap} are equal to
\begin{equation} \label{middleofrungswap} {
\labellist
\tiny\hair 2pt
 \pinlabel {$a$} [ ] at 0 6
 \pinlabel {$b$} [ ] at 0 67
 \pinlabel {$x$} [ ] at 53 6
 \pinlabel {$y$} [ ] at 53 67
 \pinlabel {$g$} [ ] at 17 10
 \pinlabel {$f$} [ ] at 17 63
 \pinlabel {$f$} [ ] at 36 11
 \pinlabel {$g$} [ ] at 36 62
\endlabellist
\centering
\ig{1}{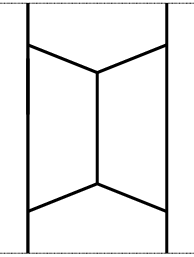}
}. \end{equation}
Starting from the left-hand side of \eqref{rungswap}, and noting that $b > f$, we can apply \eqref{square2} with bottom inputs $(a,f)$ to get \eqref{middleofrungswap}. Similarly, we can apply \eqref{square1} with bottom inputs $(g,x)$ to the right-hand side of \eqref{rungswap}.
\end{proof}

\subsection{Relationship to ordinary type $A$ webs}

This section is written for the reader familiar with type $A$ webs.

In \cite{CKM}, Cautis-Kamnitzer-Morrison construct a $\Z[q,q\inv]$-linear monoidal category $\Webs(\mathfrak{sl}_n)$ which describes morphisms between tensor products of
fundamental representations of $\mathfrak{sl}_n$. Variants on their construction include a category for $\mathfrak{gl}_n$ as well, and a category $\Webs^\infty$ which is the limit
as $n \to \infty$ of $\Webs(\mathfrak{gl}_n)$ (see \cite{QueSar,TVW}). The objects in $\Webs^\infty$ agree with the objects of $\SCWebDiag$, where $1$ represents the
standard representation $V$. By Schur-Weyl duality, the endomorphism ring of $(\C^n)^{\otimes m}$ has the same dimension as $S_m$ when $m \le n$. In $\Webs^\infty$, the
endomorphism ring of $(1,1,\ldots,1)$, having $m$ copies of $1$, always has the same dimension as $S_m$. Meanwhile, in $\SCWebDiag$, the endomorphism algebra of the object
$(1,1,\ldots,1)$ is a set, having the same size as $S_m$. By the same token, other morphism spaces in $\Webs^\infty$ have dimensions equal to the size of the corresponding morphism
spaces in $\SCWebDiag$.

The degree function on webs which we introduce in \eqref{degreeofweb} can be used to equip $\Webs(\mathfrak{gl}_n)$ or $\Webs^\infty$ or (the $\Z$-linearization\footnote{To define the associated graded of a category one must first linearize it. This introduces a zero morphism.} of) $\SCWebDiag$ with a
filtration. A morphism is in degree $\le k$ if it can be represented by a linear combination of webs whose degrees are at most $k$. Since the dimension of an associated graded algebra agrees with the original dimension, the following result might be expected from the discussion of dimensions above.

\begin{thm} The $\Z[q,q\inv]$-linearization of $\SCWebDiag$ and $\Webs^\infty$ have isomorphic associated graded categories. \end{thm}

\begin{proof} Given a category presented by generators and relations, one obtains a presentation of the associated graded category by taking only the top degree terms in each relation. For example, both \eqref{bigon} and the corresponding bigon relation \cite[(2.4)]{CKM} descend to the relation
\begin{equation} \label{bigongr} \ig{1}{bigon} = 0 \end{equation}
in the associated graded. All the remaining (braid) relations of $\SCWebDiag$ are homogeneous, and descend verbatim to the associated graded. The associativity and coassociativity relations \cite[(2.6)]{CKM} are homogeneous as well.

In the associated graded, \eqref{nonredsquare} becomes \begin{equation} \label{nonredsquaregr} \squareone{a}{b}{x}{y}{>b}{>a}{}{}{1} = 0, \end{equation} and \eqref{rungswap} is
unchanged.  We will conclude by proving that the associate graded of the only remaining relations from Cautis-Kamnitzer-Morrison, the square flop relations
\cite[(2.10)]{CKM}, is either \eqref{nonredsquaregr} or the rung swap relation \eqref{rungswap}. With our indexing conventions, the square flop relation has the form
\begin{equation} \label{squareflop} \squareone{a}{b}{x}{y}{f}{g}{j}{k}{1.2} \quad = \sum_{t \ge 0} c_t \qquad \squaretwo{a}{b}{x}{y}{g-t}{f-t}{l}{m}{1.2} \end{equation}
for some coefficients $c_t$.

It it easy to verify that the degree of the left-hand side of \eqref{squareflop} is $f(N-f) + g(N-g)$. The reader should verify that $N+t-f-g = k+l$. Thus
\begin{equation} \label{degreediffflop}  f(N-f) + g(N-g) - \left[ (f-t)(N-(f-t)) + (g-t)(N-(g-t))\right] = 2t(k+l) \ge 0. \end{equation}
In particular, the degree of the left-hand side of \eqref{squareflop} is at least as large as any term on the right-hand side, and only those terms with $2t(k+l)=0$ contribute to the associated graded.

If $f > b$ then the right-hand side of \eqref{squareflop} has no $t=0$ term. If $k \ne 0$ then there is no term with $2t(k+l)=0$, and in the associated graded we get
\eqref{nonredsquaregr}. If $k=0$ then there is a unique term where $2t(k+l)=0$, namely $t = f-b$ and $l=0$. In the associated graded we get the trivial square relation.

If $f < b$ then the right-hand side of \eqref{squareflop} has a $t=0$ term, which matches the degree of the left-hand side. Moreover, there are no terms where $l = 0$, so no other
terms with $2t(k+l) = 0$. Thus in the associated graded we get the relation \eqref{rungswap}, but with some coefficient $c_0$. Thankfully, $c_0 = 1$, so we get \eqref{rungswap} with no scalars.\end{proof}

This associated graded category is the singular analogue of the nilCoxeter algebra, where $s^2 = 0$ (rather than $s^2 = 1$ as in $W$, or $s^2 = s$ as in $(W,*,S)$). For the rest of this section we paint with broad brush strokes, in order to state a rough conjecture.

In \cite[Definition 2.37]{ELLCC}, a basis for morphisms in $\Webs^\infty$ is constructed, the so-called \emph{double ladders basis}. We wish to conjecture that double ladders are
reduced expressions, and provide an enumeration of the morphisms in $\SCWebDiag$, though this statement requires some clarification. Double ladders are constructed by gluing
together elementary light ladders (see \cite[(2.17)]{ELLCC}) and neutral ladders (see \cite[p22]{ELLCC}), and these neutral ladders can be placed willy-nilly. In $\Webs^\infty$,
neutral ladders are isomorphisms modulo ``lower weights,'' and composing with them will alter the basis by an upper-triangular change of basis matrix. In $\SCWebDiag$ however,
neutral ladders have positive degree and can not be placed arbitrarily within a reduced expression.

The construction \cite[(2.17)]{ELLCC} explains how to construct an elementary light ladder when the input sequences are sorted (i.e. when the condition $y_1 < x_1 < \ldots$ is
satisfied). When the input sequences are not sorted, the procedure from \cite{ELLCC} is to apply neutral ladders to sort the sequences (and also to sort any extraneous
strands out of the picture), and then apply the elementary light ladder for sorted sequences. However, the resulting web can often be simplified modulo lower terms, and this
simplified web will have lower degree. We give an example below. The combinatorics of this minimal-degree representative of the light ladder (modulo lower terms) has not yet been
developed. We conjecture that, once they are well-defined, the minimal-degree double ladders will be reduced expressions.

\begin{ex} The (sorted) elementary light ladder giving a morphism $(2,2) \to (1,3)$ is
\[ \rungNE{2}{2}{1}{3}{1}{1}. \]
The construction in \cite{ELLCC} gives the (unsorted) light ladder $(2,2) \to (3,1)$ below, but it can be simplified modulo lower weights to a lower degree diagram, which is reduced.
\[ \squaretwo{2}{3}{2}{1}{1}{2}{1}{3}{1} \qquad \rightsquigarrow \qquad \rungNW{2}{2}{3}{1}{1}{1} \]
The construction in \cite{ELLCC} gives the (unsorted) light ladder $(2,5,2) \to (1,3,5)$ below, but it can be simplified modulo lower weights to a lower degree diagram, which is reduced.
\[ {
\labellist
\tiny\hair 2pt
 \pinlabel {$2$} [ ] at 0 5
 \pinlabel {$5$} [ ] at 23 5
 \pinlabel {$2$} [ ] at 40 5
 \pinlabel {$5$} [ ] at 0 27
 \pinlabel {$2$} [ ] at 23 19
 \pinlabel {$1$} [ ] at 17 37
 \pinlabel {$3$} [ ] at 40 45
 \pinlabel {$1$} [ ] at 0 63
 \pinlabel {$3$} [ ] at 17 63
 \pinlabel {$5$} [ ] at 40 63
 \pinlabel {$3$} [ ] at 11 19
 \pinlabel {$4$} [ ] at 11 46
 \pinlabel {$1$} [ ] at 29 34
 \pinlabel {$2$} [ ] at 29 62
\endlabellist
\centering
\ig{1.1}{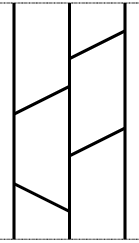}
} \quad \rightsquigarrow \quad {
\labellist
\tiny\hair 2pt
 \pinlabel {$2$} [ ] at 0 5
 \pinlabel {$5$} [ ] at 16 5
 \pinlabel {$2$} [ ] at 39 5
 \pinlabel {$1$} [ ] at 0 21
 \pinlabel {$3$} [ ] at 16 21
 \pinlabel {$5$} [ ] at 39 21
 \pinlabel {$1$} [ ] at 10 16
 \pinlabel {$3$} [ ] at 26 16
\endlabellist
\centering
\ig{1.1}{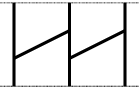}
} \]
\end{ex}

\section{Singular Coxeter complexes} \label{sec:coxcomplex}

The dual Coxeter complex is an important combinatorial CW complex, and we would like to suggest the existence of a singular analogue. We are able to construct the $2$-skeleton, but
do not know the combinatorial structure of the higher cells.

\subsection{Completed dual Coxeter complexes}

We will be brief in recalling the classical definitions, as \cite{AbBr} is an excellent reference (see \S 3 and \S 12.3). Fix a Coxeter system $(W,S)$ where $n$ is the size of $S$. There is a well-known $(n-1)$-dimensional simplicial complex called the Coxeter complex. Its dual is a CW complex or polytope known as the dual Coxeter complex. By adding in an $n$-cell when $W$ is finite, one obtains the \emph{(completed) dual Coxeter complex} $\Cox(W)$, which is always contractible. The completed dual Coxeter complex in type $A$ is
called the \emph{(solid) permutohedron}, which the reader can easily find pictures of online.

The complex $\Cox(W)$ has one $0$-cell for each $w \in W$, each of which is attached to $n$ $1$-cells. The $0$-cell associated to $w$ is connected to the $0$-cell for $ws$ by a
$1$-cell labeled by $s$. More generally, for each $0 \le k \le n$ there is a $k$-cell associated to each coset in $W/W_I$ for each finitary subset $I \subset S$ of size $k$. The
$0$-cells in the closure of this $k$-cell correspond to the elements in this coset.

Let $\hg$ denote the reflection representation of $W$ over $\R$, which has a hyperplane $H_{\alpha}$ for each root $\alpha$, and a chamber for each $w \in W$. When $W$ is finite,
the Coxeter complex can be obtained by intersecting this hyperplane arrangement with the unit sphere in $\hg$. For example, a chamber intersects the unit sphere in a simplex, while
a wall intersects the unit sphere in a face of a simplex. The hyperplane interpretation equips the Coxeter complex and its dual with extra structure. For example, we know what it
means for two facets to lie in the same flat (see \S\ref{subsec:coxeterintro}) of the hyperplane arrangement. Viewing $\Cox(W)$ as a purely topological structure forgets this inherent
``linear'' structure, which will be significant for us.

The $1$-skeleton of $\Cox(W)$ is the covering diagram of the weak right Bruhat order. In particular, one can orient each $1$-cell, so that $w \to ws$ when
$\ell(w) < \ell(ws)$. This orientation is a phenomenon special\footnote{In type $A$ one can place semiorientations on the higher cells, related to the higher Bruhat orders of
Manin-Schechtmann \cite{ManSch}.} to the $1$-skeleton. Oriented paths from the source to $w$ can be interpreted as reduced expressions for $w$, and any two such paths are related by
the braid relations. In fact, the $2$-cells precisely correspond to such braid relations.

\subsection{The singular $2$-skeleton: one-sided cosets}

Let us construct an oriented graph, which we wish to consider as the $1$-skeleton of a singular analogue of the completed dual Coxeter complex. For lack of a better name, we will
call it the \emph{singular Coxeter complex}. The example of type $A_2$ was given in \S\ref{subsec:coxeterintro}. Before we introduce it, let us discuss what we will \emph{not} do.

\begin{defn} Let $\Cox_{\cube}$ denote the CW complex constructed in \cite[Proposition 12.63]{AbBr}. We only describe the $1$-skeleton. There is one vertex for each $(\mt,I)$-coset
$p$. Whenever $I' = Is$ and the $(\mt,I)$-coset $p$ is contained in the $(\mt,I')$-coset $q$, there is an edge from $p$ to $q$. \end{defn}

The complex $\Cox_{\cube}$ is a refinement of $\Cox(W)$ and is homeomorphic to it. The complex is \emph{cubical} (see \cite[Proposition 12.63]{AbBr}) in that the closure of each
top-dimensional cell is a cube with its standard CW structure. One can interpret $\Cox_{\cube}$ within $\hg$ as follows. There is one point on each facet of the hyperplane
arrangement for which the stabilizer is a finite parabolic subgroup. All edges occur between facets which differ in dimension by $1$, where one is contained in the closure of the
other.

We think of $\Cox_{\cube}$ as being relatively boring. The $2$-cells are unrelated to the braid relations. Instead, we throw away those edges not involved in any reduced expression, and get a different CW complex, where the $2$-cells correspond to braid relations.

\begin{defn} Let $J = \mt$. Let $\Cox^1_{\mt}(W,S)$ denote the oriented graph defined as follows. There is one vertex for each $(J,I)$-coset $p$, for all finitary $I \subset S$.
There is an edge $p \to q$ if $[p,q]$ is reduced, see Definition \ref{defn:pqred}. This can be made into a graded graph, where the length of each double coset is defined as in Definition \ref{def:lengths}. \end{defn}

\begin{ex} Here is $\Cox^2_{\mt}$ in type $B_2$. The orientation goes from bottom to top.
\begin{equation}\label{B2cox} \ig{1}{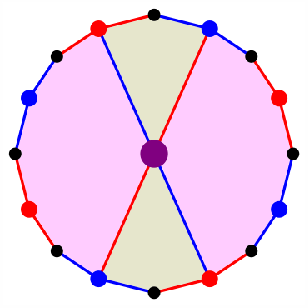} \end{equation} \end{ex}

\begin{subequations}
\begin{ex} \label{ex:coxA3} Let us describe $\Cox^1_{\mt}$ in type $A_3$ \footnote{This nice depiction is due to Janelle Currey.}. Recall that
the ordinary Coxeter complex of $S_4$ is a barycentrically-subdivided tetrahedron. Cut this tetrahedron along all the edges adjacent to one vertex, and unfold it to make it flat.
The result is a large equilateral triangle built from four smaller equilateral triangles (the original faces of the tetrahedron), which is what we draw below. 
\begin{equation} \label{A3cox} {
\labellist
\tiny\hair 2pt
 \pinlabel {$e$} [ ] at 56 39
 \pinlabel {$t$} [ ] at 46 53
 \pinlabel {$s$} [ ] at 96 38
 \pinlabel {$st$} [ ] at 106 53
 \pinlabel {$ts$} [ ] at 66 88
 \pinlabel {$sts$} [ ] at 84 88
 \pinlabel {$stsu$} [ ] at 102 94
 \pinlabel {$stu$} [ ] at 125 59
 \pinlabel {$stut$} [ ] at 146 59
 \pinlabel {$w_0 s$} [ ] at 114 106
 \pinlabel {$w_0$} [ ] at 156 105
 \pinlabel {$w_0 t$} [ ] at 167 94
 \pinlabel {$usts$} [ ] at 188 87
 \pinlabel {$ust$} [ ] at 165 51
 \pinlabel {$su$} [ ] at 176 39
 \pinlabel {$u$} [ ] at 215 39
 \pinlabel {$ut$} [ ] at 227 51
 \pinlabel {$uts$} [ ] at 204 87
 \pinlabel {$w_0 u$} [ ] at 157 124
 \pinlabel {$tuts$} [ ] at 166 137
 \pinlabel {$tut$} [ ] at 144 172
 \pinlabel {$tu$} [ ] at 125 172
 \pinlabel {$tsu$} [ ] at 104 137
 \pinlabel {$tsut$} [ ] at 114 123
 \pinlabel {$\mt$} [ ] at 238 197
 \pinlabel {$s$} [ ] at 238 189
 \pinlabel {$t$} [ ] at 238 181
 \pinlabel {$u$} [ ] at 238 172
 \pinlabel {$st$} [ ] at 238 160
 \pinlabel {$su$} [ ] at 238 148
 \pinlabel {$tu$} [ ] at 238 136
 \pinlabel {$stu$} [ ] at 238 122
\endlabellist
\centering
\ig{1.5}{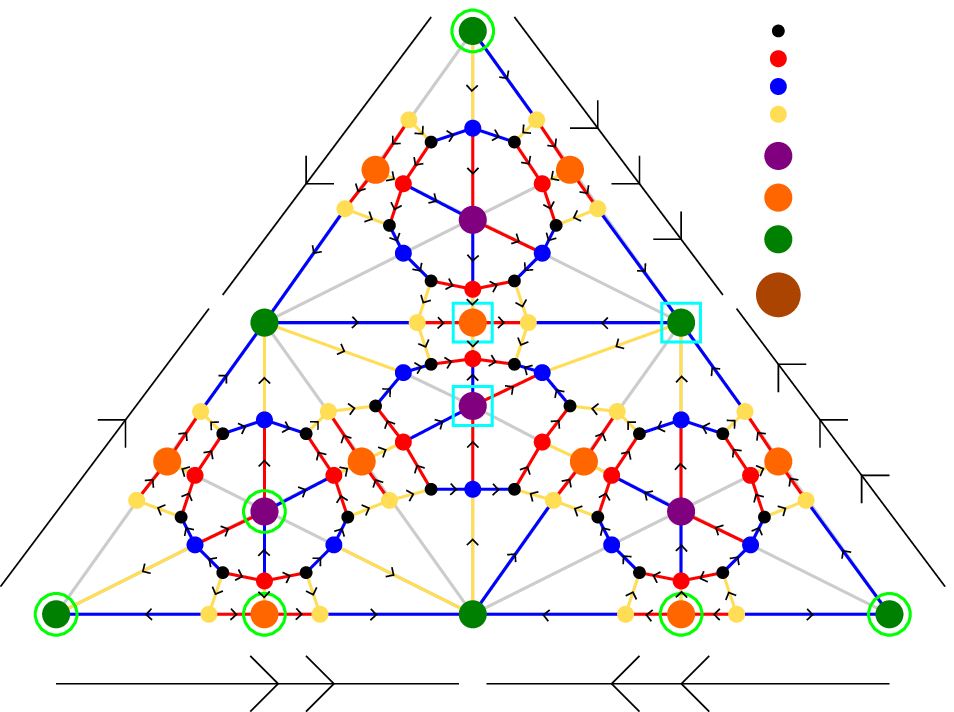}
}
\end{equation}
One can recover the
original tetrahedron by gluing back together the cut edges as indicated (one arrow to one arrow, two arrows to two arrows, etc). We label only the black vertices, corresponding to  $(\mt,\mt)$-cosets. The labels on every other coset are determined by which black vertices are nearby. The singular Coxeter complex also has a vertex
corresponding to the coset $W / W$, the \emph{generic point}. We do not draw the generic point, which fills in the tetrahedron, but we circle the three vertices with edges to the generic point, and square the three vertices with edges from the generic point. The gray edges\footnote{The reader can observe that, with the gray edges included, every pictured face is a square. Including the generic point and the gray edges to it, every $3$-cell is a cube, as expected in $\Cox_{\cube}$.} should be deleted, as they correspond to edges in $\Cox_{\cube}$ which do not appear in $\Cox_{\mt}$. For the $2$-cells, see Example \ref{ex:coxA3redux}.
\end{ex}

If $p$ is an $(\mt,I)$-double coset, then a reduced expression for $p$ is the same data as an oriented path to $p$ from the $(\mt, \mt)$-coset $\{e\}$. Any two such reduced
expressions are related by the braid relations (Theorem \ref{thm:matsumoto}), and we can use these as our $2$-cells.

\begin{defn} (Loosely worded) Let $\Cox^2_{\mt}(W,S)$ be the $2$-dimensional CW complex defined by adding $2$-cells to $\Cox^1_{\mt}(W,S)$. These $2$-cells are attached along the
braid relations, for any application of a braid relation within any reduced expression. \end{defn}

\begin{ex}\label{ex:coxA3redux} Let us continue Example \ref{ex:coxA3} by indicating the $2$-cells. The duller colors (gray, tan, slate) represent $\pm$-associativity relations. The brighter colors (fuschia, lemon, lime) represent switchback relations. There are six $2$-cells we could not draw, the switchback relations involving the generic point. These $2$-cells are attached to the pictured region along a path of length four, from a circled vertex to a squared vertex. We've indicated these six paths with thick brown or blue lines.
\begin{equation} \ig{1}{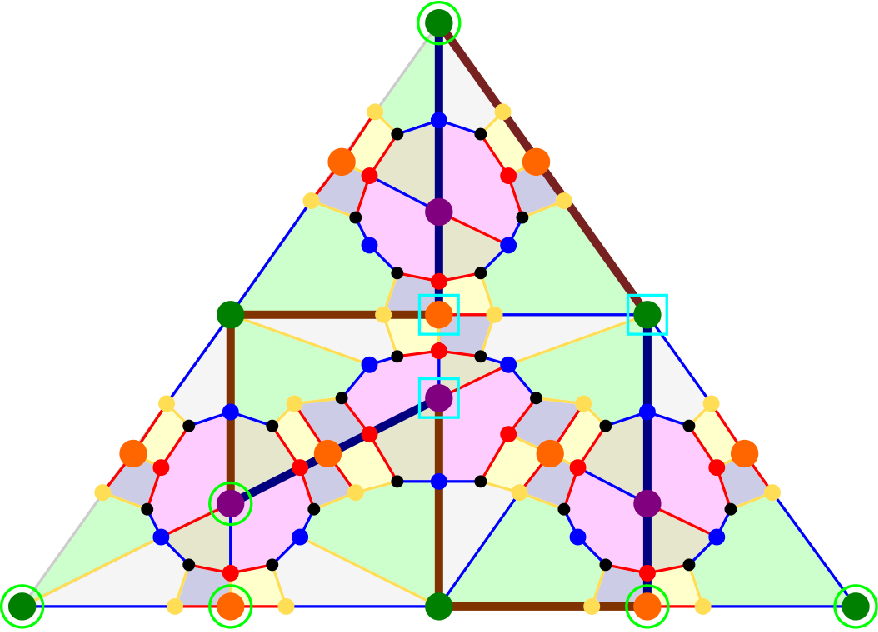} \end{equation}
\end{ex}

\begin{rem} Because $\Cox^1_{\mt}$ is obtained from $\Cox^1_{\cube}$ by throwing away some edges, it seems plausible that there is some formal way to adjoint higher cells (gluing
together various cells of $\Cox_{\cube}$) to construct a CW complex $\Cox_{\mt}$ which is homeomorphic to $\Cox$, but we do not know how. Formal constructions aside, initial
explorations reveal that there is interesting combinatorics involved in describing the attaching maps of the higher cells explicitly. \end{rem}

\subsection{The singular $2$-skeleton: general case}

Let us construct the analogous $2$-dimensional complex for an arbitrary finitary subset $J \subset S$.

\begin{defn} Let $J \subset S$ be finitary. Let $\Cox^1_J(W,S)$ denote the oriented graph defined as follows. There is one vertex for each $(J,I)$-coset $p$, for all finitary $I
\subset S$. There is an edge $p \to q$ if $[p,q]$ is reduced, see Definition \ref{defn:pqred}. This can be made into a graded graph, where the length of each double coset is
defined as in Definition \ref{def:lengths}. Let $\Cox^2_J(W,S)$ be the $2$-dimensional CW complex defined by adding $2$-cells to $\Cox^1_J(W,S)$. These $2$-cells are attached along
the braid relations, for any application of a braid relation within any reduced expression. \end{defn}
	
\begin{ex} Here are the singular Coxeter complexes in type $A_3$ with $J$ equal to $s$, $t$, and $st$ respectively. We've labeled the source $a$ (the $(J,J)$-coset containing $e$) and the sink $z$ (the $(J,\mt)$-coset containing $w_0$); the orientation of arrows is the same as in \eqref{A3cox}. In $\Cox_{stu}$ the source $a$ is the undrawn generic vertex.

\begin{equation} \label{coxA3t} \Cox_s := {
\labellist
\small\hair 2pt
 \pinlabel {$a$} [ ] at 0 36
 \pinlabel {$z$} [ ] at 84 101
\endlabellist
\centering
\ig{1}{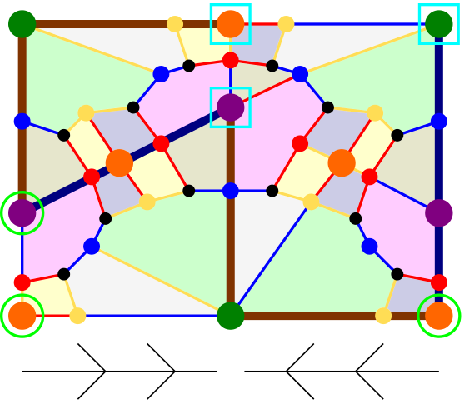}
}, \qquad \Cox_t := {
\labellist
\small\hair 2pt
 \pinlabel {$a$} [ ] at 50 19
 \pinlabel {$z$} [ ] at 144 83
\endlabellist
\centering
\ig{1}{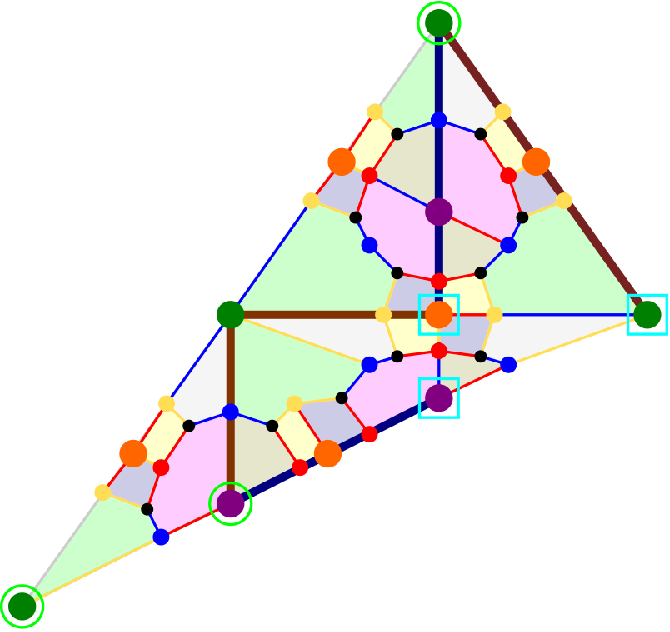}
}, \end{equation}
\begin{equation} \Cox_{st} := {
\labellist
\small\hair 2pt
 \pinlabel {$a$} [ ] at 16 5
 \pinlabel {$z$} [ ] at 84 54
\endlabellist
\centering
\ig{1}{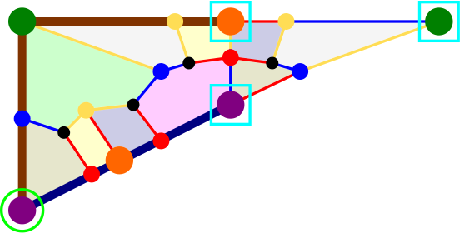}
}, \qquad \Cox_{stu} := {
\labellist
\small\hair 2pt
 \pinlabel {$z$} [ ] at 23 22
\endlabellist
\centering
\ig{1}{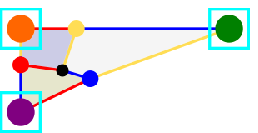}
}. \end{equation}\end{ex}

\end{subequations}

All these examples above have one thing in common: the complex $\Cox^2_J(W,S)$ looks a portion of $\Cox_{\mt}(W,S)$, though with the labels changed! Inside $\hg$, we interpret $\Cox_J$ as the portion of $\Cox_{\mt}$ living within a fundamental domain for the left action of $W_J$.

\begin{prop} \label{singcoxembeds} Let $X \subset \Cox^2_{\mt}(W,S)$ denote the sub-CW-complex defined as follows. It contains all the $0$-cells associated to $(\mt,I)$-cosets $p$
such that $J \subset \leftdes(\ma{p})$. It contains all the $k$-cells whose closure only contains these $0$-cells. Then there is a cell-preserving homeomorphism $X \to
\Cox^2_J(W,S)$, which on $0$-cells sends each $(\mt,I)$-coset $p$ to the $(J,I)$-coset $q$ with $\ma{p} = \ma{q}$. \end{prop}

\begin{proof} That there is some (unique) $q$ with $\ma{q} = \ma{p}$ is immediate from \eqref{beingmaximal}. By Proposition \ref{embedding of expression}, we can take any reduced
expression of $q$ and compose it with $[[\mt,J]]$, to obtain a reduced expression of $p$. An implication is that $[p,p']$ is reduced if and only if $[q,q']$ is reduced. Said
another way, every edge in $\Cox^1_J$ appears along some reduced expression, and we can send it to the corresponding edge in the corresponding reduced expression in $\Cox^1_{\mt}$.
We can match up the $2$-cells (braid relations between reduced expressions) using the same technique. \end{proof}

\subsection{Reduced expressions and hyperplanes} \label{moregeometry}

Let $H$ be a root hyperplane in $\hg$. Then $H$ divides the real vector space $\hg$ into two halves, the \emph{negative half} $H_-$ containing the chamber associated to $e$, and
the \emph{positive half} $H_+$. It is well-known that ordinary reduced expressions are paths between chambers which always cross walls from the negative side to the positive side.
Each wall that is crossed is an ``inversion'' which contributes to the length of the element being expressed. If $x$ and $xs$ are the chambers on opposite sides of a facet inside
$H$, and $x < xs$, then $x \in H_-$ and $xs \in H_+$. Now we reinterpret these ideas in $\Cox_{\mt}$, where we do not just walk between
chambers but also walk on the walls. A given double coset will either lie within $H_-$, within $H_+$, or on $H$ itself.

\begin{prop} \label{thm:inandout} Let $I' = Is$ and $J$ be finitary subsets of $S$. Let $p$ be a $(J,I)$-coset and $q$ a $(J,I')$-coset with $p \subset q$. Then $[p,q]$ is reduced
if and only if $p \notin H_+$ for every root hyperplane $H$ containing $q$. Similarly, $[q,p]$ is reduced if and only if $p \notin H_-$ for every root hyperplane $H$ containing
$q$. \end{prop}

\begin{proof} If we can prove this result for $J = \mt$, then it will hold for arbitrary $J$ using Proposition~\ref{singcoxembeds}. So assume $J = \mt$. Now $[p,q]$ is reduced if and only if $\mi{p} = \mi{q}$.

Suppose $\mi{p} = \mi{q}$. Since $p$ is a facet adjacent to the chamber containing $\mi{p}$, it is contained in the closure of that chamber. If $q$ lies on $H$ then the chamber of
$\mi{q}$ lies in $H_-$ (because $\rightdes(\mi{q}) \cap I' = \mt$, and this facet of $H$ corresponds to right multiplication by some $t \in I'$). Hence $p$ is in the closure of
$H_-$, and is disjoint from $H_+$.

Conversely, suppose that $p \notin H_+$ for any $H$ containing $q$. Then $\mi{p} \in H_-$, and hence $\rightdes(\mi{p}) \cap I' = \mt$. Thus by
\eqref{beingminimal}, $\mi{p} = \mi{q}$.

Similarly, $[q,p]$ is reduced if and only if $\ma{q} = \ma{p}$. The arguments in this case are nearly identical. \end{proof}

Now let us give some intuition why Proposition~\ref{thm:inandout} should match with Definition \ref{defn:pqred} when $J \ne \mt$. The geometric version of the condition \eqref{sameK}
states that a reduced expression will never return to the flat determined by $J$ after leaving it. For example, consider the complex $\Cox_t$ from \eqref{coxA3t}, which takes the
half of $\Cox_{\mt}$ living on one side of the reflection hyperplane of $t$. The edges can be oriented precisely as a subgraph of $\Cox_{\mt}$, see \eqref{A3cox}. The important
thing to notice is that no edges point into the reflection hyperplane of $t$! Once you leave this fixed hyperplane, you can never return.

Let us justify this statement about \eqref{sameK}. This mostly involves trying to better understand the connection between the label on a facet and its stabilizer.

Let $F$ be a facet in $\hg$ adjacent to the dominant chamber. If the stabilizer of $F$ is $W_I$ (because $F$ lies on the intersection of the reflecting hyperplanes of $s \in I$)
then $F$ is labeled with the minimal $(\mt, I)$-coset in $\Cox_{\mt}$. In $\Cox_I$, $F$ would be labeled by the minimal $(I,I)$-coset instead.

If $F$ is a facet adjacent to the dominant chamber, when $w \cdot F$ is a facet adjacent to the $w$ chamber. If $I \subset S$ is the stabilizer of $F$, then the non-standard
parabolic subgroup $w W_I w\inv$ is the stabilizer of $w \cdot F$. Note that $w \cdot F$ still corresponds to an $(\mt, I)$-coset, because right multiplication by $W_I$ is
intertwined with left multiplication by $w w_I w\inv$. In fact, $w \cdot F$ is the unique $(\mt, I)$-coset containing $w$. Working backwards, suppose that a facet $F$ is labeled by
the $(\mt, I)$-coset $q$, and $w \in q$. Then the stabilizer of $F$ is $w W_I w\inv$.

An edge which goes onto the $s$-hyperplane would add $s$ to the set $w w_I w\inv$. If $s \in J$, this would increase the redundancy. In conclusion, violating \eqref{sameK} means
that one returned to an $s$-hyperplane, where $s \in J$.

\subsection{Geometric intuition for the switchback relation} \label{moregeometry2}

Let $(W,S)$ be a finite Coxeter system, acting on the left on its reflection representation $\hg$. Then $S$ generates the stabilizer of the origin in $\hg$. For any $s \in S$,
$\hat{s} := S \setminus s$ generates the stabilizer of a line in $\hg$, the intersection of all the root hyperplanes for $t \ne s$. This line splits into two rays $R$ and $R'$,
where $R$ is adjacent to the dominant chamber, and $R'$ to the antidominant chamber. The ray $R$ is labeled with the minimal $(\mt, \hat{s})$-coset $p$ in $\Cox_{\mt}$. Note that
$R'$ need not equal $w_0 \cdot R$, as the former has stabilizer $W_{\hat{s}}$ and the latter has stabilizer $W_{\hat{w_0 s w_0}}$. For any $t \in S$, we are interested in
describing reduced paths from $p$ to $q$, where $q$ is the maximal $(\mt, \hat{t})$-coset.

When $t = w_0 s w_0$ then $q$ is the label on $R'$, which is colinear with $R$. A reduced expression can never leave this line by Proposition~\ref{thm:inandout}, so one expects a
single reduced expression from $p$ to $q$, which passes through the origin. This matches up with Proposition \ref{prop:sss}\eqref{uniqrex}.

If $t \ne w_0 s w_0$ then $q$ is the label on some other ray $T$. The rays $R$ and $T$ span a two-dimensional subspace $P$ of $\hg$, and identify a ``quadrant'' of this plane. By
Proposition~\ref{thm:inandout}, a reduced path from $p$ to $q$ stays within (the closure of) this quadrant. The intersection of the hyperplane arrangement with $P$ is a collection of lines, so as one passes from $R$ to $T$ one crosses a number of these lines, which we call $T_1, T_2, \ldots, T_{\snum} = T$. The picture is as below.
\begin{equation} {
\labellist
\tiny\hair 2pt
 \pinlabel {$R$} [ ] at 14 6
 \pinlabel {$R'$} [ ] at 56 60
 \pinlabel {$T$} [ ] at 29 59
 \pinlabel {$T_{-1}$} [ ] at 58 6
 \pinlabel {$T_1$} [ ] at 0 13
 \pinlabel {$T_2$} [ ] at -5 33
 \pinlabel {$T_3$} [ ] at 0 53
\endlabellist
\centering
\ig{1}{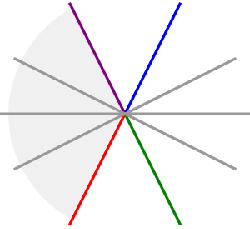}
} \end{equation}

The ray opposite $T$, which we call $T_{-1}$, has stabilizer $\hat{u_{-1}}$ where $u_{-1} = w_0 t w_0$. Visually, it appears that $T_1$ is the reflection (in the plane) of $T_{-1}$
over the ray $R$. The longest element $w_{\hat{s}}$ does not usually act on $\hg$ as a reflection, but it does act by a reflection on the plane $P$, since it is an involution whose
fixed points form a line. This reflection sends $T_{-1}$ to $T_1$, so the stabilizer of $T_1$ is $\hat{u_1}$ where $u_1 = w_{\hat{s}} u_{-1} w_{\hat{s}}$. To obtain the
stabilizer of $T_2$, we reflect $R$ across $T_1$. Similar procedures indicate that the stabilizer of $T_i$ is $\hat{u_i}$, where $(u_i)$ is the rotation sequence. That this
intuition holds true is justified by Proposition \ref{prop:sss}\eqref{eqsss}.

\bibliographystyle{alpha}
\bibliography{bib}

\newcommand{\etalchar}[1]{$^{#1}$}
\begin{thebibliography}{BKP{\etalchar{+}}18}

\bibitem[AB08]{AbBr}
Peter Abramenko and Kenneth~S. Brown.
\newblock {\em Buildings}, volume 248 of {\em Graduate Texts in Mathematics}.
\newblock Springer, New York, 2008.
\newblock Theory and applications.

\bibitem[BKP{\etalchar{+}}18]{BKPST}
Sara~C. Billey, Matja\v{z} Konvalinka, T.~Kyle Petersen, William Slofstra, and
  Bridget~E. Tenner.
\newblock Parabolic double cosets in {C}oxeter groups.
\newblock {\em Electron. J. Combin.}, 25(1):Paper No. 1.23, 66, 2018.

\bibitem[CKM14]{CKM}
Sabin Cautis, Joel Kamnitzer, and Scott Morrison.
\newblock Webs and quantum skew {H}owe duality.
\newblock {\em Math. Ann.}, 360(1-2):351--390, 2014.

\bibitem[DHP18]{DHP}
Aram Dermenjian, Christophe Hohlweg, and Vincent Pilaud.
\newblock The facial weak order and its lattice quotients.
\newblock {\em Trans. Amer. Math. Soc.}, 370(2):1469--1507, 2018.

\bibitem[EK10]{EKho}
Ben Elias and Mikhail Khovanov.
\newblock Diagrammatics for {S}oergel categories.
\newblock {\em Int. J. Math. Math. Sci.}, pages Art. ID 978635, 58, 2010.

\bibitem[EL17]{ELosev}
Ben Elias and Ivan Losev.
\newblock Modular representation theory in type {A} via {S}oergel bimodules.
\newblock preprint, 2017.
\newblock arXiv 1701.00560.

\bibitem[Eli15]{ELLCC}
Ben Elias.
\newblock Light ladders and clasp conjectures.
\newblock Preprint, 2015.
\newblock arXiv 1510.06840.

\bibitem[Eli17]{EQuantumI}
Ben Elias.
\newblock Quantum {S}atake in type {A}: part {I}.
\newblock {\em J. Comb. Algebra}, 1(1):63--125, 2017.
\newblock arXiv:1403.5570.

\bibitem[EW14]{EWHodge}
Ben Elias and Geordie Williamson.
\newblock The {H}odge theory of {S}oergel bimodules.
\newblock {\em Ann. of Math. (2)}, 180(3):1089--1136, 2014.

\bibitem[EW16]{EWGr4sb}
Ben Elias and Geordie Williamson.
\newblock Soergel calculus.
\newblock {\em Represent. Theory}, 20:295--374, 2016.
\newblock arXiv:1309.0865.

\bibitem[EWar]{EWFenn}
Ben Elias and Geordie Williamson.
\newblock Diagrammatics for {C}oxeter groups and their braid groups.
\newblock {\em Quantum Topol.}, to appear.
\newblock arXiv:0902.4700.

\bibitem[Gau01]{Gaussent}
St\'{e}phane Gaussent.
\newblock The fibre of the {B}ott-{S}amelson resolution.
\newblock {\em Indag. Math. (N.S.)}, 12(4):453--468, 2001.

\bibitem[GS84]{GarsiaStanton}
A.~M. Garsia and D.~Stanton.
\newblock Group actions of {S}tanley-{R}eisner rings and invariants of
  permutation groups.
\newblock {\em Adv. in Math.}, 51(2):107--201, 1984.

\bibitem[Lib08]{LibLL}
Nicolas Libedinsky.
\newblock Sur la cat\'egorie des bimodules de {S}oergel.
\newblock {\em J. Algebra}, 320(7):2675--2694, 2008.

\bibitem[MS89]{ManSch}
Yu.~I. Manin and V.~V. Schechtman.
\newblock Arrangements of hyperplanes, higher braid groups and higher {B}ruhat
  orders.
\newblock In {\em Algebraic number theory}, volume~17 of {\em Adv. Stud. Pure
  Math.}, pages 289--308. Academic Press, Boston, MA, 1989.

\bibitem[QS19]{QueSar}
Hoel Queffelec and Antonio Sartori.
\newblock Mixed quantum skew {H}owe duality and link invariants of type {$A$}.
\newblock {\em J. Pure Appl. Algebra}, 223(7):2733--2779, 2019.

\bibitem[Rea06]{Reading}
Nathan Reading.
\newblock Cambrian lattices.
\newblock {\em Adv. Math.}, 205(2):313--353, 2006.

\bibitem[RW16]{RicWil}
Simon Riche and Geordie Williamson.
\newblock Tilting modules and the $p$-canonical basis.
\newblock Preprint, 2016.
\newblock arXiv 1512.08296.

\bibitem[Soe90]{Soer90}
Wolfgang Soergel.
\newblock Kategorie {$\mathcal{O}$}, perverse {G}arben und {M}oduln \"uber den
  {K}oinvarianten zur {W}eylgruppe.
\newblock {\em J. Amer. Math. Soc.}, 3(2):421--445, 1990.

\bibitem[Soe07]{Soer07}
Wolfgang Soergel.
\newblock Kazhdan-{L}usztig-{P}olynome und unzerlegbare {B}imoduln \"uber
  {P}olynomringen.
\newblock {\em J. Inst. Math. Jussieu}, 6(3):501--525, 2007.

\bibitem[TVW17]{TVW}
Daniel Tubbenhauer, Pedro Vaz, and Paul Wedrich.
\newblock Super {$q$}-{H}owe duality and web categories.
\newblock {\em Algebr. Geom. Topol.}, 17(6):3703--3749, 2017.

\bibitem[Wem21]{WemyssSurvey}
Michael Wemyss.
\newblock A lockdown survey on c{DV} singularities.
\newblock Preprint, 2021.
\newblock arXiv 2103.16990.

\bibitem[Wil08]{GWthesis}
Geordie Williamson.
\newblock Singular {S}oergel bimodules, {P}h{D} thesis.
\newblock 2008.

\bibitem[Wil11]{WillSingular}
Geordie Williamson.
\newblock Singular {S}oergel bimodules.
\newblock {\em Int. Math. Res. Not. IMRN}, (20):4555--4632, 2011.

\end{thebibliography}

\end{document}